\documentclass[13pt,reqno]{amsart}
\UseRawInputEncoding

\usepackage{amsmath,amssymb,amsthm}
\usepackage{bm}
\usepackage{graphicx}
\usepackage{color}
\usepackage{float}
\usepackage{url}
\usepackage{stmaryrd}
\usepackage{mathrsfs}
\usepackage{caption}

\newcommand{\bel}[1]{\begin{equation*}\label{#1}}
	
	\newcommand{\be}{\begin{equation}}

		\newcommand{\ba}{\begin{eqnarray}}
			\newcommand{\ea}{\end{eqnarray}}

		\newcommand{\qe}{\end{equation}}
	\newcommand{\R}{{\mathbb R}}
	\newcommand{\N}{{\mathbb N}}
	\newcommand{\Z}{{\mathbb Z}}
	\newcommand{\C}{{\mathbb C}}
	\newcommand{\T}{{\mathbb T}}
	\newcommand{\supp}{{\mathrm{supp}}}

	\newcommand{\eg}{\begin{example}}
		\newcommand{\egd}{\end{example}}
	\newcommand{\tm}{\begin{thm}}
		\newcommand{\tmd}{\end{thm}}
	\newcommand{\co}{\begin{coro}}
		\newcommand{\cod}{\end{coro}}
	\newcommand{\enu}{\begin{enumerate}}
		\newcommand{\enud}{\end{enumerate}}
	\newcommand{\rmk}{\begin{rem}}
		\newcommand{\rmkd}{\end{rem}}
	
	\theoremstyle{theorem}
	\newtheorem{thm}{Theorem}[section]
	\newtheorem{prop}[thm]{Proposition}
	\theoremstyle{example}
	\newtheorem{example}[thm]{Example}
	\newtheorem{coro}[thm]{Corollary}
	\theoremstyle{lemma}
	\newtheorem{lemma}[thm]{Lemma}
	\theoremstyle{definition}
	\newtheorem{defi}[thm]{Definition}
	\theoremstyle{proof}
	
	\theoremstyle{remark}
	\newtheorem{rem}[thm]{Remark}
	\theoremstyle{remark}

	\UseRawInputEncoding

\begin{document}

		\title[Global well-posedness of cubic fractional Schr\"{o}dinger equation with rough data]{Global well-posedness of cubic fractional Schr\"{o}dinger equation with rough data}

		\author{Jiajun Wang}
		\address{Jiajun Wang: Courant Institute of Mathematical Sciences,
			New York University, New York, NY}
		\email{jw9409@nyu.edu}
		
		\begin{abstract}
		In this paper, we apply the $I$-method to establish global well-posedness for the fractional nonlinear Schr\"{o}dinger equation with initial data $u_0 \in H^{s}(\mathbb{R}^d)$ for $s < \frac{\alpha}{2}$, i.e., below the energy threshold. Moreover, for radial initial data, we combine a modified Morawetz estimate---recovered via Balakrishnan's formula---with the $I$-method to obtain improved results. In the same spirit, we employ the ``upside-down'' $I$-method to derive polynomial-in-time growth bounds for the higher-order Sobolev norm. 
		
		The main difficulty stems from the fact that Strichartz estimates for the fractional Schr\"{o}dinger equation has a loss of derivatives, and the problem is always $L^2$-supercritical, thereby requiring more delicate analysis.
		\end{abstract}
		
		\maketitle
		\vspace{-0.7cm}
		\numberwithin{equation}{section}
		\section{Introduction}
		In this paper, we consider the following defocusing cubic fractional Schr\"{o}dinger equation:
			\begin{equation}\label{FNLS}
			\begin{cases}
				i\partial_t u(t,x) = |D|^{\alpha} u + |u|^2 u, \\[4pt]
				u(0,x) = u_0(x)\in H^{s}(\R^d), \qquad (t,x)\in \R \times \R^d,
			\end{cases}
		\end{equation}
		where $|D|^{\alpha} := (\sqrt{-\Delta})^{\alpha}$, $s\ge 0$, and $\alpha\in(0,1)\cup(1,\infty)$. 
			
		The above equation arises in various physical contexts. 
		It appears naturally in fractional quantum mechanics introduced by Laskin~\cite{24}, where the classical Feynman path integral based on Brownian motion is generalized to $\alpha$-stable L\'evy processes. Moreover, it is also closely connected to other physical models, including water waves \cite{25}.
				
    	And by a standard approximation argument, one can see that the following quantities (mass and energy) are conserved by $H^\frac{\alpha}{2}$-flow of (\ref{FNLS}):
		
		\begin{itemize}
			\item 
			
			\vspace{-4pt}
			\begin{equation*}
				M(u)(t):= \int_{\R^d}|u(t,x)|^2 dx\equiv M(u)(0) \qquad \qquad(\text{Mass conservation});
			\end{equation*}

			\item 
			\begin{equation*}
				E(u)(t):=\int_{\R^d} \frac{1}{2}\big||D|^{\frac{\alpha}{2}}u(t,x)\big|^2 +\frac{1}{4}|u(t,x)|^4 dx\equiv E(u)(0) \qquad (\text{Energy conservation}).
			\end{equation*}
		\end{itemize}
		
		Recall that, the critical regularity exponent for (\ref{FNLS}), associated with the scaling symmetry, is given by $s_c:=\frac{d-\alpha}{2}.$ And the equation is called subcritical, critical, or supercritical according as $s>s_c$, $s=s_c$, or $s<s_c$, respectively.

	Our goal is to decrease the regularity requirement in establishing global well-posedness of (\ref{FNLS}). The major obstacle is the failure of energy conservation when regularity is below the energy threshold, i.e., $s<\frac{\alpha}{2}$.
	
	Since we focus on the subcritical case, we require $\alpha>\frac{d}{2}$ so that the range
	\[
	\frac{d-\alpha}{2}<s<\frac{\alpha}{2}
	\]
	is non-empty. Moreover, as the regime $\alpha\in(1,2)$ is of particular interest in the study of fractional Schr\"odinger equations, we restrict ourselves to
	\[
	(d=2,\;1<\alpha<2)
	\quad\text{and}\quad
	(d=3,\;\tfrac32<\alpha<2).
	\] 
		
		To overcome the difficulty from the lack of energy conservation, J. Colliander, M. Keel, G. Staffilani, H. Takaoka and T. Tao first developed $I$-method, which reduces the global well-posedness to the ``almost conservation" of modified solution $Iu$ (see e.g. \cite{26,27,21,28}). Applying this method to traditional Schr\"{o}dinger equation ($\alpha=2$), they derived the 2D global well-posedness when $s>\frac{4}{7}$ in \cite{18}. After that, they also introduced a correction term to the modified energy, which damps out oscillations and yields an improved result $s>\frac{1}{2}$ (see \cite{29}). It should be mentioned that Fang and Grillakis proved $s\ge \frac{1}{2}$ by using a totally different method based on a new Morawetz estimate \cite{30}. Then \cite{19} improved the Fang-Grillakis interaction Morawetz estimate and combined it with $I$-method to give a better result $s>\frac{2}{5}$. The general scheme of our paper is inspired by \cite{18,36,19}.
		
		Although the global well-posedness theory of traditional Schr\"{o}dinger equation ($\alpha=2$) has been extensively studied, similar results for the fractional one are mainly focused on the 1D case, local well-posedness or global well-posedness but in the $H^{\frac{\alpha}{2}}$-regime.
		
		For instance, in the periodic case $\T$, \cite{31} proved 1D local well-posedness for $s>\frac{2-\alpha}{4}$ and global well-posedness for $s>\frac{5\alpha+1}{12}$. \cite{32} gave the local theory at the endpoint $s=\frac{2-\alpha}{4}$. Recently, \cite{15} utilized $I$-method and established global well-posedness for $s\ge \frac{2-\alpha}{4}$ and ill-posedness for $s<\frac{2-\alpha}{4}$. In the Euclidean case $\R^d$, $d\ge 2$, \cite{33,4} showed the local well-posedness for subcritical range $s>\frac{d-\alpha}{2}$. As a relatively direct consequence, they can establish the global one at energy threshold $s=\frac{\alpha}{2}$ by energy conservation.
		
		The difficulty in higher dimensional fractional Schr\"{o}dinger equation partially lies in the suitable characterization of its resonant set:
		\begin{equation*}
			\Gamma_{res}:=\left\{(\xi_1, \xi_2, \xi_3,\xi_4)\in \Sigma_4: -|\xi_1|^\alpha+|\xi_2|^{\alpha}-|\xi_3|^\alpha+|\xi_4|^{\alpha}=0\right\},
		\end{equation*}
		\begin{equation*}
			\Sigma_4:=\left\{(\xi_1, \xi_2, \xi_3,\xi_4): \xi_i\in \R^d, \; \xi_1+\xi_2+\xi_3+\xi_4=0\right\}.
		\end{equation*}
		If we change the variables
		\begin{equation*}
			\xi_1:=-(y_1+y_2), \;\; \xi_2:=y_1+y_2+y_3,\;\; \xi_3:=-(y_1+y_3), \;\; \xi_4:=y_1,
		\end{equation*}
		then, from calculus, one can rewrite 
		\begin{equation*}
			-|\xi_1|^\alpha+|\xi_2|^{\alpha}-|\xi_3|^\alpha+|\xi_4|^{\alpha}=-|y_1+y_2|^{\alpha}+|y_1+y_2+y_3|^\alpha-|y_1+y_3|^{\alpha}+|y_1|^\alpha
		\end{equation*}
		\begin{equation}\label{m}
			=\alpha\int_{0}^{1}\!\int_{0}^{1} |z|^{\alpha-4}\left(|z|^{2}(y_2\cdot y_3)+(\alpha-2)(z\cdot y_2)(z\cdot y_3)\right)d\tau ds, 
		\end{equation}
		where $z=y_1+\tau y_2+s y_3.$
		
		If $\alpha=2$, we see $(\ref{m})=2(y_2\cdot y_3)$, which implies
		\begin{equation*}
			\Gamma_{res}:=\left\{(\xi_1, \xi_2, \xi_3,\xi_4)\in \Sigma_4: (\xi_1+\xi_2)\perp(\xi_1+\xi_4)\right \}.
		\end{equation*}
		This intrinsic orthogonality within the resonant set in fact underlies the introduction of the correction term in \cite{29}.
		
		If $\alpha\in (1,2)$ and $d=1$, we can also see $(\ref{m})=c(\alpha, y_1, y_2) y_1y_2$, where the factor $c(\alpha, y_1, y_2) > 0$. Similarly, the resonant set can be characterized as 
		\begin{equation*}
			\Gamma_{res}:=\left\{(\xi_1, \xi_2, \xi_3,\xi_4)\in \Sigma_4: (\xi_1+\xi_2)(\xi_1+\xi_4)=0\right \},
		\end{equation*}
		which is the key observation in \cite{15}.
		
	However, when $\alpha\in(1,2)$ and $d\ge2$, such a neat characterization no longer holds. Indeed, in the case $d=2$, one may choose
	\begin{equation*}
		y_1:=0,\quad y_2:=R\theta_1,\quad y_3:=R\theta_2,\quad \theta_1\cdot\theta_2=2^{\frac{2}{\alpha}-1}-1\in(0,1),
	\end{equation*}
	where $\theta_1,\theta_2\in\mathbb{S}^1$. Then
	\begin{equation*}
		-|\xi_1|^\alpha+|\xi_2|^\alpha-|\xi_3|^\alpha+|\xi_4|^\alpha=0,
	\end{equation*}
	even though no orthogonality condition is present.
	
	In order to obtain improved global well-posedness results, our strategy is to avoid relying on a characterization of the resonant set $\Gamma_{res}$. Instead, we combine the $I$-method with a Morawetz estimate, as in \cite{19}. However, the non-locality of the fractional derivative $|D|^{\alpha}$ also poses a difficulty in deriving a Morawetz estimate, especially the interaction one. 
	
	For example, in deriving the interaction Morawetz inequality, we require that the tensor product $u_1\otimes u_2$ of two solutions $u_1$, $u_2$ to (\ref{FNLS}) still satisfies a Schr\"{o}dinger-type equation, which relies on the locality:
	\begin{equation*}
		\Delta_x+\Delta_y=\Delta_{x,y}, \quad (x,y)\in \R^d\times\R^d.
	\end{equation*}
	We refer to \cite{19} for more details.
	
	To overcome the non-locality, we apply Balakrishnan's formula, which expresses the fractional operator $|D|^{\alpha}$ in terms of the Laplacian $\Delta$. This allows us to partially recover a local structure. For applications of this formula in nonlinear dispersive equations, we refer to \cite{7,10,34}.
		
		\vspace{7pt}
		Next, we briefly recall the setup of $I$-method: given $s<\frac{\alpha}{2}$ and a parameter $N\gg 1$, we define the $I$-operator $I_N$ as follows.
		\begin{equation*}
			\widehat{I_N f}(\xi):=m_N(\xi) \widehat{f}(\xi),
		\end{equation*}
		where the multiplier $m_N(\xi)$ is smooth, radially symmetric, and 
		\begin{equation*}
			m_N(\xi):=
			\begin{cases}
				\;\;1\quad\;\;\;\;\;\;\;,  \quad |\xi|\le N\\[4pt]
				\left(\frac{N}{|\xi|}\right)^{\frac{\alpha}{2}-s}, \quad |\xi|\ge 2N.
			\end{cases}
		\end{equation*}
		
		We also have the following important relations between $\|\phi\|_{H^s(\R^d)}$ and $\|I_N \phi\|_{H^{\frac{\alpha}{2}}(\R^d)}$:
		\begin{equation*}
			\|I_N \phi\|_{H^{\frac{\alpha}{2}}(\R^d)}\lesssim N^{\frac{\alpha}{2}-s}\|\phi\|_{H^s(\R^d)},
		\end{equation*}
		\begin{equation}\label{second}
			\|\phi\|_{H^s(\R^d)}\lesssim \|I_N \phi\|_{H^{\frac{\alpha}{2}}(\R^d)}.
		\end{equation}
		For convenience, we will drop the subscript $N$ from the notation and write $m(|\xi|):=m(\xi)$.
		
		By establishing a bilinear estimate for the fractional Schr\"{o}dinger equation, we can use the $I$-method to prove global well-posedness from rough initial data $u_0\in H^{s}(\R^d)$, $s<\frac{\alpha}{2}$.
		
		\begin{thm}\label{main d=2}
			Let $d=2$, $\alpha\in (1,2)$, and $s$ satisfying
			\begin{equation*}
				\frac{\alpha}{2}>s>\frac{\alpha}{2}-\frac{(4-\alpha)(\alpha-1)^2}{-\alpha^2+7\alpha-4}:=r_2(\alpha).
			\end{equation*}
			Then the fractional Schr\"{o}dinger equation (\ref{FNLS}) is globally well-posed in $H^{s}(\R^2)$, i.e., there exists a unique solution $u\in C([0,+\infty); H^s(\R^2))$. Moreover, the $H^s(\R^2)$-norm has 
			at most polynomial-in-time growth:
			\begin{equation*}
				\|u(T)\|_{H^s(\R^2)}\lesssim_{\|u_0\|_{H^s}} (1+T)^{\frac{2(\alpha-1)(\frac{\alpha}{2}-s)^+}{(-\alpha^2+7\alpha-4)(s-r_2(\alpha))}}, \quad \forall T>0.
			\end{equation*}
			Here $a^{+}$ (resp. $a^{-}$) denotes $a+\varepsilon$ (resp. $a-\varepsilon$) for an arbitrarily small $\varepsilon>0$.
		\end{thm}
		
		For the 3D case, we also have the following result:
		\begin{thm}\label{main d=3}
			Let $d=3$, $\alpha\in \left(\frac{3}{2},2\right)$, and $s$ satisfying
			\begin{equation*}
				\frac{\alpha}{2}>s>\frac{\alpha}{2}-\frac{(2\alpha-3)^2}{6(\alpha-1)}=:r_3(\alpha).
			\end{equation*}
			Then the fractional Schr\"{o}dinger equation (\ref{FNLS}) is globally well-posed in $H^{s}(\R^3)$, i.e., there exists a unique solution $u\in C([0,+\infty); H^s(\R^3))$. Moreover, the $H^s(\R^3)$-norm has 
			at most polynomial-in-time growth:
			\begin{equation*}
				\|u(T)\|_{H^s(\R^3)}\lesssim_{\|u_0\|_{H^s}} (1+T)^{\frac{(2\alpha-3)(\frac{\alpha}{2}-s)^+}{6(\alpha-1)(s-r_{3}(\alpha))}}, \quad \forall T>0.
			\end{equation*}
		\end{thm}
		
	\vspace{7pt}
	If we further consider the radial case, i.e., the solution $u(t,x)$ is radially symmetric in $x$-variable, some improved results can be obtained by combining with the Morawetz estimate and a new framework.
	
	\begin{thm}\label{main radial}
		Let $d=3$, $\alpha\in \left(\frac{3}{2}, \frac{8-\sqrt{10}}{3} \right)$, and $s$ satisfying
		\begin{equation*}
			\frac{\alpha}{2}>s>\frac{\alpha}{2}-\frac{\alpha(2\alpha-3)^2}{2(\alpha^2+\alpha-3)}=:r_3^{rad}(\alpha).
		\end{equation*}
		If the initial data $u_0$ is radially symmetric, then the fractional Schr\"{o}dinger equation (\ref{FNLS}) has a unique solution $u\in C([0,+\infty); H^s(\R^3))$. Moreover, we have a uniform estimate 
		\begin{equation*}
			\sup_{t\in [0,+\infty)}\|u(t)\|_{H^s(\R^3)}\lesssim_{\|u_0\|_{H^s}} 1.
		\end{equation*}
	\end{thm}
	For convenience, we will denote
	\begin{equation*}
		H_{rad}^s(\R^d):=\left\{\phi\in H^s(\R^d): \phi \; \text{is radially symmetric}\right\}.
	\end{equation*} 
	
	\begin{figure}
		\centering
		\includegraphics[width=0.7\linewidth]{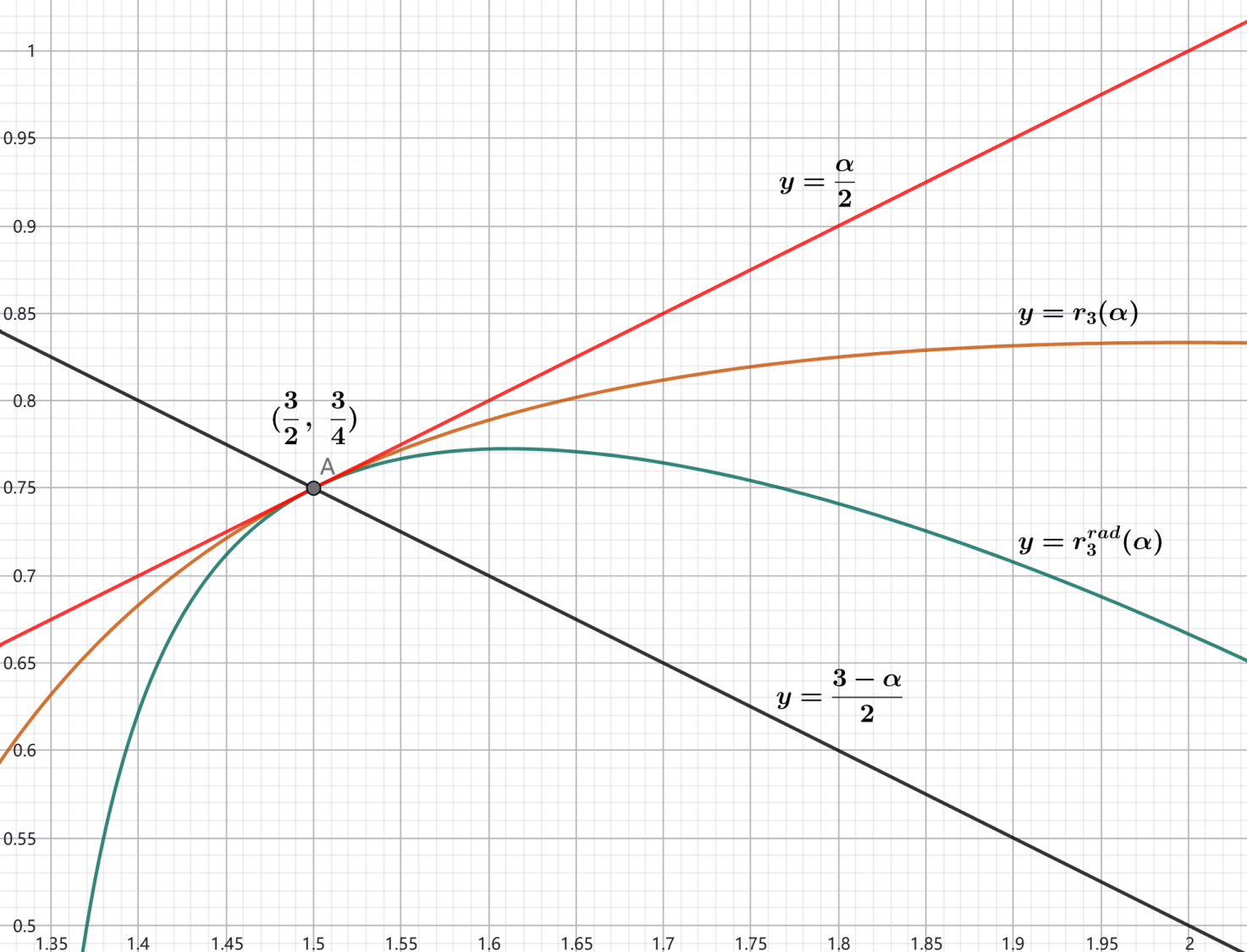}
		\caption{Comparison}
		\label{fig:comp}
	\end{figure}
	
	\begin{rem}
		To visualize the improvement in the lower bound for $s$, we refer to Figure~1. Another significant improvement lies in the improved control of the growth of the Sobolev norm. In the radial case, Theorem~\ref{main radial} ensures that $\|u(t)\|_{H^{s}(\mathbb{R}^3)}$ remains bounded uniformly in time. In contrast, Theorem~\ref{main d=3} only yields a polynomial-in-time growth bound.
	\end{rem}
	
	\vspace{7pt}
	Following the proof of Theorem \ref{main radial} (with some modifications), we can further derive the scattering results. Before stating them formally, we first recall relevant definitions \cite{11}, which are adapted to the radially symmetric regime. 
	
	\begin{defi}
		We say a global solution $u\in C(\R; H_{rad}^s(\R^d))$ to the fractional nonlinear Schr\"{o}dinger equation (\ref{FNLS}) with initial data $u_0\in H_{rad}^s(\R^d)$ scatters to a free solution $e^{-it|D|^\alpha}u_{\pm}$ as $t\to \pm \infty$, if we have
		\begin{equation}\label{scattering}
			\left\|u(t)-e^{-it|D|^\alpha}u_{\pm}\right\|_{H^s(\R^d)}\to 0, \quad \text{as}\;\; t\to \pm \infty.
		\end{equation}
		Suppose that for every asymptotic state $u_{\pm}\in H_{rad}^s(\R^d)$, there exists a unique global solution $u$ to (\ref{FNLS}) such that (\ref{scattering}) holds, we can define the wave operator
		\[
		\begin{array}{rcl}
			\Omega_{\pm}: H_{rad}^s(\R^d)& \longrightarrow & H_{rad}^s(\R^d) \\
			u_{\pm} &\mapsto & u_0
		\end{array}
		\]
	\end{defi}
	Note that the uniqueness in well-posedness theory ensures the injectivity of wave operator $\Omega_{\pm}$. If it is also surjective, we say that the equation has asymptotic completeness.
	
	\vspace{6pt}
	\begin{thm}\label{main scattering}
		With the same assumptions in Theorem \ref{main radial}, the wave operator $\Omega_{\pm}$ exists and is also surjective.
	\end{thm}
	
	\vspace{7pt}
	This paper is organized as follows. In Section~2, we will establish several useful bilinear estimates, which play a central role in the subsequent sections. In Section 3, the local well-posedness theory of the modified solution will be built, which is the basis of further analysis. Section~4 and Section 5 are devoted to the proofs of Theorem \ref{main d=3} and Theorem \ref{main d=2}, respectively. In Section~6, we derive the Morawetz estimate for the radial case. Section 7 will then give a detailed proof of our improved results---Theorem \ref{main radial} and Theorem \ref{main scattering}. In Section 8, we will discuss the case $\alpha>2$ and present corresponding results. Finally, the polynomial-in-time growth of higher-order Sobolev norms will be derived in Section 9.
		
		\vspace{10pt}
		\noindent
		\textbf{Notation.}
		\begin{itemize}
			\item By $u\in C^{k}([0,T]; B)( \mbox{or} \; L^{p}([0,T];B))$ for a Banach space $B,$ we mean $u$ is a $C^{k}(\mbox{or} \; L^{p})$ map from $[0,T]$ to $B;$ see page 301 in \cite{20}.
			\item By $A\lesssim B$ (resp. $A\sim B$), we mean there is a positive constant $C$, such that $A\le CB$ (resp. $C^{-1}B\le A \le C B$). If the constant $C$ depends on $p,$ then we write $A\lesssim_{p}B$ (resp. $A\sim_{p} B$).
			\item By $A\ll B$, we mean $\frac{A}{B}$ is sufficiently smaller than $1$.
			\item By $\langle t\rangle$ for $t\in\R$, we mean $(1+|t|^2)^{\frac{1}{2}}$.
			\item By $\mathcal{S}(\R^d)$, we mean the Schwarz space on $\R^d$.
			\item By $P_N$ for a dyadic integer $N\in 2^{\Z}$, we mean a projection to the frequency $|\xi|\sim N$.
		\end{itemize}
		
	\section{Strichartz estimate and bilinear estimate}
	We first introduce Strichartz estimate for the fractional Schr\"{o}dinger equation, which has a loss of derivative, due to the lack of uniform convexity (see \cite{3,4}). 
	
	\begin{lemma}[Strichartz estimate with derivative loss]\label{Strichartz1}
		\
		
		A pair $(q,r)$ is called 2-admissible, if 
		\begin{equation*}
			\frac{2}{q}+\frac{d}{r}\le \frac{d}{2}, \quad 2\le q,r\le \infty, \quad (q,r,d)\ne (2,\infty, 2).
		\end{equation*}
		For $(q,r)\in [1,\infty]^2$, we also denote 
		\begin{equation*}
			\gamma_{q,r}:=\frac{d}{2}-\frac{\alpha}{q}-\frac{d}{r}.
		\end{equation*}
		Then we have the following Strichartz estimates for any spacetime slab $[0,T]\times\R^d:$
		\begin{equation}\label{dd}
			\left\|e^{-it|D|^{\alpha}}P_N \varphi\right\|_{L_t^{q}L_x^{r}([0,T]\times\R^d)}\lesssim N^{\gamma_{q,r}}\|P_N\varphi\|_{L^{2}(\R^d)},
		\end{equation}
		\begin{equation*}
			\left\|\int_{0}^{t}e^{-i(t-s)|D|^{\alpha}}P_N F(s)ds\right\|_{L_t^{q}L_x^{r}([0,T]\times\R^d)}\lesssim N^{\gamma_{q,r}-\gamma_{\widetilde{q}',\widetilde{r}'}-\alpha}\|P_N F\|_{L_t^{\widetilde{q}'}L_x^{\widetilde{r}'}([0,T]\times\R^d)},
		\end{equation*}
		for all $N\in 2^{\Z}$, and $(q,r), (\widetilde{q},\widetilde{r})$ 2-admissible.
	\end{lemma}
	We first establish the following bilinear estimate, which is an analogue of Bourgain's refined Strichartz estimate in \cite{2} or Lemma 3.4 in \cite{1}.
	\begin{lemma}\label{bilinear}
		Let $d\ge 2$. For any spacetime slab $I\times \R^d$ and $t_0\in I$, we have
		\begin{equation*}
			\|P_{N_1}u_1 P_{N_2}u_2\|_{L_{t,x}^{2}(I\times \R^d)}\lesssim \frac{\min\lbrace N_1, N_2\rbrace^{\frac{d-1}{2}}}{\max\lbrace N_1, N_2\rbrace^{\frac{\alpha-1}{2}}}\left(\|u_1(t_0)\|_{L^{2}(\R^d)}+\|(i\partial_t-|D|^{\alpha})u_1\|_{L_t^{1}L_x^2(I\times\R^d)}\right)
		\end{equation*}
		\begin{equation}\label{inhomo}
			\times \left(\|u_2(t_0)\|_{L^{2}(\R^d)}+\|(i\partial_t-|D|^{\alpha})u_2\|_{L_t^{1}L_x^2(I\times\R^d)}\right).
		\end{equation}
	\end{lemma}
		\begin{proof}
			Without loss of generality, we assume $P_{N_1}u_1=u_1, P_{N_2}u_2=u_2$, $N_1\le N_2$ and $t_0=0\in I$.
			
			We first deal with the homogeneous case, i.e., 
			\begin{equation*}
				u_1=e^{-it|D|^{\alpha}}\psi_1, \quad u_2=e^{-it|D|^{\alpha}}\psi_2.
			\end{equation*}

			If $N_1\sim N_2$, we can directly apply H\"{o}lder's inequality and Strichartz estimate and obtain 
			\begin{equation*}
				\|P_{N_1}u_1 P_{N_2}u_2\|_{L_{t,x}^{2}(I\times \R^d)}\le \|P_{N_1}u_1\|_{L_{t,x}^{4}(I\times \R^d)}\|P_{N_2}u_2\|_{L_{t,x}^{4}(I\times \R^d)}
			\end{equation*}
			\begin{equation*}
				\lesssim (N_1 N_2)^{\gamma_{4,4}}\|u_1(0)\|_{L^{2}(\R^d)}\|u_2(0)\|_{L^{2}(\R^d)}\sim \frac{N_1^{\frac{d-1}{2}}}{N_2^{\frac{\alpha-1}{2}}}\|u_1(0)\|_{L^{2}(\R^d)}\|u_2(0)\|_{L^{2}(\R^d)}.
			\end{equation*}
			Thus, we can further assume $N_1\ll N_2$. Writing as follows,
			\begin{equation*}
				(e^{-it|D|^{\alpha}}\psi_1)(e^{-it|D|^{\alpha}}\psi_2)=\int\!\!\!\int_{\R^d\times\R^d} \widehat{\psi_1}(\xi_1)\widehat{\psi_2}(\xi_2)e^{i[(\xi_1+\xi_2)\cdot x-(|\xi_1|^{\alpha}+|\xi_2|^{\alpha})t]}d\xi_1d\xi_2.
			\end{equation*}
			Then from Plancherel, we can deduce that
			\begin{equation*}
				\|P_{N_1}u_1 P_{N_2}u_2\|_{L_{t,x}^{2}(I\times \R^d)}^2=\|(e^{-it|D|^{\alpha}}\psi_1)(e^{-it|D|^{\alpha}}\psi_2)\|_{L_{t,x}^{2}(I\times \R^d)}^2
			\end{equation*}
			\begin{equation*}
				=\int_{\R^d}\int_{\R} d\xi d\tau \Big|\int_{\R^d}\widehat{\psi_1}(\xi_1)\widehat{\psi_2}(\xi-\xi_1)\delta_0(|\xi_1|^{\alpha}+|\xi-\xi_1|^{\alpha}-\tau) d\xi_1\Big|^2
			\end{equation*}
			We denote $F_{\tau,\xi}(\xi_1):=|\xi_1|^{\alpha}+|\xi-\xi_1|^{\alpha}-\tau$, and the corresponding hyper-surface 
			\begin{equation*}
				S_{\tau,\xi}:=\lbrace\xi_1:F_{\tau,\xi}(\xi_1)=0, |\xi_1|\sim N_1, |\xi-\xi_1|\sim N_2\rbrace.
			\end{equation*}
			Now, we apply Cauchy-Schwarz, the coarea formula, and obtain
			\begin{equation*}
				\|P_{N_1}u_1 P_{N_2}u_2\|_{L_{t,x}^{2}(I\times \R^d)}^2\le \int_{\R^d}\int_{\R} d\xi d\tau \Big|\int_{S_{\tau,\xi}}|\widehat{\psi_1}(\xi_1)||\widehat{\psi_2}(\xi-\xi_1)|\frac{1}{|\nabla_{\xi_1} F_{\tau, \xi}|}dS_{\tau,\xi}\Big|^2
			\end{equation*}
			\begin{equation*}
				\le \int_{\R^d}\int_{\R} d\xi d\tau \int_{S_{\tau,\xi}}|\widehat{\psi_1}(\xi_1)|^2|\widehat{\psi_2}(\xi-\xi_1)|^2\frac{1}{|\nabla_{\xi_1} F_{\tau, \xi}|}dS_{\tau,\xi}
			\end{equation*}
			\begin{equation*}
				\times N_2^{-(\alpha-1)}
				\left[\sup_{\tau\in \R, |\xi|\sim N_2} \operatorname{mes}_{(d-1)}\left\lbrace \xi_1: |\xi_1|\sim N_1 \;\text{and } |\xi_1|^{\alpha}+|\xi-\xi_1|^{\alpha}=\tau\right\rbrace\right]
			\end{equation*}
			\begin{equation*}
				\lesssim \frac{N_1^{d-1}}{N_2^{\alpha-1}} \|\psi_1\|_{L^2(\R^d)}^2 \|\psi_2\|_{L^2(\R^d)}^2.
			\end{equation*}
			Here we used the simple fact:
			\begin{equation*}
				|\nabla_{\xi_1}F_{\tau, \xi}|\gtrsim |\xi-\xi_1|^{\alpha-1}\gtrsim N_2^{\alpha-1}.
			\end{equation*}
			For the inhomogeneous case, we set $F:=(i\partial_t-|D|^{\alpha})u_1$ and $G:=(i\partial_t-|D|^{\alpha})u_2$. Applying the following Duhamel formula:
			\[
			u_1 = e^{-it|D|^{\alpha}} u_1(0) - i \int_{0}^t e^{-i(t-t')|D|^{\alpha}} F(t') \, dt', \quad u_2 = e^{-it|D|^{\alpha}} u_2(0) - i \int_{0}^t e^{-i(t-t')|D|^{\alpha}} G(t')dt'.
			\]
			
			We obtain
			\[
			\|u_1 u_2\|_{L_{t,x}^2(I\times \mathbb{R}^d)} \lesssim \left\| e^{-it|D|^{\alpha}} u_1(0) e^{-it|D|^{\alpha}} u_2(0) \right\|_{L_{t,x}^2(I\times \mathbb{R}^d)}
			\]
			\[
			+ \left\| e^{-it|D|^{\alpha}} u_1(0) \int_{0}^t e^{-i(t-t')|D|^{\alpha}} G(t') \, dt' \right\|_{L_{t,x}^2(I\times \mathbb{R}^d)}
			\]
			\[
			+ \left\| e^{-it|D|^{\alpha}} u_2(0) \int_{0}^t e^{-i(t-t')|D|^{\alpha}} F(t') \, dt' \right\|_{L_{t,x}^2(I\times \mathbb{R}^d)}
			\]
			\[
			+ \left\| \int_{0}^t e^{-i(t-t')|D|^{\alpha}} F(t') \, dt' \int_{0}^t e^{-i(t-t'')|D|^{\alpha}} G(x, t'') \, dt'' \right\|_{L_{t,x}^2(I\times \mathbb{R}^d)}
			\]
			\[
			:= I_1 + I_2 + I_3 + I_4.
			\]
			
			The first term is just the homogeneous case. By symmetry, we consider only $I_2$ and $I_4$. For $I_2$, we can apply Minkowski's inequality, 
			\begin{equation*}
				I_2\lesssim \int_{\R}\|e^{-it|D|^{\alpha}}u_1(0)e^{-it|D|^{\alpha}}(e^{it'|D|^{\alpha}}G(t'))\|_{L_{t,x}^2(I\times \mathbb{R}^d)}dt'\lesssim \frac{N_1^{\frac{d-1}{2}}}{N_2^{\frac{\alpha-1}{2}}} \|u_1(0)\|_{L^{2}(\R^d)}\|G\|_{L_t^{1}L_x^{2}(I\times\R^d)}.
			\end{equation*}
			For $I_4$, we similarly have 
			\begin{equation*}
				I_4\lesssim \int_{\R}\int_{\R} \|e^{-it|D|^{\alpha}}(e^{it'|D|^{\alpha}}F(t'))e^{-it|D|^{\alpha}}(e^{it''|D|^{\alpha}}G(t''))\|_{L_{t,x}^2(I\times \mathbb{R}^d)}dt'dt''
			\end{equation*}
			\begin{equation*}
				\lesssim \frac{N_1^{\frac{d-1}{2}}}{N_2^{\frac{\alpha-1}{2}}}\|F\|_{L_t^{1}L_x^{2}(I\times\R^d)}\|G\|_{L_t^{1}L_x^{2}(I\times\R^d)}.
			\end{equation*}
		\end{proof}
		\begin{rem}
			In fact, the bilinear estimate (\ref{bilinear}) is always superior to the Strichartz estimate (\ref{Strichartz1}) when dealing with terms like
			\begin{equation*}
				\|P_{N_1}u_1 P_{N_2}u_2\|_{L_{t,x}^2(I\times\R^d)}.
			\end{equation*}
			For example, in the homogeneous case, by applying the Strichartz estimate and H\"{o}lder's inequality, we can derive
			\begin{equation*}
				\|P_{N_1}e^{-it|D|^{\alpha}}\psi_1 P_{N_2}e^{-it|D|^{\alpha}}\psi_2\|_{L_{t,x}^2(I\times\R^d)} \lesssim \|P_{N_1}e^{-it|D|^{\alpha}}\psi_1\|_{L_t^{q_1}L_x^{r_1}(I\times\R^d)}\|P_{N_2}e^{-it|D|^{\alpha}}\psi_2\|_{L_t^{q_2}L_x^{r_2}(I\times\R^d)}
			\end{equation*}
			\begin{equation*}
				\lesssim N_1^{\gamma_{q_1,r_1}}N_2^{\gamma_{q_2,r_2}}\|\psi_1\|_{L^2(\R^d)}\|\psi_2\|_{L^2(\R^d)}.
			\end{equation*}
			By optimizing all 2-admissible pairs $(q_i, r_i)$, $i=1,2$, satisfying
			\begin{equation*}
				\frac{1}{q_1}+\frac{1}{q_2}=\frac{1}{2}, \quad  \frac{1}{r_1}+\frac{1}{r_2}=\frac{1}{2},
			\end{equation*}
			we can only obtain 
			\begin{equation*}
				\|P_{N_1}e^{-it|D|^{\alpha}}\psi_1 P_{N_2}e^{-it|D|^{\alpha}}\psi_2\|_{L_{t,x}^2(I\times\R^d)}\lesssim \min \lbrace N_1, N_2\rbrace^{\frac{d-\alpha}{2}}\|\psi_1\|_{L^2(\R^d)}\|\psi_2\|_{L^2(\R^d)},
			\end{equation*}
			which is strictly weaker than our bilinear estimate.
			
			However, it should be noted that Strichartz estimates can be improved in certain cases when one is allowed to apply H\"older's inequality in time, i.e., when $|I|\lesssim 1$. This simple observation plays an important role in the two-dimensional case.
		\end{rem}
		
		To reformulate our bilinear estimate in the Bourgain space setting, we first recall the definition of the $X^{s,b}$ spaces \cite{38}.
			\begin{defi}
			The space $X^{s,b}(\R\times\R^d)$ is the closure of Schwarz functions $\mathcal{S}_{t,x}(\R\times \R^d)$ under the following norm:
			\begin{equation*}
				\|u\|_{X^{s,b}(\R\times\R^d)}:=\left(\frac{1}{(2\pi)^{d}}\int_{\R^d}\int_{\R}(1+|\xi|^{2})^{s}(1+|\tau+|\xi|^{\alpha}|^{2})^{b}|\mathcal{F}_{t,x}u(\tau,\xi)|^{2} \, d\tau d\xi\right)^{1/2}.
			\end{equation*}
		\end{defi}
		\begin{rem}
			If we denote $H^{s,b}(\R\times\R^d):=H^{b}(\R;H^{s}(\R^d))$, then we have another representation of the Bourgain space $X^{s,b}(\R\times\R^d)$: 
			\begin{equation*}
				\|u\|_{X^{s,b}(\R\times\R^d)}=\|e^{it|D|^{\alpha}}u\|_{H^{s,b}(\R\times\R^d)}.
			\end{equation*}
		\end{rem}
		The Bourgain space $X^{s,b}(\R\times\T^d)$ defined above is global in time, which requires the following modification, especially when establishing local well-posedness theory.
		\begin{defi}
			For $T>0$, the local Bourgain space $X^{s,b}([-T,T]\times\R^d)$ consists of all functions $u$ such that there exists another function $\widetilde{u}\in X^{s,b}(\R\times\R^d)$ whose restriction to $[-T,T]$ is $u$, i.e., $\widetilde{u}|_{[-T,T]}\equiv u$. The norm is given by 
			\begin{equation*}
				\|u\|_{X^{s,b}([-T,T]\times\R^d)}:=\inf\left\lbrace  \|\widetilde{u}\|_{X^{s,b}(\R\times\R^d)}\; \Big| \; \widetilde{u}\in X^{s,b}(\R\times\R^d),\; \widetilde{u}|_{[-T,T]}\equiv u\right\rbrace.
			\end{equation*}
		\end{defi}
		
		\vspace{7pt}
		As a direct corollary of Lemma \ref{bilinear}, we have the following bilinear estimate in Bourgain space.
		\begin{coro}\label{bilinear2}
			If $d\ge 2$, then for any spacetime slab $I\times \R^d$, we have 
			\begin{equation}\label{e}
				\|P_{N_1}u_1 P_{N_2}u_2\|_{L_{x,t}^{2}(I\times\R^d)}\lesssim \frac{\min\lbrace N_1, N_2\rbrace^{\frac{d-1}{2}}}{\max\lbrace N_1, N_2\rbrace^{\frac{\alpha-1}{2}}}\|u_1\|_{X^{0,\frac{1}{2}^{+}}(I\times \R^d)}\|u_2\|_{X^{0,\frac{1}{2}^{+}}(I\times \R^d)}
			\end{equation}
		\end{coro}
		\begin{proof}
			For any $\widetilde{u_1}|_{I}\equiv u_1, \; \widetilde{u_2}|_{I}\equiv u_2$, we let $v_1=e^{it|D|^{\alpha}}\widetilde{u_1},\;  v_2=e^{it|D|^{\alpha}}\widetilde{u_2}$; then we can obtain 
			\begin{equation*}
				\|P_{N_1}u_1 P_{N_2}u_2\|_{L_{x,t}^{2}(I\times\R^d)}=\|(e^{-it|D|^{\alpha}}P_{N_1}v_1)(e^{-it|D|^{\alpha}}P_{N_2}v_2)\|_{L_{x,t}^{2}(I\times\R^d)}
			\end{equation*}
			\begin{equation*}
				=\left\|\left(\int_{\R}e^{-it|D|^{\alpha}}\mathcal{F}_{t}(P_{N_1}v_1)(\tau_1)e^{it\tau_1}d\tau_1\right)\left(\int_{\R}e^{-it|D|^{\alpha}}\mathcal{F}_{t}(P_{N_2}v_2)(\tau_2)e^{it\tau_2}d\tau_2\right)\right\|_{L_{x,t}^{2}(I\times\R^d)}
			\end{equation*}
			Applying Minkowski's inequality, Lemma \ref{bilinear} and Cauchy-Schwarz, we obtain
			\begin{equation*}
				\|P_{N_1}u_1 P_{N_2}u_2\|_{L_{x,t}^{2}(I\times\R^d)}\le\int\!\!\!\int_{\R\times\R}\left\|\left(e^{-it|D|^{\alpha}}\mathcal{F}_t(P_{N_1}v_1)\right)\left(e^{-it|D|^{\alpha}}\mathcal{F}_t(P_{N_2}v_2)\right)\right\|_{L_{x,t}^{2}(I\times\R^d)} d\tau_1 d\tau_2.
			\end{equation*}
			\begin{equation*}
				\lesssim \frac{\min\lbrace N_1, N_2\rbrace^{\frac{d-1}{2}}}{\max\lbrace N_1, N_2\rbrace^{\frac{\alpha-1}{2}}}\left(\int_{\R}\|\mathcal{F}_t(P_{N_1}v_1)\|_{L^{2}(\R^d)}d\tau_1\right)\left(\int_{\R}\|\mathcal{F}_t(P_{N_2}v_2)\|_{L^{2}(\R^d)}d\tau_2\right)
			\end{equation*}
			\begin{equation*}
				\lesssim \frac{\min\lbrace N_1, N_2\rbrace^{\frac{d-1}{2}}}{\max\lbrace N_1, N_2\rbrace^{\frac{\alpha-1}{2}}}\left(\int_{\R}\|\mathcal{F}_t(P_{N_1}v_1)\|_{L^{2}(\R^d)}^{2} (1+|\tau_1|^{2})^{\frac{1}{2}^{+}}d\tau_1\right)
			\end{equation*}
			\begin{equation*}
				\times\left(\int_{\R}\|\mathcal{F}_t(P_{N_2}v_2)\|_{L^{2}(\R^d)}^{2} (1+|\tau_2|^{2})^{\frac{1}{2}^{+}}d\tau_2\right)
			\end{equation*}
			\begin{equation*}
				\le\frac{\min\lbrace N_1, N_2\rbrace^{\frac{d-1}{2}}}{\max\lbrace N_1, N_2\rbrace^{\frac{\alpha-1}{2}}}\|\widetilde{u_1}\|_{X^{0,\frac{1}{2}^{+}}(\R\times \R^d)}\|\widetilde{u_2}\|_{X^{0,\frac{1}{2}^{+}}(\R\times \R^d)}.
			\end{equation*}
			Then, taking the infimum over all $\widetilde{u_1}, \widetilde{u_2}$, we conclude the bilinear estimate.
		\end{proof}
		\begin{rem}
			Using the same technique, we can also derive the following Strichartz estimate in Bourgain space from Lemma \ref{Strichartz1}:
			\begin{itemize}
				\item  For any 2-admissible pair $(q,r)$ and any spacetime slab $[0,T]\times \R^d$, we have 
				\begin{equation}\label{Strichartz3}
					\|P_N u\|_{L_t^{q}L_x^{r}([0,T]\times\R^d)}\lesssim N^{\gamma_{q,r}}\|u\|_{X^{0,\frac{1}{2}^{+}}}.
				\end{equation}
			\end{itemize}
		\end{rem}
		
		\vspace{7pt}
		Besides the above ``symmetric" bilinear estimate, we also need the following ``asymmetric" bilinear estimate, which will be used in establishing local well-posedness. This estimate can be viewed as a continuous and higher-dimensional analogue of Lemma 3.7 in \cite{15}.
		\begin{lemma}\label{asymmetric}
			If $d\ge 1$, then for any spacetime slab $I\times \R^d$ and $\phi_1: \R^d\to \C$, $f_2:\R^d\times\R\to \C$, the following estimate holds:
			\begin{equation*}
				\left\|\left(e^{-it|D|^{\alpha}}P_{N_1}\phi_1\right)\left( e^{-it|D|^{\alpha}}P_{N_2}f_2\right)\right\|_{L_{x,t}^{2}(I\times\R^d)}\lesssim \min\lbrace N_1, N_2\rbrace^{\frac{d}{2}-\frac{\alpha}{4}}\|\phi_1\|_{L^{2}(\R^d)}\|f_2\|_{H_{t}^{\frac{1}{4}^{+}}L_x^{2}(I\times\R^d)}.
			\end{equation*}
			In particular, by using the same trick as in Corollary \ref{bilinear2}, we have
			\begin{equation}\label{ee}
				\|P_{N_1}u_1 P_{N_2}u_2\|_{L_{x,t}^{2}(I\times\R^d)}\lesssim \min\lbrace N_1, N_2\rbrace^{\frac{d}{2}-\frac{\alpha}{4}}\|u_1\|_{X^{0,\frac{1}{2}^{+}}(I\times\R^d)}\|u_2\|_{X^{0,\frac{1}{4}^{+}}(I\times\R^d)}.
			\end{equation}
		\end{lemma}
		\begin{proof}
			For convenience, we suppose $I=[0,T]$ and  $P_{N_1}\phi_1=\phi_1, P_{N_2}f_2=f_2$.
			
			We first claim the following estimate with time-frequency localization:
			\begin{equation*}
				\left\|\left(e^{-it|D|^{\alpha}}\phi_1\right)\left( e^{-it|D|^{\alpha}}f_2\right)\right\|_{L_{x,t}^{2}([0,T]\times\R^d)}^{2}\lesssim \min\lbrace N_1, N_2\rbrace^{d-\frac{\alpha}{2}}T^{-\frac{1}{2}}L^{\frac{1}{2}}\|\phi_1\|_{L^{2}(\R^d)}^2 \|f_2\|_{L_{x,t}^{2}([0,T]\times\R^d)}^2,
			\end{equation*}
			where the space-time Fourier transform of $f_{2}\in L_{t,x}^{2}([0,T]\times\R^d)$ satisfies
			\[
			\operatorname{supp}\widetilde f_2\subset\{(\omega,\xi): |\omega|\sim LT^{-1},\ |\xi|\sim N_2\}, \quad L\in 2^{\Z}.
			\]
			If the claim holds, we can decompose $f_2=\sum_{L}f_2^{L}$, such that
			\begin{equation*}
				\supp \widetilde{f_2^{L}}\subset\{(\omega,\xi): |\omega|\sim LT^{-1},\ |\xi|\sim N_2\}.
			\end{equation*}
			Then applying Cauchy-Schwarz, we can obtain
			\begin{equation*}
				\left\|\left(e^{-it|D|^{\alpha}}\phi_1\right)\left( e^{-it|D|^{\alpha}}f_2\right)\right\|_{L_{x,t}^{2}([0,T]\times\R^d)}\le \sum_{L\in 2^{\Z_{\ge0}}}\left\|\left(e^{-it|D|^{\alpha}}\phi_1\right)\left( e^{-it|D|^{\alpha}}f_2^{L}\right)\right\|_{L_{x,t}^{2}([0,T]\times\R^d)}
			\end{equation*}
			\begin{equation*}
				\lesssim \min\lbrace N_1, N_2\rbrace^{\frac{d}{2}-\frac{\alpha}{4}}\|\phi_1\|_{L^{2}(\R^d)}\sum_{L\in 2^{\Z_{\ge0}}}T^{-\frac{1}{4}}L^{\frac{1}{4}} \|f_2^{L}\|_{L_{x,t}^{2}([0,T]\times\R^d)}
			\end{equation*}
			\begin{equation*}
				\lesssim \min\lbrace N_1, N_2\rbrace^{\frac{d}{2}-\frac{\alpha}{4}}\|\phi_1\|_{L^{2}(\R^d)}\sum_{L\in 2^{\Z_{\ge0}}}\frac{T^{-\frac{1}{4}}L^{\frac{1}{4}}}{\langle T^{-1}L\rangle^{\frac{1}{4}^+}} \|f_2^{L}\|_{H_t^{\frac{1}{4}^{+}}L_{x}^{2}([0,T]\times\R^d)}
			\end{equation*}
			\begin{equation*}
				\lesssim \min\lbrace N_1, N_2\rbrace^{\frac{d}{2}-\frac{\alpha}{4}}\|\phi_1\|_{L^{2}(\R^d)}\|f_2\|_{H_t^{\frac{1}{4}^{+}}L_{x}^{2}([0,T]\times\R^d)}.
			\end{equation*}
			To prove the claim, we apply Plancherel and obtain 
			\[
			\left\|\left(e^{-it|D|^{\alpha}}\phi_1\right)\left( e^{-it|D|^{\alpha}}f_2\right)\right \|_{L^2(\R^d)}^2
			=\int_{\mathbb{R}^d}\Bigl|\int_{\mathbb{R}^d} e^{-it\psi(\xi,\xi_1)}\widehat\phi_1(\xi_1)\widehat f_2(t,\xi-\xi_1)\,d\xi_1\Bigr|^2 d\xi,
			\]
			where \(\psi(\xi,\xi_1):=|\xi_1|^{\alpha}+|\xi-\xi_1|^{\alpha}\).  
			Expanding the square, it reduces to
			\begin{equation}\label{q}
				\int_{\mathbb{R}^d} d\xi\int\!\!\!\int_{\mathbb{R}^{d}\times\R^d} 
				e^{-it(\Psi(\xi,\xi_1,\xi_2))}\,
				\widehat\phi_1(\xi_1)\overline{\widehat\phi_1(\xi_2)}\,
				\widehat f_2(t,\xi-\xi_1)\overline{\widehat f_2(t,\xi-\xi_2)}d\xi_1\,d\xi_2,
			\end{equation}
			where $\Psi(\xi,\xi_1,\xi_2):=|\xi_1|^{\alpha}-|\xi_2|^{\alpha}+|\xi-\xi_1|^{\alpha}-|\xi-\xi_2|^{\alpha}$.

			Now substitute the following Fourier inversion in time
			\[
			\widehat f_2(t,\eta)=\int_{\mathbb{R}} e^{i\omega_1 t}\widetilde f_2(\omega_1,\eta)\,d\omega_1,\qquad
			\overline{\widehat f_2(t,\eta)}=\int_{\mathbb{R}} e^{-i\omega_2 t}\overline{\widetilde f_2(\omega_2,\eta)}\,d\omega_2,
			\]
			into (\ref{q}), it suffices to control
			\[
			\int_{\mathbb{R}^d} d\xi\int\!\!\!\int_{\mathbb{R}^{d}\times \R^d} \int\!\!\!\int_{\mathbb{R}\times\R} 
			e^{-it(\Psi(\xi,\xi_1,\xi_2)-\omega_1+\omega_2)}\;
			\widehat\phi_1(\xi_1)\overline{\widehat\phi_1(\xi_2)}\;
			\widetilde f_2(\omega_1,\xi-\xi_1)\overline{\widetilde f_2(\omega_2,\xi-\xi_2)}d\omega_1 d\omega_2\, d\xi_1\,d\xi_2.
			\]
			
			Let \(\eta\) be a smooth cut‑off function with \(\eta(t)=1\) for \(|t|\le1\) and \(\eta(t)=0\) for \(|t|\ge2\), and set \(\eta_T(t)=\eta(t/T)\). Then multiply $\eta_T$ and integrate on time, it reduces to 
			\begin{equation*}
				\int_{\mathbb{R}^d} d\xi \int\!\!\!\int_{\R^{d}\times\R}\int\!\!\!\int_{\R^{d}\times\R}\widehat{\eta_T}(\Psi(\xi, \xi_1, \xi_2)-\omega_1+\omega_2)\left[\widehat{\phi_1}(\xi_1)\cdot\widetilde{f_2}(\omega_1, \xi-\xi_1)\right]\left[\overline{\widehat{\phi_1}(\xi_2)}\cdot\overline{\widetilde{f_2}(\omega_2, \xi-\xi_2)}\right] d\xi_1 d\omega_1 d\xi_2 d\omega_2. 
			\end{equation*}
			
			For each fixed \(\xi\), we define the kernel
			\[
			K_\xi\bigl((\xi_1,\omega_1),(\xi_2,\omega_2)\bigr):=
			T\,\bigl|\widehat{\eta}\bigl(T(\Psi(\xi,\xi_1,\xi_2)-\omega_1+\omega_2)\bigr)\bigr|,
			\]
			and the sequences
			\[
			a_\xi(\xi_j,\omega_j):=|\widehat\phi_j(\xi_j)|\;|\widetilde f_2(\omega_j,\xi-\xi_j)|, \quad \forall j=1,2.
			\]
			Then it suffices to bound
			\[
			\int_{\mathbb{R}^d} d\xi\;
			\int\!\!\!\int_{\R^d\times\R}\int\!\!\!\int_{\R^d\times\R}
			K_\xi\bigl((\xi_1,\omega_1),(\xi_2,\omega_2)\bigr)\,
			a_\xi(\xi_1,\omega_1)\,a_\xi(\xi_2,\omega_2)d\xi_1 d\omega_1 d\xi_2 d\omega_2. 
			\]
			
			By Schur test and symmetry, it remains to control
			\begin{equation}\label{w}
				T\sup_{\xi,\xi_1}
				\int_{\substack{|\xi_2|\sim N_1\\ |\xi-\xi_2|\sim N_2}}
				\int_{|\omega|\lesssim LT^{-1}} \;
				\bigl|\widehat{\eta}\bigl(T(\Psi(\xi,\xi_1,\xi_2)-\omega)\bigr)\bigr|d\omega d\xi_2.
			\end{equation}
			Then for $L\gtrsim 1$, we can further have 
			\begin{equation*}
				(\ref{w})\lesssim T\sup_{\substack{J\subset \R  \\ |J|\lesssim 1}}\sup_{\xi,\xi_1}
				\int_{\substack{|\xi_2|\sim N_1\\ |\xi-\xi_2|\sim N_2}}\int_{\substack{|\omega|\lesssim LT^{-1}\\ T(\Psi(\xi,\xi_1,\xi_2)-\omega)\in J }} 1\; d\omega d\xi_2
			\end{equation*}
			\begin{equation*}
				\lesssim \sup_{\substack{J\subset \R  \\ |J|\lesssim 1}}\sup_{\xi,\xi_1}
				\Big|\lbrace\xi_2: |\xi_2|\sim N_1, |\xi-\xi_2|\sim N_2,\; \textbf{dist}(T\Psi(\xi,\xi_1,\xi_2), J)\lesssim L\rbrace\Big|
			\end{equation*}
			\begin{equation*}
				\lesssim \sup_{\substack{\widetilde{J}\subset \R  \\ |\widetilde{J}|\lesssim 1}}\sup_{\xi,\xi_1}
				\Big|\lbrace\xi_2: |\xi_2|\sim N_1, |\xi-\xi_2|\sim N_2,\; L^{-1} T \Psi(\xi,\xi_1,\xi_2)\in \widetilde{J}\rbrace\Big|.
			\end{equation*}
			If $N_1\le N_2$, we can calculate the radial derivatives (w.r.t $\xi_2$) as follows:
			\begin{equation*}
				-\partial_{r}\Psi(\xi, \xi_1, \xi_2)=\alpha|\xi_2|^{\alpha-1}-\alpha|\xi-\xi_2|^{\alpha-2}(\xi-\xi_2)\cdot\frac{\xi_2}{|\xi_2|},
			\end{equation*}
			\vspace{-10pt}
			\begin{equation*}
				-\partial_{r}^{2}\Psi(\xi, \xi_1, \xi_2)=\alpha(\alpha-1)|\xi_2|^{\alpha-2}+\alpha|\xi-\xi_2|^{\alpha-2}+\alpha(\alpha-2)|\xi-\xi_2|^{\alpha-4}\left(\frac{\xi_2\cdot(\xi_2-\xi)}{|\xi_2|}\right)^{2}.
			\end{equation*}
			Since $1<\alpha<2$, we see $|\partial_r^{2}\Psi|\gtrsim N_1^{\alpha-2}$, which implies 
			\begin{equation*}
				\Big|\lbrace\xi_2: |\xi_2|\sim N_1, |\xi-\xi_2|\sim N_2,\; L^{-1} T \Psi(\xi,\xi_1,\xi_2)\in \widetilde{J}\rbrace\Big|\lesssim \left(L^{-1}T N_1^{\alpha-2}\right)^{-\frac{1}{2}}N_1^{d-1}=T^{-\frac{1}{2}}L^{\frac{1}{2}}N_1^{d-\frac{\alpha}{2}}.
			\end{equation*}
			Similarly, if $N_1\ge N_2$, one can derive the upper bound $T^{-\frac{1}{2}}L^{\frac{1}{2}}N_2^{d-\frac{\alpha}{2}}$. 
			
		Now if $L\lesssim 1$, we similarly have 
		\begin{equation*}
			(\ref{w})\lesssim T\sup_{\substack{J\subset \R  \\ |J|\lesssim 1}}\sup_{\xi,\xi_1}
			\int_{\substack{|\xi_2|\sim N_1\\ |\xi-\xi_2|\sim N_2}}\int_{\substack{|\omega|\lesssim LT^{-1}\\ T(\Psi(\xi,\xi_1,\xi_2)-\omega)\in J }} 1\; d\omega d\xi_2
		\end{equation*}
		\begin{equation*}
			\lesssim L\sup_{\substack{\widetilde{J}\subset \R  \\ |\widetilde{J}|\lesssim 1}}\sup_{\xi,\xi_1}
			\Big|\lbrace\xi_2: |\xi_2|\sim N_1, |\xi-\xi_2|\sim N_2,\; T \Psi(\xi,\xi_1,\xi_2)\in \widetilde{J}\rbrace\Big|.
		\end{equation*}
		\begin{equation*}
			\lesssim T^{-\frac{1}{2}}L\min\lbrace N_1, N_2\rbrace^{d-\frac{\alpha}{2}}\lesssim T^{-\frac{1}{2}}L^{\frac{1}{2}}\min\lbrace N_1, N_2\rbrace^{d-\frac{\alpha}{2}}
		\end{equation*}
		
		Thus we complete the whole proof.
		\end{proof}
		\begin{rem}
			Interpolating the bilinear estimates (\ref{e}) and (\ref{ee}), we can immediately obtain
			\begin{equation}\label{aa}
				\|P_{N_1}u_1 P_{N_2}u_2\|_{L_{x,t}^{2}(I\times\R^d)}\lesssim \frac{\min\lbrace N_1, N_2\rbrace^{\frac{d-1}{2}^{+}}}{\max\lbrace N_1, N_2\rbrace^{\frac{\alpha-1}{2}}}\|u_1\|_{X^{0,\frac{1}{2}^{\pm}}(I\times \R^d)}\|u_2\|_{X^{0,\frac{1}{2}^{\mp}}(I\times \R^d)},
			\end{equation}
			which allows one of the exponents to be strictly smaller than $\frac{1}{2}$.
		\end{rem}
		\begin{rem}
			From the proof of Lemma \ref{bilinear2} and Lemma \ref{asymmetric}, we can actually replace the product $P_{N_1}u_1 \cdot P_{N_2}u_2$ on the left with $P_{N_1}\overline{u_1}\cdot P_{N_2}u_2$ or $P_{N_1}u_1\cdot P_{N_2}\overline{u_2}$. In fact, throughout this paper, we can always ignore the conjugate for convenience.
		\end{rem}
		
	\vspace{15pt}
	\section{LWP of fractional Schr\"{o}dinger equation in $\R^d$}
	As an application of the bilinear estimates in Section 2, we can establish the local well-posedness of the following modified fractional Schr\"{o}dinger equation (\ref{IFNLS}) whenever $s>\frac{d-\alpha}{2}$, i.e., the full subcritical range.
	\begin{equation}\label{IFNLS}
		\begin{cases}
			i\partial_t Iu(t,x) = |D|^{\alpha} Iu + I(|u|^2 u), \\[4pt]
			Iu(0,x) = Iu_0(x), \qquad (t,x)\in \R \times \R^d,
		\end{cases}
	\end{equation}
	Recall that the case we are interested in before the last section is $s<\frac{\alpha}{2}$, i.e., the regularity is below the energy threshold. Thus, we will always suppose $\alpha>\frac{d}{2}$ to ensure the range of $s$ is not empty, which actually restricted our attention to $d=2,3$.
	
	It should also be mentioned that the well-posedness of the modified equation (\ref{IFNLS}) is essentially equivalent to that of the original equation (\ref{FNLS}). Dinh has established the latter in \cite{4}, fully relying on the Strichartz estimate. However, Dinh's theory was built on a ``non-Bourgain" regime, which is incompatible with our further discussions. Thus, we provide a detailed proof here, and readers may find the power of the bilinear estimates.
	
	Before proving the local well-posedness, we recall some classic lemmas:
	\begin{lemma}\label{Gini}
		Let $0\le b'<\frac{1}{2}$, $0<b<1-b'$, and $0<T \le 1$. Then we have 
		\begin{equation*}
			\left\|\varphi\left(\frac{t}{T}\right)\int_{0}^{t}e^{-i(t-\tau)|D|^{\alpha}}f(\tau)\,d\tau\right\|_{X_{\alpha}^{\gamma,b}(\R\times\T^d)} \le C\, T^{1-(b+b')} \|f\|_{X_{\alpha}^{\gamma,-b'}(\R\times\T^d)},
		\end{equation*}
		for all $s \in \mathbb{R}$, with some constant $C>0$.
	\end{lemma}
	\begin{proof}
		See Lemma 3.2 in \cite{16}.
	\end{proof}
	For any $x\in \R^d$, we define $\tau_x$ to be the translation operator
	\begin{equation*}
		\tau_x u(y,t):=u(x+y,t), \quad \forall y\in \R^d,\; t\in \R.
	\end{equation*}
	A multi-linear operator $T(u_1, u_2, \cdots u_n)$ is said to be translation invariant if 
	\begin{equation*}
		\tau_{x}T(u_1, u_2, \cdots u_n)=T(\tau_x u_1, \tau_x u_2, \cdots, \tau_x u_n), \quad \forall x\in \R^d.
	\end{equation*}
	Similarly, a Banach space $X$ of space-time functions defined on $\R^d\times \R$ is called translation invariant if 
	\begin{equation*}
		\|\tau_x u\|_{X}=\|u\|_{X}, \quad \forall u\in X,\; x\in \R^d.
	\end{equation*}
	\begin{lemma}\label{interpolation}
		Let \(\beta_0 > 0\) and \(n \ge 1\). Suppose that \(Z, X_1, \dots, X_n\) are translation invariant Banach spaces and \(T\) is a translation invariant \(n\)-linear operator such that the following estimates hold
		\[
		\|I_1^\beta T(u_1, \dots, u_n)\|_{Z} \lesssim \prod_{i=1}^n \|I_1^\beta u_i\|_{X_i}
		\]
		for all \(u_1, \dots, u_n\) and all \(0 \le \beta \le \beta_0\). Then one has the estimates
		\[
		\|I_N^\beta T(u_1, \dots, u_n)\|_{Z} \lesssim \prod_{i=1}^n \|I_N^\beta u_i\|_{X_i}
		\]
		for all \(u_1, \dots, u_n\), all \(0 \le \beta \le \beta_0\), and \(N \ge 1\), with the implicit constant independent of \(N\).
	\end{lemma}
	\begin{proof}
		See Lemma 12.1 in \cite{17}.
	\end{proof}
		
		Now, we can give our statement on local well-posedness:
		\begin{prop}\label{LWP}
			Suppose $\frac{\alpha}{2}>s>\frac{d-\alpha}{2}$, and the initial data for the modified fractional Schr\"{o}dinger equation (\ref{IFNLS}) satisfies $\||D|^{\frac{\alpha}{2}}Iu_0\|_{L^{2}(\R^d)}\le M$. Then there exists a constant $T=T(M)>0$, such that the equation is locally well-posed on time interval $[0,T]$, with the following bound:
			\begin{equation*}
				\||D|^{\frac{\alpha}{2}}Iu\|_{X^{0,\frac{1}{2}^{+}}([0,T]\times\R^d)}\lesssim \||D|^{\frac{\alpha}{2}}Iu_0\|_{L^{2}(\R^d)}.
			\end{equation*}
		\end{prop}
		\begin{proof}
			We consider the following solution map:
			\begin{equation*}
				\Phi(u):= \varphi(t)e^{-it|D|^{\alpha}}u_0-i\varphi\left(\frac{t}{T}\right)\int_{0}^{t}e^{-i(t-\tau)|D|^{\alpha}}|u(\tau)|^2u(\tau)d\tau,
			\end{equation*}
			where $\varphi\in C_c^{\infty}(\R^d)$ is a smooth cutoff adapted to $[-1,1]$.
			
			Then, following a typical iteration argument, we can set 
			\begin{equation*}
				B_R:=\left\{u:\;\||D|^{\frac{\alpha}{2}}Iu\|_{X^{0,\frac{1}{2}^{+}}([0,T]\times\R^d)}\le R\||D|^{\frac{\alpha}{2}}Iu_0\|_{L^{2}(\R^d)}\right\}.
			\end{equation*}
			It remains to show that $\Phi$ maps $B_R$ to itself and is also a contraction. 
			
			First, applying Lemma \ref{Gini}, we can obtain
			\begin{equation*}
				\||D|^{\frac{\alpha}{2}}I \Phi(u)\|_{X^{0,\frac{1}{2}^{+}}([0,T]\times\R^d)}\lesssim \||D|^{\frac{\alpha}{2}}Iu_0\|_{L^{2}(\R^d)}+\left\|\varphi\left(\frac{t}{T}\right)\int_{0}^{t}e^{-i(t-\tau)|D|^{\alpha}}|D|^{\frac{\alpha}{2}}I(|u|^2u)d\tau\right\|_{X^{0,\frac{1}{2}^{+}}([0,T]\times\R^d)}
			\end{equation*}
			\begin{equation*}
				\lesssim \||D|^{\frac{\alpha}{2}}Iu_0\|_{L^{2}(\R^d)}+T^{0^{+}}\||D|^{\frac{\alpha}{2}}I(|u|^2 u)\|_{X^{0,-\frac{1}{2}^{-}}([0,T]\times\R^d)}.
			\end{equation*}
			Thus, it suffices to show
			\begin{equation}\label{key}
				\||D|^{\frac{\alpha}{2}}I(\phi_1 \overline{\phi_2}\phi_3)\|_{X^{0,-\frac{1}{2}^{-}}([0,T]\times\R^d)}\lesssim \prod_{j=1}^{3}\||D|^{\frac{\alpha}{2}}I\phi_j\|_{X^{0,\frac{1}{2}^{+}}([0,T]\times\R^d)}.
			\end{equation}
			Applying the interpolation lemma (Lemma \ref{interpolation}), we can assume $N=1$. We will split the different frequency interactions into several cases. 
			
			In fact, we say $\phi$ has high frequency if $\supp\; \widehat{\phi}\subset \lbrace|\xi|\gtrsim 1\rbrace$; we say $\phi$ has low frequency if $\supp\; \widehat{\phi}\subset \lbrace|\xi|\lesssim 1\rbrace$.
			
			\textbf{Case 1:} $\phi_1, \phi_2, \phi_3$ all have high frequency.
			
			We can actually prove a slightly stronger statement:
			\begin{equation*}
				\|I(\phi_1 \overline{\phi_2}\phi_3)\|_{X^{\frac{\alpha}{2}, -\frac{1}{2}^{-}}([0,T]\times\R^d)}\lesssim \prod_{j=1}^{3}\|I\phi_j\|_{X^{\frac{\alpha}{2},\frac{1}{2}^{+}}([0,T]\times\R^d)}.
			\end{equation*}
			From the definition of the $I$-operator, it is equivalent to 
			\begin{equation*}
				\|\phi_1 \overline{\phi_2}\phi_3\|_{X^{s, -\frac{1}{2}^{-}}([0,T]\times\R^d)}\lesssim \prod_{j=1}^{3}\|\phi_j\|_{X^{s,\frac{1}{2}^{+}}([0,T]\times\R^d)}.
			\end{equation*}
			Then by duality and dyadic decomposition on frequency, it remains to prove
			\begin{equation*}
				L(N_1, N_2, N_3, N_0):=\Bigg|\int_{0}^{T}\int_{\R^d}P_{N_1}\phi_1\cdot \overline{P_{N_2}\phi_2}P_{N_3}\phi_3\cdot \overline{P_{N_0}\phi_0} dx dt\Bigg|
			\end{equation*}
			\begin{equation*}
				\lesssim (N_1N_2N_3)^{0^-}N_0^{0\pm}\|\phi_0\|_{X^{-s, \frac{1}{2}^{-}}([0,T]\times\R^d)}\prod_{j=1}^{3}\|\phi_j\|_{X^{s,\frac{1}{2}^{+}}([0,T]\times\R^d)},
			\end{equation*}
			where we use the notations $N_i^{0\pm}$ to denote 
			\begin{equation*}
				N_i^{0\pm}:=
				\begin{cases}
					N_i^{0-}, & \text{if } N_i\ge 1 \\[4pt]
					N_i^{0+}, & \text{if } N_i\le 1.
				\end{cases}     
			\end{equation*}
			
			By symmetry, we can assume $N_1\ge N_2\ge N_3$. Furthermore, we suppose $N_1\gtrsim N_0$, otherwise $L(N_1, N_2, N_3, N_0)$ is obviously vanishing.
			
			\textbf{Sub-case 1.1:} $N_1\ge N_2\ge N_3$ \& $N_0\ge N_3$.
			
			We pair up $(N_1, N_2)$, $(N_0, N_3)$ and apply the bilinear estimates (\ref{e}), (\ref{aa}) respectively,
			\begin{equation*}
				L(N_1, N_2, N_3, N_0) \lesssim \frac{N_2^{\frac{d-1}{2}}}{N_1^{\frac{\alpha-1}{2}}}\cdot\frac{N_3^{\frac{d-1}{2}^{+}}}{N_0^{\frac{\alpha-1}{2}}}\cdot \frac{N_0^{s}}{N_1^{s}N_2^{s}N_3^{s}}\|\phi_0\|_{X^{-s, \frac{1}{2}^{-}}([0,T]\times\R^d)}\prod_{j=1}^{3}\|\phi_j\|_{X^{s,\frac{1}{2}^{+}}([0,T]\times\R^d)}.
			\end{equation*}
			Simple calculation shows that when $s>\frac{d-\alpha}{2}$, we have 
			\begin{equation*}
				\frac{N_2^{\frac{d-1}{2}}}{N_1^{\frac{\alpha-1}{2}}}\cdot\frac{N_3^{\frac{d-1}{2}^{+}}}{N_0^{\frac{\alpha-1}{2}}}\cdot \frac{N_0^{s}}{N_1^{s}N_2^{s}N_3^{s}}\lesssim N_1^{0^{-}}.
			\end{equation*}
			
			\textbf{Sub-case 1.2:} $N_1\ge N_2\ge N_3\ge N_0$.
			
			We pair up $(N_1, N_0)$, $(N_2, N_3)$ and apply the bilinear estimates (\ref{aa}), (\ref{e}) respectively,
			\begin{equation*}
				L(N_1, N_2, N_3, N_0) \lesssim \frac{N_0^{\frac{d-1}{2}^{+}}}{N_1^{\frac{\alpha-1}{2}}}\cdot\frac{N_3^{\frac{d-1}{2}}}{N_2^{\frac{\alpha-1}{2}}}\cdot \frac{\langle N_0\rangle^{s}}{N_1^{s}N_2^{s}N_3^{s}}\|\phi_0\|_{X^{-s, \frac{1}{2}^{-}}([0,T]\times\R^d)}\prod_{j=1}^{3}\|\phi_j\|_{X^{s,\frac{1}{2}^{+}}([0,T]\times\R^d)}.
			\end{equation*}
			Similarly, one can check 
			\begin{equation*}
				\frac{N_0^{\frac{d-1}{2}^{+}}}{N_1^{\frac{\alpha-1}{2}}}\cdot\frac{N_3^{\frac{d-1}{2}}}{N_2^{\frac{\alpha-1}{2}}}\cdot \frac{\langle N_0\rangle^{s}}{N_1^{s}N_2^{s}N_3^{s}}\lesssim N_1^{0^{-}}N_0^{0\pm}.
			\end{equation*}
			
			\vspace{10pt}
			Before starting the low frequency case, we note that (\ref{key}) can be reduced to 
			\begin{equation}\label{f}
				\Bigg|\int_{0}^{T}\int_{\R^d}|D|^{\frac{\alpha}{2}}I\phi_1 \cdot\overline{\phi_2}\cdot\phi_3\cdot\overline{\phi_0} dxdt\Bigg|\lesssim \|\phi_0\|_{X^{0,\frac{1}{2}^{-}}([0,T]\times\R^d)}\prod_{j=1}^{3}\||D|^{\frac{\alpha}{2}}I\phi_j\|_{X^{0,\frac{1}{2}^{+}}([0,T]\times\R^d)},
			\end{equation}
			from the fractional Leibniz rule and duality.
			
			\textbf{Case 2:} $\phi_2$, $\phi_3$ have low frequency.
			
			Applying H\"{o}lder's inequality with the factors in $L_t^{\infty}L_x^{2}$, $L_t^{2^{+}}L_x^{\infty}$, $L_t^{2^{+}}L_x^{\infty}$, and $L_t^{\infty^{-}}L_x^{2}$, respectively. Note that by Bernstein's inequality, H\"{o}lder in time and Sobolev embedding, we have, for $i=2,3$,
			\begin{equation*}
				\|\phi_i\|_{L_t^{2^{+}}L_x^{\infty}([0,T]\times\R^d)}\lesssim \|\phi_i\|_{L_t^{2^{+}}L_x^{\infty^{-}}([0,T]\times\R^d)}\lesssim_T \||D|^{\frac{d}{2}^{-}}\phi_i\|_{L_t^{\infty}L_x^{2}([0,T]\times\R^d)}
			\end{equation*}
			\begin{equation*}
				\lesssim \||D|^{\frac{d}{2}^{-}}\phi_i\|_{X^{0,\frac{1}{2}^{+}}([0,T]\times\R^d)}\lesssim \||D|^{\frac{\alpha}{2}}\phi_i\|_{X^{0,\frac{1}{2}^{+}}([0,T]\times\R^d)}\sim \||D|^{\frac{\alpha}{2}}I\phi_i\|_{X^{0,\frac{1}{2}^{+}}([0,T]\times\R^d)}.
			\end{equation*}
			Interpolating between $\|\phi\|_{L_{x,t}^{2}}\lesssim \|\phi\|_{X^{0,0}}$ and $\|\phi\|_{L_t^{\infty}L_{x}^{2}}\lesssim \|\phi\|_{X^{0,\frac{1}{2}^{+}}}$, we can derive
			\begin{equation*}
				\|\phi_0\|_{L_t^{\infty^{-}}L_x^{2}([0,T]\times\R^d)}\lesssim \|\phi_0\|_{X^{0,\frac{1}{2}^{-}}([0,T]\times\R^d)}.
			\end{equation*}
			
			\textbf{Case 3:} $\phi_1$, $\phi_2$ have high frequency and $\phi_3$ has low frequency.
			
			Applying H\"{o}lder's inequality, the left side of (\ref{f}) can be bounded by
			\begin{equation*}
				\|(|D|^{\frac{\alpha}{2}}I\phi_1) \phi_2\|_{L_{x,t}^{2}([0,T]\times\R^d)}\|\phi_3\|_{L_t^{2^{+}}L_x^{\infty}([0,T]\times\R^d)}\|\phi_0\|_{L_t^{\infty^{-}}L_x^{2}([0,T]\times\R^d)}.
			\end{equation*}
			The last two terms can be controlled as in \textbf{Case 2}. For the first term, we can use dyadic decomposition on frequency and the bilinear estimate (\ref{e}), which yields
			\begin{equation*}
				\|(|D|^{\frac{\alpha}{2}}I\phi_1) \phi_2\|_{L_{x,t}^{2}([0,T]\times\R^d)}\lesssim \sum_{N_1, N_2\gtrsim1} 
				\frac{N_2^{\frac{1}{2}(d-\alpha)^{+}}}{N_1^{0^{+}}}\||D|^{\frac{\alpha}{2}}I\phi_1\|_{X^{0,\frac{1}{2}^{+}}([0,T]\times\R^d)}\|\phi_2\|_{X^{0,\frac{1}{2}^{+}}([0,T]\times\R^d)}
			\end{equation*}
			\begin{equation*}
				\lesssim \sum_{N_1, N_2\gtrsim1}  N_1^{0^{-}}N_2^{\frac{1}{2}(d-\alpha)^{+}}N_2^{-s}\||D|^{\frac{\alpha}{2}}I\phi_1\|_{X^{0,\frac{1}{2}^{+}}([0,T]\times\R^d)}\||D|^{\frac{\alpha}{2}}I\phi_2\|_{X^{0,\frac{1}{2}^{+}}([0,T]\times\R^d)}.
			\end{equation*}
			Since $s>\frac{d-\alpha}{2}$, the summation converges.
			
			\textbf{Case 4:} $\phi_1$, $\phi_2$ have low frequency and $\phi_3$ has high frequency.
			
			Applying H\"{o}lder's inequality with the factors in $L_t^{2^{+}}L_x^{\infty}$, $L_t^{2^{+}}L_x^{\infty}$, $L_t^{\infty}L_x^{2}$, and $L_t^{\infty^{-}}L_x^{2}$, respectively. The bounds for the second and fourth terms can be found in \textbf{Case 2}. 
			
			For the first term, we can use Bernstein's inequality and H\"{o}lder in time:
			\begin{equation*}
				\||D|^{\frac{\alpha}{2}}I\phi_1\|_{L_t^{2^{+}}L_x^{\infty}([0,T]\times\R^d)}\lesssim \||D|^{\frac{\alpha}{2}}I\phi_1\|_{L_t^{\infty}L_x^{2}([0,T]\times\R^d)}\lesssim \||D|^{\frac{\alpha}{2}}I\phi_1\|_{X^{0,\frac{1}{2}^{+}}([0,T]\times\R^d)}.
			\end{equation*}
			
			For the third term, we also have 
			\begin{equation*}
				\|\phi_3\|_{L_t^{\infty}L_x^{2}([0,T]\times\R^d)}\lesssim \|\phi_3\|_{X^{0,\frac{1}{2}^{+}}([0,T]\times\R^d)}\lesssim \||D|^s\phi_3\|_{X^{0,\frac{1}{2}^{+}}([0,T]\times\R^d)}\lesssim\||D|^{\frac{\alpha}{2}}I\phi_3\|_{X^{0,\frac{1}{2}^{+}}([0,T]\times\R^d)}.
			\end{equation*}
			
			\textbf{Case 5:} $\phi_1$ has low frequency and $\phi_2, \phi_3$ have high frequency.
			
			Applying H\"{o}lder's inequality, the left side of (\ref{f}) can be bounded by
			\begin{equation*}
				\||D|^{\frac{\alpha}{2}}I\phi_1 \|_{L_{t}^{2^{+}}L_x^{\infty}([0,T]\times\R^d)}\|\phi_2\phi_3\|_{L_{x,t}^{2}([0,T]\times\R^d)}\|\phi_0\|_{L_t^{\infty^{-}}L_x^{2}([0,T]\times\R^d)}.
			\end{equation*}
			The first and third terms can be found in \textbf{Case 4} and \textbf{Case 2}, respectively. Using the bilinear estimate (\ref{e}) as in \textbf{Case 3}, we can also control the second term.
			
			\vspace{10pt}
			So far, we have completed the proof of (\ref{key}). Thus, the solution map $\Phi$ sends $B_R$ to itself. Applying (\ref{key}) and Lemma \ref{Gini}, we can similarly prove that $\Phi$ is also a contraction, which ends the whole proof of local well-posedness.
		\end{proof}
		\begin{rem}
			Proposition \ref{LWP} establishes a ``homogeneous" version of local well-posedness. In fact, we can also have the following ``inhomogeneous" version, which is more common in the literature.
			\begin{itemize}
				\item 
				Suppose $\frac{\alpha}{2}>s>\frac{d-\alpha}{2}$, and the initial data for the modified fractional Schr\"{o}dinger equation (\ref{IFNLS}) satisfies $\|Iu_0\|_{H^{\frac{\alpha}{2}}(\R^d)}\le M$. Then there exists a constant $T=T(M)>0$, such that the equation is locally well-posed on the time interval $[0,T]$, with the following bound:
				\begin{equation*}
					\|Iu\|_{X^{\frac{\alpha}{2},\frac{1}{2}^{+}}([0,T]\times\R^d)}\lesssim \|Iu_0\|_{H^{\frac{\alpha}{2}}(\R^d)}.
				\end{equation*}
			\end{itemize}
			The proof is even simpler and relies only on the argument in \textbf{Case 1}, based entirely on the bilinear estimates \eqref{e}, \eqref{aa}. 
			This highlights the strength of bilinear estimates: there is no need for delicate choices of the associated exponents. 
			
			In contrast, the proof of Proposition 4.2 in \cite{18} relies on Strichartz estimates; a comparison with our argument highlights the simplification achieved here.
		\end{rem}
		
		\vspace{15pt}
		\section{GWP of fractional Schr\"{o}dinger equation in $\R^3$}
		The following main proposition represents the ``almost conservation law" of the modified solution $Iu$.
		\begin{prop}\label{d=3}
			For $d=3$, $s>1-\frac{\alpha}{6}$, we suppose $u$ is the solution in Proposition \ref{LWP}; then the following estimate holds:
			\begin{equation*}
				E(Iu)(t)-E(Iu)(0)=O(N^{-(2\alpha-3)^{-}}), \quad \forall t\in [0,T].
			\end{equation*}
		\end{prop}
		\vspace{7pt}
		With the result in Proposition \ref{d=3}, we can directly prove Theorem \ref{main d=3}.
		
		\begin{proof}[Proof of Theorem \ref{main d=3}]
			From the scaling symmetry, if $u(x,t)$ solves (\ref{FNLS}), then so does
			\begin{equation*}
				u^{\lambda}(x,t):=\frac{1}{\lambda^{\frac{\alpha}{2}}}u\left(\frac{x}{\lambda},\frac{t}{\lambda^{\alpha}}\right), \quad u_0^{\lambda}(x,t):=\frac{1}{\lambda^{\frac{\alpha}{2}}}u_0\left(\frac{x}{\lambda}\right).
			\end{equation*}
			
			A simple calculation shows that 
			\begin{equation*}
				\|Iu_0^{\lambda}\|_{\dot{H}^{\frac{\alpha}{2}}(\R^3)}\lesssim N^{\frac{\alpha}{2}-s}\|u_0^{\lambda}\|_{\dot{H}^{s}(\R^3)}=N^{\frac{\alpha}{2}-s} \lambda^{-\left(s+\frac{\alpha}{2}-\frac{3}{2}\right)}\|u_0\|_{\dot{H}^{s}(\R^3)}.
			\end{equation*}
			To ensure that the modified solution $Iu^{\lambda}$ has bounded initial data $Iu_0^{\lambda}$, we can choose
			\begin{equation}\label{cccc}
				\lambda\sim N^{\frac{\frac{\alpha}{2}-s}{s+\frac{\alpha}{2}-\frac{3}{2}}}.
			\end{equation}
			
			Now, we apply Proposition \ref{d=3} to the scaled initial data $Iu_0^{\lambda}$ and repeat it $N^{(2\alpha-3)^-}$ times until $E(Iu^{\lambda}(t))$ reaches $1$, which leads to 
			\begin{equation}\label{x}
				\sup_{t\in [0,C N^{(2\alpha-3)^-}]}E(Iu^{\lambda}(t))\lesssim 1,
			\end{equation}
			for some universal constant $C>0$.
			
			Then, given any $T_0\gg 1$, we can choose the parameter $N$ satisfying
			\begin{equation}\label{cc}
				T_0\sim \frac{N^{(2\alpha-3)^-}}{\lambda^{\alpha}}\sim N^{2\alpha-3-\frac{\alpha(\frac{\alpha}{2}-s)}{s+\frac{\alpha}{2}-\frac{3}{2}}^-},
			\end{equation}
			when the exponent of $N$ is positive. One can check that this is equivalent to 
			\begin{equation*}
				\frac{\alpha}{2}>s>\frac{\alpha}{2}-\frac{(2\alpha-3)^2}{6(\alpha-1)}.
			\end{equation*}
			
			To obtain the polynomial-in-time growth, we notice that
			\begin{equation*}
				E(I_N u^{\lambda})(\lambda^{\alpha}t)=\lambda^{3-2\alpha}E(I_{\lambda N}u)(t).
			\end{equation*}
			Then, applying (\ref{second}), mass conservation and substituting (\ref{cccc}), (\ref{x}), (\ref{cc}), one can conclude
			\begin{equation*}
				\|u(T_0)\|_{H^s(\R^3)}^2\lesssim \|I_Nu(t)\|_{H^{\frac{\alpha}{2}}(\R^3)}^2\lesssim E(I_{\lambda N}u)(t)+\|I_N u(t)\|_{L^{2}(\R^3)}^2
			\end{equation*}
			\begin{equation*}
				\lesssim \lambda^{2\alpha-3}E(I_N u^{\lambda})(\lambda^{\alpha}T_0)+\|u_0\|_{L^2(\R^3)}^2
				\lesssim T_0^{\frac{(2\alpha-3)(\frac{\alpha}{2}-s)^+}{3(\alpha-1)(s-r_{3}(\alpha))}}.
			\end{equation*}
		\end{proof}

		Before starting the proof of Proposition \ref{d=3}, we first recall that
		\begin{equation*}
			E(Iu)(t)=\int_{\R^3}\frac{1}{2}\big||D|^{\frac{\alpha}{2}}Iu(t)\big|^{2}dx+\int_{\R^3}\frac{1}{4}|Iu(t)|^4 dx.
		\end{equation*}
		Then we take the time derivative and obtain 
		\begin{equation*}
			\partial_t E(I\phi)(t)=\Re e\int_{\R^3}\overline{\partial_t Iu}\left(|Iu|^{2}Iu+|D|^{\alpha}u\right)dx
			=\Re e\int_{\R^3}\overline{\partial_t Iu}\left(|Iu|^{2}Iu-I(|u|^2u)\right)dx.
		\end{equation*}
		Integrating in time, we can derive
		\begin{equation*}
			E(Iu)(t)-E(Iu)(0)=\int_{0}^{t}\int_{\sum_{j=1}^{4}\xi_j=0}\left(1-\frac{m(\xi_1))}{m(\xi_2)m(\xi_3)m(\xi_4)}\right)\widehat{\overline{I\partial_t u}}(\xi_1)\widehat{Iu}(\xi_2)\widehat{\overline{Iu}}(\xi_3)\widehat{Iu}(\xi_4)
		\end{equation*}
		\begin{equation*}
			=: \textbf{Term}_1+\textbf{Term}_2,
		\end{equation*}
		where $\textbf{Term}_1$ and $\textbf{Term}_2$ are defined as
		\begin{itemize}
			\item $\displaystyle
			|\textbf{Term}_1| := \Bigg|\int_{0}^{t}\int_{\sum_{j=1}^{4}\xi_j=0}
			\left(1-\frac{m(\xi_1)}{m(\xi_2)m(\xi_3)m(\xi_4)}\right)
			\widehat{\overline{|D|^{\alpha}Iu}}(\xi_1)
			\widehat{Iu}(\xi_2)
			\widehat{\overline{Iu}}(\xi_3)
			\widehat{Iu}(\xi_4)\Bigg|;
			$
			\item $\displaystyle
			|\textbf{Term}_2| := \Bigg|\int_{0}^{t}\int_{\sum_{j=1}^{4}\xi_j=0}
			\left(1-\frac{m(\xi_1)}{m(\xi_2)m(\xi_3)m(\xi_4)}\right)
			\widehat{\overline{I(|u|^2 u)}}(\xi_1)
			\widehat{Iu}(\xi_2)
			\widehat{\overline{Iu}}(\xi_3)
			\widehat{Iu}(\xi_4)\Bigg|.
			$
		\end{itemize}
		Thus, the proof of Proposition \ref{d=3} is reduced to 
		\begin{equation*}
			|\textbf{Term}_1|, |\textbf{Term}_2|\lesssim N^{-(2\alpha-3)^{-}}.
		\end{equation*}
		
		By the definition of Bourgain space, we have 
		\begin{equation*}
			\||D|^{\alpha}I\phi\|_{X^{-\frac{\alpha}{2}, \frac{1}{2}^{+}}}\lesssim \||D|^{\frac{\alpha}{2}}I\phi\|_{X^{0, \frac{1}{2}^{+}}}\lesssim 1. 
		\end{equation*}
		Thus, we can conclude $|\textbf{Term}_1|, |\textbf{Term}_2|\lesssim N^{-(2\alpha-3)^{-}}$ once we prove
		\begin{equation}\label{term1}
			\Bigg|\int_{0}^{t}\int_{\sum_{j=1}^{4}\xi_j=0}
			\left(1-\frac{m(\xi_1)}{m(\xi_2)m(\xi_3)m(\xi_4)}\right)
			\widehat{\overline{\phi_1}}(\xi_1)
			\widehat{\phi_2}(\xi_2)
			\widehat{\overline{\phi_3}}(\xi_3)
			\widehat{\phi_4}(\xi_4)\Bigg|
		\end{equation}
		\begin{equation*}
			\lesssim N^{-(2\alpha-3)^{-}} N_1^{0\pm} N_2^{0\pm} N_3^{0\pm} N_4^{0\pm}\|\phi_1\|_{X^{-\frac{\alpha}{2}, \frac{1}{2}^{+}}}\||D|^{\frac{\alpha}{2}}\phi_2\|_{X^{0, \frac{1}{2}^{+}}}\||D|^{\frac{\alpha}{2}}\phi_3\|_{X^{0, \frac{1}{2}^{+}}}\||D|^{\frac{\alpha}{2}}\phi_4\|_{X^{0, \frac{1}{2}^{+}}},
		\end{equation*}
		and 
		\begin{equation}\label{term2}
			\Bigg|\int_{0}^{t}\int_{\sum_{j=1}^{6}\xi_j=0}
			\left(1-\frac{m(\xi_1+\xi_2+\xi_3)}{m(\xi_4)m(\xi_5)m(\xi_6)}\right)
			P_{N_{123}}\widehat{\overline{I(\phi_1\phi_2\phi_3)}}(\xi_1+\xi_2+\xi_3)
			\widehat{I\phi_4}(\xi_4)
			\widehat{\overline{I\phi_5}}(\xi_5)
			\widehat{I\phi_6}(\xi_6)\Bigg|
		\end{equation}
		\begin{equation*}
			\lesssim N^{-(2\alpha-3)^{-}}N_{123}^{0\pm}N_1^{0\pm}N_2^{0\pm}N_3^{0\pm}N_4^{0\pm}N_5^{0\pm}N_6^{0\pm}\prod_{i=1}^{6}\||D|^{\frac{\alpha}{2}}I\phi_i\|_{X^{0,\frac{1}{2}^{+}}},
		\end{equation*}
		for any functions $\phi_i$, with spatial Fourier transform supported on $|\xi_i|\sim N_i$. 
		
		Next, we introduce the Coifman-Meyer multiplier theorem (see \cite{13,35,14}), which plays an important role in establishing the ``almost conservation law".
		 
		 \vspace{5pt}
		 For the multilinear operator $\Lambda$ given by
		 \begin{equation*}
		 	\Lambda(f_1,\cdots, f_k)(x):=\int_{\R^{dk}}e^{ix(\xi_1+\cdots+\xi_k)}\sigma(\xi_1,\cdots,\xi_k)\widehat{f_1}(\xi_1)\cdots\widehat{f_k}(\xi_k)d\xi_1\cdots d\xi_k,
		 \end{equation*}
		 we suppose that the symbol $\sigma:\R^{dk}\to \C$ is smooth and for all $\alpha\in \N^{dk}$ and $\xi=(\xi_1,\cdots,\xi_k)\in \R^{dk}$, there exists a constant $C(\alpha)$ such that
		 \begin{equation*}
		 	|\partial_{\xi}^{\alpha}\sigma(\xi)|\le C(\alpha)(1+|\xi|)^{-|\alpha|}.
		 \end{equation*}
		 Then the following Coifman-Meyer multiplier theorem holds:
		 \begin{thm}
		 	Let $p_j\in (1,+\infty)$, $j=1,\cdots,k$ satisfy
		 	\begin{equation*}
		 		\frac{1}{p}=\frac{1}{p_1}+\frac{1}{p_2}+\cdots+\frac{1}{p_k}\le 1.
		 	\end{equation*}
		 	Then there exists a constant $C=C(p_i,d,k,c(\alpha))$ such that for $f_1, \cdots ,f_k\in \mathcal{S}(\R^d)$,
		 	\begin{equation*}
		 		\|\Lambda(f_1,\cdots,f_k)\|_{L^p(\R^d)}\le C\|f_1\|_{L^{p_1}(\R^d)}\cdots\|f_k\|_{L^{p_k}(\R^d)}.
		 	\end{equation*}
		 \end{thm}
		 Take (\ref{term1}) as an example. If we have a pointwise bound on the symbol,
		 \begin{equation*}
		 	\Bigg|1-\frac{m(\xi_2+\xi_3+\xi_4)}{m(\xi_2)m(\xi_3)m(\xi_4)}\Bigg|\le M(N_2, N_3, N_4),
		 \end{equation*}
		 we can factor out $M(N_2, N_3, N_4)$ and rewrite (\ref{term1}) as
		 \begin{equation*}
		 	\Bigg|M(N_2, N_3, N_4) \int_0^t \int_{\R^3} \widehat{\Lambda (\overline{\phi_1}, \phi_2, \overline{\phi_3})} (\xi_4) \widehat{\phi_4}(\xi_4)d\xi_4 dt\Bigg|.
		 \end{equation*}
		 Then we can apply Plancherel and H\"{o}lder's inequality to estimate (\ref{term1}). 
		 
		 In the following discussions, we will erase the conjugate and use the above Coifman-Meyer multiplier theorem implicitly.
		 
		\begin{proof}[Proof of Proposition \ref{d=3}]
			For $\textbf{Term}_1$, we can assume $N_2\ge N_3\ge N_4$, $N_2\gtrsim N$. Furthermore, since $\sum_{j=1}^{4}\xi_j=0$, we have $N_2\gtrsim N_1$. 
			
			Next, we can break the interactions into several cases, depending on which frequency is comparable to $N_2$.
			
			\vspace{7pt}
			$\textbf{Term}_1$, \textbf{Case 1:} $N_1\sim N_2\gtrsim N\gg N_3\ge N_4.$
			
			Note that, by the mean value theorem, 
			\begin{equation*}
				\Bigg|1-\frac{m(\xi_1)}{m(\xi_2)m(\xi_3)m(\xi_4)}\Bigg|=\Bigg|\frac{m(\xi_2)-m(\xi_1)}{m(\xi_2)}\Bigg|\lesssim \frac{|\nabla m(\xi_2)\cdot(\xi_3+\xi_4)|}{m(\xi_2)}\lesssim \frac{N_3}{N_2}.
			\end{equation*}
			Applying the bilinear estimate (\ref{e}) to $(\phi_1, \phi_3)$ and $(\phi_2, \phi_4)$, one can derive
			\begin{equation*}
				(\ref{term1})\lesssim 	\frac{N_3}{N_2}\|\phi_1 \phi_3\|_{L_{x,t}^{2}([0,T]\times\R^2)}\|\phi_2 \phi_4\|_{L_{x,t}^{2}([0,T]\times\R^2)}
			\end{equation*}
			\begin{equation*}
				\lesssim \frac{N_3}{N_2}\cdot \frac{N_3}{N_1^{\frac{\alpha-1}{2}}}\cdot \frac{N_4}{N_2^{\frac{\alpha-1}{2}}}\cdot N_1^{\frac{\alpha}{2}}N_2^{-\frac{\alpha}{2}}N_3^{-\frac{\alpha}{2}}N_4^{-\frac{\alpha}{2}}\|\phi_1\|_{X^{-\frac{\alpha}{2}, \frac{1}{2}^{+}}}\||D|^{\frac{\alpha}{2}}\phi_2\|_{X^{0, \frac{1}{2}^{+}}}\||D|^{\frac{\alpha}{2}}\phi_3\|_{X^{0, \frac{1}{2}^{+}}}\||D|^{\frac{\alpha}{2}}\phi_4\|_{X^{0, \frac{1}{2}^{+}}}.
			\end{equation*}
			We can easily verify that
			\begin{equation*}
				\frac{N_3}{N_2}\cdot \frac{N_3}{N_1^{\frac{\alpha-1}{2}}}\cdot \frac{N_4}{N_2^{\frac{\alpha-1}{2}}}\cdot N_1^{\frac{\alpha}{2}}N_2^{-\frac{\alpha}{2}}N_3^{-\frac{\alpha}{2}}N_4^{-\frac{\alpha}{2}}\lesssim N^{-(2\alpha-3)^{-}}N_2^{0^{-}}N_3^{0\pm}N_4^{0\pm}.
			\end{equation*}
			Note that the positivity of the exponent on $N_4$, namely $1-\frac{\alpha}{2}$, is crucial here, especially in the regime $N_4 \leq 1$. This is precisely the feature that distinguishes the three-dimensional case from the two-dimensional one. For a more detailed discussion, we refer to Section 5.
			
			\vspace{7pt}
			$\textbf{Term}_1$, \textbf{Case 2:} $N_1\sim N_2\ge N_3\gtrsim N\gg 1.$
			
			Similarly, it suffices to control
			\begin{equation}\label{d}
				\frac{1}{m(N_3)m(N_4)}\cdot\frac{N_3}{N_1^{\frac{\alpha-1}{2}}}\cdot \frac{N_4}{N_2^{\frac{\alpha-1}{2}}}\cdot N_1^{\frac{\alpha}{2}}N_2^{-\frac{\alpha}{2}}N_3^{-\frac{\alpha}{2}}N_4^{-\frac{\alpha}{2}}.
			\end{equation}
			For simplicity, we just check the case when $N_3, N_4\gtrsim N$. Substituting $m(N_i)\sim\left(\frac{N}{N_i}\right)^{\frac{\alpha}{2}-s}$, $i=3,4$, into (\ref{d}), one can obtain
			
			\begin{equation*}
				(\ref{d})\sim N^{-(\alpha-2s)}\cdot \frac{N_3^{1-s}N_4^{1-s}}{N_1^{\alpha-1}}\lesssim N^{-(2\alpha-3)^{-}}N_1^{0^{-}},
			\end{equation*}
			where $s>\frac{1}{2}(3-\alpha)$ ensures $\alpha-1>2(1-s)$.
			
			\vspace{7pt}
			$\textbf{Term}_1$, \textbf{Case 3:} $N_2\sim N_3\gtrsim N\gg 1 \; \&\; N_1\ge1.$
			
			We can apply the bilinear estimate (\ref{e}) to $(\phi_1, \phi_2)$ and $(\phi_3, \phi_4)$, which reduces to 
			\begin{equation}\label{ddd}
				\frac{m(N_1)}{m(N_2)m(N_3)m(N_4)}\cdot \frac{N_1}{N_2^{\frac{\alpha-1}{2}}}\cdot \frac{N_4}{N_3^{\frac{\alpha-1}{2}}}\cdot N_1^{\frac{\alpha}{2}} N_2^{-\frac{\alpha}{2}}N_3^{-\frac{\alpha}{2}}N_4^{-\frac{\alpha}{2}}.
			\end{equation}
			For convenience, we only verify the worst case $N_1, N_4\gtrsim N$. Note that the function $m(x)x^{1+\frac{\alpha}{2}}$ is increasing; then we have
			\begin{equation}\label{ff}
				m(N_1)N_1^{1+\frac{\alpha}{2}}\lesssim m(N_2)N_2^{1+\frac{\alpha}{2}}.
			\end{equation}
			Plugging (\ref{ff}) and $m(N_i)\sim\left(\frac{N}{N_i}\right)^{\frac{\alpha}{2}-s}$, $i=2,3,4$, into (\ref{ddd}), we can derive 
			
			\begin{equation*}
				(\ref{ddd})\lesssim \frac{m(N_2)N_2^{1+\frac{\alpha}{2}}N_4^{1-\frac{\alpha}{2}}}{m(N_2)m(N_3)m(N_4)N_2^{2\alpha-1}}\sim \frac{N_4^{1-s}}{N^{\alpha-2s}N_2^{\alpha+s-2}}\lesssim N^{-(2\alpha-3)^{-}}N_2^{0^{-}},
			\end{equation*}
			where $s>\frac{1}{2}(3-\alpha)$ ensures $\alpha+s-2>1-s$.
			
			\vspace{7pt}
			$\textbf{Term}_1$, \textbf{Case 4:} $N_2\sim N_3\gtrsim N\gg 1 \; \&\; N_1\le 1.$
			
			Similarly, one can reduce to 
			\begin{equation}\label{ggg}
				\frac{1}{m(N_2)m(N_3)m(N_4)}\cdot \frac{N_1}{N_2^{\frac{\alpha-1}{2}}}\cdot \frac{N_4}{N_3^{\frac{\alpha-1}{2}}}\cdot N_2^{-\frac{\alpha}{2}}N_3^{-\frac{\alpha}{2}}N_4^{-\frac{\alpha}{2}}.
			\end{equation}
			Direct calculation shows that 
			\begin{equation*}
				(\ref{ggg})\lesssim \frac{N_1^{0^{+}}N_4^{1-\frac{\alpha}{2}}}{m(N_4)N^{\alpha-2s} N_2^{\alpha+2s-1}}\lesssim N^{-\frac{1}{2}(5\alpha-4)^{-}}N_1^{0^+}N_4^{0\pm}N_2^{0^-},
			\end{equation*}
			and $\frac{1}{2}(5\alpha-4)>2\alpha-3$.
			
			\vspace{10pt}
			For $\textbf{Term}_2$, we need the following lemma.
			\begin{lemma}\label{1234}
				\begin{equation*}
					\|I(\phi_1\phi_2\phi_3)\|_{L_{t,x}^{2}([0,T]\times\R^3)}\lesssim N_1^{0\pm}N_2^{0\pm}N_3^{0\pm}\prod_{i=1}^{3}\||D|^{\frac{\alpha}{2}}I\phi_i\|_{X^{0,\frac{1}{2}^{+}}}.
				\end{equation*}
			\end{lemma}
			We postpone the proof of Lemma \ref{1234} to the end of this section. 
			
			\vspace{7pt}
			$\textbf{Term}_2$, \textbf{Case 1:} $N_4\sim N_{123}\gtrsim N\gg 1\ge N_5\ge N_6.$
			
			Applying H\"{o}lder's inequality, Lemma \ref{1234} and the Strichartz estimate (\ref{Strichartz3}), we can derive
			\begin{equation*}
				(\ref{term2})\lesssim \|P_{N_{123}}I(\phi_1 \phi_2 \phi_3)\|_{L_{t,x}^{2}([0,T]\times\R^3)}\|I\phi_4\|_{L_{x,t}^{\frac{10}{3}}([0,T]\times\R^3)}\|I\phi_5\|_{L_{x,t}^{10}([0,T]\times\R^3)}\|I\phi_6\|_{L_{x,t}^{10}([0,T]\times\R^3)}
			\end{equation*}
			\begin{equation*}
				\lesssim N_1^{0\pm}N_2^{0\pm}N_3^{0\pm}N_4^{-\frac{1}{5}(4\alpha-3)}N_5^{\frac{1}{5}(6-3\alpha)}N_6^{\frac{1}{5}(6-3\alpha)}\prod_{i=1}^{6}\||D|^{\frac{\alpha}{2}}I\phi_i\|_{X^{0,\frac{1}{2}^{+}}},
			\end{equation*}
			Since $\alpha<2$, we see $\frac{1}{5}(4\alpha-3)>2\alpha-3$ and $\frac{1}{5}(6-3\alpha)>0$.
			
			\vspace{7pt}
			$\textbf{Term}_2$, \textbf{Case 2:} $N_4\sim N_{123}\gtrsim N\gg 1\ge N_6\;\&\; N_5\ge 1.$
			
			Similarly, it suffices to bound
			\begin{equation}\label{h}
				\frac{N_1^{0\pm}N_2^{0\pm}N_3^{0\pm}}{m(N_5)}\cdot N_4^{-\frac{1}{5}(4\alpha-3)}N_5^{\frac{1}{5}(6-3\alpha)}N_6^{\frac{1}{5}(6-3\alpha)}.
			\end{equation}
			One can check that
			\begin{equation*}
				(\ref{h})\lesssim N^{-\frac{1}{5}(7\alpha-9)^{-}}N_1^{0\pm}N_2^{0\pm}N_3^{0\pm}N_4^{0^{-}}N_6^{0^{+}},
			\end{equation*}
			and $\frac{1}{5}(7\alpha-9)>2\alpha-3$.
			
			\vspace{7pt}
			$\textbf{Term}_2$, \textbf{Case 3:} $N_4\sim N_{123}\gtrsim N\gg 1\ge N_6\;\&\; N_5\ge N_6\ge 1.$
			
			With the condition $s>\frac{1}{2}(3-\alpha)$, we can verify
			\begin{equation*}
				\frac{N_1^{0\pm}N_2^{0\pm}N_3^{0\pm}}{m(N_5)m(N_6)}\cdot N_4^{-\frac{1}{5}(4\alpha-3)}N_5^{\frac{1}{5}(6-3\alpha)}N_6^{\frac{1}{5}(6-3\alpha)}\lesssim N^{-(2\alpha-3)^{-}}N_1^{0\pm}N_2^{0\pm}N_3^{0\pm}N_4^{0^{-}}.
			\end{equation*}
			
			\vspace{7pt}
			$\textbf{Term}_2$, \textbf{Case 4:} $N_4\sim N_5\gtrsim N\gg 1\gtrsim N_{123}.$
			
			To gain extra decay on $N_{123}$, we can slightly change our H\"{o}lder's inequality. In fact, we apply it with the factors in $L_t^{2}L_x^{2^+},\; L_{x,t}^{\frac{10}{3}},\; L_{x,t}^{10},\; L_{t}^{10}L_x^{10^{-}}$ and use Bernstein's inequality, which leads to
			
			\begin{equation*}
				\frac{1}{m(N_4)m(N_5)m(N_6)}\|P_{N_{123}}I(\phi_1 \phi_2 \phi_3)\|_{L_{t}^{2}L_x^{2^{+}}([0,T]\times\R^3)}\|I\phi_4\|_{L_{x,t}^{\frac{10}{3}}([0,T]\times\R^3)}\|I\phi_5\|_{L_{x,t}^{10}([0,T]\times\R^3)}\|I\phi_6\|_{L_{t}^{10}L_{x}^{10^-}([0,T]\times\R^3)}
			\end{equation*}
			\begin{equation*}
				\lesssim \frac{N_{123}^{0^+}N_1^{0\pm}N_2^{0\pm}N_3^{0\pm}}{m(N_4)m(N_5)m(N_6)}\cdot N_4^{-\frac{1}{5}(4\alpha-3)}N_5^{\frac{1}{5}(6-3\alpha)}N_6^{\frac{1}{5}(6-3\alpha)^{-}}\prod_{i=1}^{6}\||D|^{\frac{\alpha}{2}}I\phi_i\|_{X^{0,\frac{1}{2}^{+}}}.
			\end{equation*}
			For simplicity, we only check the worst case, i.e., $N_4, N_5, N_6\gtrsim N$. Substituting $m(N_i)\sim\left(\frac{N}{N_i}\right)^{\frac{\alpha}{2}-s}$, $i=4,5,6$, we can obtain
			\begin{equation*}
				\frac{N_{123}^{0^+}N_1^{0\pm}N_2^{0\pm}N_3^{0\pm}}{m(N_4)m(N_5)m(N_6)}\cdot N_4^{-\frac{1}{5}(4\alpha-3)}N_5^{\frac{1}{5}(6-3\alpha)}N_6^{\frac{1}{5}(6-3\alpha)^{-}}\lesssim \frac{N_4^{\alpha-2s}N_6^{\frac{\alpha}{2}-s}}{N^{\frac{3\alpha}{2}-3s}}\cdot  N_1^{0\pm}N_2^{0\pm}N_3^{0\pm} N_{123}^{0^+}N_4^{-\frac{1}{5}(7\alpha-9)}N_6^{\frac{1}{5}(6-3\alpha)^{-}}
			\end{equation*}
			\begin{equation*}
				\lesssim N^{-(2\alpha-3)^{-}}N_1^{0\pm}N_2^{0\pm}N_3^{0\pm}N_{123}^{0^+}N_4^{0^{-}},
			\end{equation*}
			where the condition $s>1-\frac{\alpha}{6}$ ensures 
			\begin{equation*}
				\frac{1}{5}(7\alpha-9)>(\alpha-2s)+(\frac{\alpha}{2}-s)+\frac{1}{5}(6-3\alpha).
			\end{equation*}
			
			\vspace{7pt}
			$\textbf{Term}_2$, \textbf{Case 5:} $N_4\sim N_5\gtrsim N \;\&\; N_{123}\gg 1.$
			
			This case can be totally reduced to $\textbf{Term}_2$, \textbf{Case 4}.
		\end{proof}
		\begin{proof}[Proof of Lemma \ref{1234}]
			From Lemma \ref{interpolation}, we can assume $N=1$. Recall that we say $\phi$ has high frequency if $\supp\; \widehat{\phi}\subset \lbrace|\xi|\gtrsim 1\rbrace$; $\phi$ has low frequency if $\supp\; \widehat{\phi}\subset \lbrace|\xi|\lesssim 1\rbrace$.
			
			By symmetry and the fractional Leibniz rule, it remains to prove 
			\begin{equation*}
				\|I(\phi_1)\cdot\phi_2\cdot \phi_3\|_{L_{t,x}^{2}([0,T]\times\R^3)}\lesssim N_1^{0\pm}N_2^{0\pm}N_3^{0\pm}\prod_{i=1}^{3}\||D|^{\frac{\alpha}{2}}I\phi_i\|_{X^{0,\frac{1}{2}^{+}}}.
			\end{equation*}
			
			Although there are different types of frequency interactions, we can have a relatively uniform way to deal with them. In fact, we can apply H\"{o}lder in time and the Strichartz estimate (\ref{Strichartz3}) to derive
			\begin{equation*}
				\|I\phi_1\|_{L_{t,x}^6([0,T]\times\R^3)}\lesssim_T 
				\begin{cases}
					N_1^{\gamma_{6,6}}\|I\phi_1\|_{X^{0,\frac{1}{2}^+}}\lesssim N_1^{-\left(\frac{2\alpha}{3}-1\right)}\||D|^{\frac{\alpha}{2}}I\phi_1\|_{X^{0,\frac{1}{2}^{+}}},\; \text{if }\; \phi_1\; \text{has high frequency}, \\[4pt]
					\|\phi_1\|_{L_t^{\infty}L_x^{6}([0,T]\times\R^3)}\lesssim N_1^{1-\frac{\alpha}{2}}\||D|^{\frac{\alpha}{2}}I\phi_1\|_{X^{0,\frac{1}{2}^{+}}},\; \text{if }\; \phi_1\; \text{has low frequency},
				\end{cases}
			\end{equation*}
			and for $i=2,3$,
			\begin{equation*}
				\|\phi_i\|_{L_{t,x}^6([0,T]\times\R^3)}\lesssim_T 
				\begin{cases}
					N_i^{\gamma_{6,6}}\|\phi_1\|_{X^{0,\frac{1}{2}^+}}\lesssim N_i^{-\left(s-1+\frac{\alpha}{6}\right)}\||D|^{\frac{\alpha}{2}}I\phi_i\|_{X^{0,\frac{1}{2}^{+}}},\; \text{if }\; \phi_i\; \text{has high frequency}, \\[4pt]
					\|\phi_i\|_{L_t^{\infty}L_x^{6}([0,T]\times\R^3)}\lesssim N_i^{1-\frac{\alpha}{2}}\||D|^{\frac{\alpha}{2}}I\phi_i\|_{X^{0,\frac{1}{2}^{+}}},\; \text{if }\; \phi_i\; \text{has low frequency}.
				\end{cases}
			\end{equation*}
			Note that $s-1+\frac{\alpha}{6}>0$, $\frac{2\alpha}{3}-1>0$, and $1-\frac{\alpha}{2}>0$; thus we conclude Lemma \ref{1234}.
		\end{proof}
		\begin{rem}\label{T}
			Investigating the proof of Proposition \ref{d=3}, the only place where we used H\"{o}lder's inequality in time is the proof of Lemma \ref{1234}. 
			
			For example, we have used 
			\begin{equation*}
				\|\phi_i\|_{L_{t,x}^6([0,T]\times\R^3)}\lesssim_T \|\phi_i\|_{L_t^{\infty}L_x^{6}([0,T]\times\R^3)}.
			\end{equation*}
			This is allowed since $T$ comes from local well-posedness theory in Proposition \ref{LWP} and we can additionally require $T\lesssim 1$. However, when dealing with the radial case, we will adopt a different framework to derive the ``almost conservation law" and $T$ will come from Proposition \ref{priori}, which is not necessarily $\lesssim 1$.
			
			To solve this technical issue, we should carefully modify the whole proof and have a more delicate choice of relevant parameters. See Proposition \ref{new almost} for more details.
		\end{rem}
	
		\vspace{15pt}
		\section{GWP of fractional Schr\"{o}dinger equation in $\R^2$}
		\begin{prop}\label{d=2}
			For $d=2$, $s>\frac{2}{3}-\frac{\alpha}{6}$, we suppose $u$ is the solution in Proposition \ref{LWP}; then the following estimate holds:
			\begin{equation*}
				E(Iu)(t)-E(Iu)(0)=O(N^{-\frac{1}{2}(-\alpha^2+5\alpha-4)^{-}}), \quad \forall t\in [0,T].
			\end{equation*}
		\end{prop}
		
		\vspace{7pt}
		Similar to the proof of Theorem \ref{main d=3}, one can prove Theorem \ref{main d=2} by using Proposition \ref{d=2} directly.
		
		\begin{proof}[Proof of Theorem \ref{main d=2}]
			A direct calculation shows that 
			\begin{equation*}
				\|Iu_0^{\lambda}\|_{\dot{H}^{\frac{\alpha}{2}}(\R^2)}\lesssim N^{\frac{\alpha}{2}-s}\|u_0^{\lambda}\|_{\dot{H}^{s}(\R^2)}=N^{\frac{\alpha}{2}-s} \lambda^{-\left(s+\frac{\alpha}{2}-1\right)}\|u_0\|_{\dot{H}^{s}(\R^2)}.
			\end{equation*}
			We choose the parameter $\lambda$ such that
			\begin{equation*}
				\lambda\sim N^{\frac{\frac{\alpha}{2}-s}{s+\frac{\alpha}{2}-1}}.
			\end{equation*}
			
			Now, we can apply Proposition \ref{d=2} to the scaled initial data $Iu_0^{\lambda}$ and repeat it $N^{\frac{1}{2}(-\alpha^2+5\alpha-4)^-}$ times as before.
			
			Then, given any $T_0\gg 1$, we can choose the parameter $N$ satisfying
			\begin{equation}
				T_0\sim \frac{N^{\frac{1}{2}(-\alpha^2+5\alpha-4)^-}}{\lambda^{\alpha}}\sim N^{\frac{1}{2}(-\alpha^2+5\alpha-4)-\frac{\alpha(\frac{\alpha}{2}-s)}{s+\frac{\alpha}{2}-1}^-},
			\end{equation}
			when the exponent of $N$ is positive. One can check that this is equivalent to 
			\begin{equation*}
				\frac{\alpha}{2}>s>\frac{\alpha}{2}-\frac{(4-\alpha)(\alpha-1)^2}{-\alpha^2+7\alpha-4}.
			\end{equation*}
			Note also that
			\begin{equation*}
				E(I_N u^{\lambda})(\lambda^{\alpha}t)=\lambda^{2-2\alpha}E(I_{\lambda N}u)(t).
			\end{equation*}
			Then we can still obtain the following polynomial-in-time control:
			\begin{equation*}
				\|u(T_0)\|
				\lesssim T_0^{\frac{2(\alpha-1)(\frac{\alpha}{2}-s)^+}{(-\alpha^2+7\alpha-4)(s-r_2(\alpha))}}.
			\end{equation*}
		\end{proof}
	In the two-dimensional setting, the analysis of $\textbf{Term}_1$ is more delicate than in the three-dimensional case, and one can no longer rely solely on bilinear estimates. In particular, when $\phi_4$ is localized at very low frequencies, the bilinear estimate yields only
	\begin{equation*}
		N_4^{\frac{d-1}{2}-\frac{\alpha}{2}} = N_4^{\frac{1-\alpha}{2}},
	\end{equation*}
	which is strictly negative for all $\alpha>1$ and therefore provides insufficient control.
	
	To overcome this difficulty, it is necessary to combine with the Strichartz estimates and H\"{o}lder in time. However, this combination comes at the expense of a weaker decay rate.
	
	\begin{proof}[Proof of Proposition \ref{d=2}]
		
		For $\textbf{Term}_1$, it suffices to consider the worst-case scenario $N_4 \le 1$, since the remaining cases can be treated similarly or more easily (for instance, by only using the bilinear estimate).
		
		\vspace{7pt}
		$\textbf{Term}_1$, \textbf{Case 1:} $N_1\sim N_2\gtrsim N\gg 1\ge N_3\ge N_4.$
		
		Applying the bilinear estimate (\ref{e}) to $(\phi_1, \phi_3)$ and the Strichartz estimate (\ref{Strichartz3}) to $\phi_2, \phi_4$, we obtain
		\begin{equation*}
			(\ref{term1})\lesssim_T \frac{N_3}{N_1}\|\phi_1 \phi_3\|_{L_{x,t}^{2}([0,T]\times\R^2)}\|\phi_2\|_{L_t^{\infty}L_x^{2}([0,T]\times\R^2)}\|\phi_4\|_{L_t^{\infty}L_x^{\infty}([0,T]\times\R^2)}
		\end{equation*}
		\begin{equation*}
			\lesssim \frac{N_3}{N_1}\cdot \frac{N_3^{\frac{1}{2}}}{N_1^{\frac{\alpha-1}{2}}}\cdot \frac{N_1^{\frac{\alpha}{2}}}{N_3^{\frac{\alpha}{2}}}\cdot N_2^{-\frac{\alpha}{2}}\cdot N_4^{1-\frac{\alpha}{2}}\|\phi_1\|_{X^{-\frac{\alpha}{2}, \frac{1}{2}^{+}}}\||D|^{\frac{\alpha}{2}}\phi_2\|_{X^{0, \frac{1}{2}^{+}}}\||D|^{\frac{\alpha}{2}}\phi_3\|_{X^{0, \frac{1}{2}^{+}}}\||D|^{\frac{\alpha}{2}}\phi_4\|_{X^{0, \frac{1}{2}^{+}}}.
		\end{equation*}
		One can check that 
		\begin{equation*}
			\frac{N_3}{N_1}\cdot \frac{N_3^{\frac{1}{2}}}{N_1^{\frac{\alpha-1}{2}}}\cdot \frac{N_1^{\frac{\alpha}{2}}}{N_3^{\frac{\alpha}{2}}}\cdot N_2^{-\frac{\alpha}{2}}\cdot N_4^{1-\frac{\alpha}{2}}\lesssim N^{-\frac{\alpha+1}{2}^{-}}N_1^{0^{-}}N_3^{0^{+}}N_4^{0^{+}},
		\end{equation*}
		and $\frac{1}{2}(\alpha+1)>\frac{1}{2}(-\alpha^2+5\alpha-4)$, when $\alpha\in (1,2)$.
		
		\vspace{7pt}
		$\textbf{Term}_1$, \textbf{Case 2:} $N_1\sim N_2\gtrsim N\gg N_3\ge 1\ge N_4.$
		
		From bilinear estimate \eqref{e} and Strichartz estimate \eqref{Strichartz3}, we obtain
		\begin{align*}
			\|\phi_2\phi_4\|_{L_{x,t}^{2}([0,T]\times \R^2)}
			&\lesssim \frac{N_4^{\frac{1}{2}}}{N_2^{\frac{\alpha-1}{2}}}
			\, N_2^{-\frac{\alpha}{2}} N_4^{-\frac{\alpha}{2}}
			\||D|^{\frac{\alpha}{2}}\phi_2\|_{X^{0, \frac{1}{2}^{+}}}
			\||D|^{\frac{\alpha}{2}}\phi_4\|_{X^{0, \frac{1}{2}^{+}}}, \\
			\|\phi_2\phi_4\|_{L_{x,t}^{2}([0,T]\times \R^2)}
			&\lesssim N_2^{-\frac{\alpha}{2}} N_4^{1-\frac{\alpha}{2}}
			\||D|^{\frac{\alpha}{2}}\phi_2\|_{X^{0, \frac{1}{2}^{+}}}
			\||D|^{\frac{\alpha}{2}}\phi_4\|_{X^{0, \frac{1}{2}^{+}}}.
		\end{align*}
		Interpolating between these two bounds yields
		\begin{equation*}
			\|\phi_2\phi_4\|_{L_{x,t}^{2}([0,T]\times \R^2)}
			\lesssim N_2^{-\frac{1}{2}(-\alpha^2+4\alpha-2)^{-}}
			\, N_4^{0^{+}}
			\||D|^{\frac{\alpha}{2}}\phi_2\|_{X^{0, \frac{1}{2}^{+}}}
			\||D|^{\frac{\alpha}{2}}\phi_4\|_{X^{0, \frac{1}{2}^{+}}}.
		\end{equation*}
		Thus one can similarly conclude that 
		\begin{equation*}
			\frac{N_3}{N_1}\cdot \frac{N_3^{\frac{1}{2}}}{N_1^{\frac{\alpha-1}{2}}}\cdot \frac{N_1^{\frac{\alpha}{2}}}{N_3^{\frac{\alpha}{2}}}\cdot N_2^{-\frac{1}{2}(-\alpha^2+4\alpha-2)^{-}}
			\cdot N_4^{0^{+}}\lesssim N^{-\frac{1}{2}(-\alpha^2+5\alpha-4)^{-}}N_1^{0^{-}}N_4^{0^{+}}.
		\end{equation*}
		
		\vspace{7pt}
		$\textbf{Term}_1$, \textbf{Case 3:} $N_1\sim N_2\ge N_3\gtrsim N\gg 1\ge N_4.$
		
		Similarly, we can reduce to 
		\begin{equation*}
			\frac{1}{m(N_3)}\cdot \frac{N_3^{\frac{1}{2}}}{N_1^{\frac{\alpha-1}{2}}}\cdot \frac{N_1^{\frac{\alpha}{2}}}{N_3^{\frac{\alpha}{2}}}\cdot N_2^{-\frac{1}{2}(-\alpha^2+4\alpha-2)^{-}}
			\cdot N_4^{0^{+}}.
		\end{equation*}
		Substituting $m(N_3)\sim \left(\frac{N}{N_3}\right)^{\frac{\alpha}{2}-s}$ and using the condition $s>\frac{2-\alpha}{2}$, one can obtain the same decay as in $\textbf{Term}_1$, \textbf{Case 2}.
		
		\vspace{7pt}
		$\textbf{Term}_1$, \textbf{Case 4:} $N_2\sim N_3\gtrsim N\gg 1\ge \max\lbrace N_1, N_4\rbrace.$
		
		Applying the bilinear estimate (\ref{e}) to $(\phi_1, \phi_2)$ and the Strichartz estimate (\ref{Strichartz3}) to $\phi_3, \phi_4$, it remains to control
		\begin{equation*}
			\frac{1}{m(N_2)m(N_3)}\cdot \frac{N_1^{\frac{1}{2}}}{N_2^{\frac{\alpha-1}{2}}}\cdot N_2^{-\frac{\alpha}{2}}\cdot N_3^{-\frac{\alpha}{2}}\cdot N_4^{1-\frac{\alpha}{2}},
		\end{equation*}
		which actually leads to a much better bound $N^{-\frac{1}{2}(3\alpha-1)^{-}}N_2^{0^-}N_1^{0^+}N_4^{0^+}$.
		
		\vspace{7pt}
		$\textbf{Term}_1$, \textbf{Case 5:} $N_2\sim N_3\gtrsim N\gg 1\ge N_4\;  \&\;  N_1\ge 1.$
		
		We can follow the same way of control as in $\textbf{Term}_1$, \textbf{Case 3}, and derive 
		\begin{equation*}
			\frac{m(N_1)}{m(N_2)m(N_3)}\cdot \frac{N_3^{\frac{1}{2}}}{N_1^{\frac{\alpha-1}{2}}}\cdot \frac{N_1^{\frac{\alpha}{2}}}{N_3^{\frac{\alpha}{2}}}\cdot N_2^{-\frac{1}{2}(-\alpha^2+4\alpha-2)^{-}}
			\cdot N_4^{0^{+}}.
		\end{equation*}
		Simple calculations show the desired bound $N^{-\frac{1}{2}(-\alpha^2+5\alpha-4)^{-}}N_2^{0^{-}}
		N_4^{0^{+}}$.
		
		\vspace{10pt}
		For $\textbf{Term}_2$, we need the following lemma:
		\begin{lemma}\label{123}
			If $N_{123}\gg 1$, then we have 
			\begin{equation}\label{gg}
				\|P_{N_{123}}I(\phi_1\phi_2\phi_3)\|_{L_{t,x}^{2}([0,T]\times\R^2)}\lesssim N_1^{0\pm}N_2^{0\pm}N_3^{0\pm}\prod_{i=1}^{3}\||D|^{\frac{\alpha}{2}}I\phi_i\|_{X^{0,\frac{1}{2}^{+}}}.
			\end{equation}
		\end{lemma}
		We still postpone the proof of Lemma \ref{123}.  
		
		\vspace{7pt}
		$\textbf{Term}_2$, \textbf{Case 1:} $N_{123}\sim N_4\gtrsim N\gg 1\ge N_5\ge N_6.$
		
		Applying Lemma \ref{123} and the Strichartz estimate (\ref{Strichartz3}), one can obtain
		\begin{equation*}
			(\ref{term2})\lesssim_T \|P_{N_{123}}I(\phi_1\phi_2\phi_3)\|_{L_{t,x}^{2}([0,T]\times\R^2)}\|I\phi_4\|_{L_t^{\infty}L_x^{2}([0,T]\times\R^2)}\|I\phi_5\|_{L_{t,x}^{\infty}([0,T]\times\R^2)}\|I\phi_6\|_{L_{t,x}^{\infty}([0,T]\times\R^2)}
		\end{equation*}
		\begin{equation*}
			\lesssim N_1^{0\pm}N_2^{0\pm}N_3^{0\pm}N_4^{-\frac{\alpha}{2}}N_5^{1-\frac{\alpha}{2}}N_6^{1-\frac{\alpha}{2}} \prod_{i=1}^{6}\||D|^{\frac{\alpha}{2}}I\phi_i\|_{X^{0,\frac{1}{2}^{+}}}.
		\end{equation*}
		One can check that $\frac{1}{2}\alpha>\frac{1}{2}(-\alpha^2+5\alpha-4)$ when $\alpha\in (1,2)$.
		
		\vspace{7pt}
		$\textbf{Term}_2$, \textbf{Case 2:} $N_{123}\sim N_4\gtrsim N\gg 1\ge N_6 \; \&\; N_5\ge1.$
		
		Applying the bilinear estimate (\ref{e}), we can derive
		\begin{equation*}
			(\ref{term2})\lesssim_T \frac{1}{m(N_5)} \|P_{N_{123}}I(\phi_1\phi_2\phi_3)\|_{L_{t,x}^{2}([0,T]\times\R^2)}\|I\phi_4I\phi_5\|_{L_{t,x}^{2}([0,T]\times\R^2)}\|I\phi_6\|_{L_{t,x}^{\infty}([0,T]\times\R^2)}
		\end{equation*}
		\begin{equation*}
			\lesssim \frac{N_1^{0\pm}N_2^{0\pm}N_3^{0\pm}}{m(N_5)}\cdot \frac{N_5^{\frac{1}{2}}}{N_4^{\frac{\alpha-1}{2}}}N_4^{-\frac{\alpha}{2}}N_5^{-\frac{\alpha}{2}}N_6^{1-\frac{\alpha}{2}}\prod_{i=1}^{6}\||D|^{\frac{\alpha}{2}}I\phi_i\|_{X^{0,\frac{1}{2}^{+}}}.
		\end{equation*}
		One can check that
		\begin{equation*}
			\frac{N_1^{0\pm}N_2^{0\pm}N_3^{0\pm}}{m(N_5)}\cdot \frac{N_5^{\frac{1}{2}}}{N_4^{\frac{\alpha-1}{2}}}N_4^{-\frac{\alpha}{2}}N_5^{-\frac{\alpha}{2}}N_6^{1-\frac{\alpha}{2}}\lesssim N^{-\frac{1}{2}(2\alpha-1)^{-}}N_1^{0\pm}N_2^{0\pm}N_3^{0\pm}N_4^{0^{-}}N_6^{0^+},
		\end{equation*}
		and $\frac{1}{2}(2\alpha-1)>\frac{1}{2}(-\alpha^2+5\alpha-4)$.
		
		\vspace{7pt}
		$\textbf{Term}_2$, \textbf{Case 3:} $N_{123}\sim N_4\gtrsim N \; \&\; N_5\ge N_6\ge1.$
		
		Applying the same control as in $\textbf{Term}_2$, \textbf{Case 2}, we can obtain
		\begin{equation*}
			(\ref{term2})\lesssim \frac{N_1^{0\pm}N_2^{0\pm}N_3^{0\pm}}{m(N_5)m(N_6)}N_4^{-\frac{2\alpha-1}{2}}N_5^{-\frac{\alpha-1}{2}}N_6^{1-\frac{\alpha}{2}}\prod_{i=1}^{6}\||D|^{\frac{\alpha}{2}}I\phi_i\|_{X^{0,\frac{1}{2}^{+}}}.
		\end{equation*}
		From the condition $s>\frac{1}{2}(2-\alpha)$, we can verify 
		\begin{equation*}
			\frac{N_1^{0\pm}N_2^{0\pm}N_3^{0\pm}}{m(N_5)m(N_6)}N_4^{-\frac{2\alpha-1}{2}}N_5^{-\frac{\alpha-1}{2}}N_6^{1-\frac{\alpha}{2}}\lesssim N^{-\frac{1}{2}(-\alpha^2+5\alpha-4)^{-}}N_1^{0\pm}N_2^{0\pm}N_3^{0\pm}N_4^{0^{-}}.
		\end{equation*}
		
		\vspace{7pt}
		$\textbf{Term}_2$, \textbf{Case 4:} $N_4\sim N_5\gtrsim N \; \&\; N_{123}\gg1.$
		
		If $N_6\ll N$, the proof is similar to the one in $\textbf{Term}_2$, \textbf{Case 2}. If $N_6\gtrsim N$, it remains to control
		\begin{equation*}
			\frac{m(N_{123})}{m(N_4)m(N_5)m(N_6)}N_1^{0\pm}N_2^{0\pm}N_3^{0\pm}N_4^{-\frac{2\alpha-1}{2}}N_5^{-\frac{\alpha-1}{2}}N_6^{1-\frac{\alpha}{2}}.
		\end{equation*}
		Substituting $m(N_{123})\le 1$, $m(N_i)\sim \left(\frac{N}{N_i}\right)^{\frac{\alpha}{2}-s}$, $i=4,5,6$, it suffices to show
		\begin{equation*}
			N^{-3\left(\frac{\alpha}{2}-s\right)}N_4^{-\frac{1}{2}(\alpha+2s-1)}N_5^{-\left(s-\frac{1}{2}\right)}N_6^{-(s-1)}\lesssim N^{-(2\alpha-2)^{-}}N_4^{0^{-}},
		\end{equation*}
		which requires 
		\begin{equation*}
			\frac{1}{2}(\alpha+2s-1)+\left(s-\frac{1}{2}\right)+(s-1)=3\left(s-\left(\frac{2}{3}-\frac{\alpha}{6}\right)\right)>0.
		\end{equation*}
		One also has $2\alpha-2>\frac{1}{2}(-\alpha^2+5\alpha-4).$
		
		\vspace{7pt}
		$\textbf{Term}_2$, \textbf{Case 5:} $N_4\sim N_5\gtrsim N \gg1\gtrsim N_{123}.$
		
		Investigating the proof of Lemma \ref{123}, we just need to deal with the low-low-low interaction case, i.e., $\max\lbrace N_1, N_2, N_3\rbrace\lesssim 1$. The remaining cases can be directly reduced to $\textbf{Term}_2$, \textbf{Case 4}.
		
		For simplicity, we also assume $N_6\ge 1$ and the case $N_6\le 1$ can be obtained by replacing the following $L_t^{2^{+}}L_x^{\infty}$ with $L_{t,x}^{\infty}$.
		
		Applying H\"{o}lder's inequality and Bernstein's inequality, we can reduce to 
		\begin{equation*}
			\|P_{N_{123}}I(\phi_1\phi_2\phi_3)\|_{L_t^{\infty}L_x^{\infty}([0,T]\times\R^2)}
			\|I\phi_4\|_{L_t^{\infty}L_x^{2}([0,T]\times\R^2)}
			\|I\phi_5\|_{L_t^{\infty}L_x^{2}([0,T]\times\R^2)}
			\|I\phi_6\|_{L_t^{2^+}L_x^{\infty}([0,T]\times\R^2)}
		\end{equation*}
		\begin{equation*}
			\lesssim N_{123}^{0^{+}}N_4^{-\frac{\alpha}{2}}N_5^{-\frac{\alpha}{2}}N_6^{-\left(\alpha-1\right)^{-}}\|P_{N_{123}}I(\phi_1\phi_2\phi_3)\|_{L_t^{\infty}L_x^{\infty^{-}}([0,T]\times\R^2)}\prod_{i=4}^{6}\||D|^{\frac{\alpha}{2}}I\phi_i\|_{X^{0,\frac{1}{2}^{+}}}.
		\end{equation*}
		Note that, from the fractional Leibniz rule and the Strichartz estimate (\ref{Strichartz3}),
		\begin{equation*}
			\|P_{N_{123}}I(\phi_1\phi_2\phi_3)\|_{L_t^{\infty}L_x^{\infty^{-}}([0,T]\times\R^2)}\lesssim \prod_{i=1}^{3}\|\phi_i\|_{L_t^{\infty}L_x^{\infty^{-}}([0,T]\times\R^2)}
		\end{equation*}
		\begin{equation*}
			\lesssim N_1^{\left(1-\frac{\alpha}{2}\right)^{-}}N_2^{\left(1-\frac{\alpha}{2}\right)^{-}}N_3^{\left(1-\frac{\alpha}{2}\right)^{-}}\prod_{i=1}^{3}\||D|^{\frac{\alpha}{2}}I\phi_i\|_{X^{0,\frac{1}{2}^{+}}}.
		\end{equation*}
		One can also easily check that
		\begin{equation*}
			\frac{N_{123}^{0^{+}}N_1^{0^{+}}N_2^{0^{+}}N_3^{0^{+}}}{m(N_4)m(N_5)m(N_6)}N_4^{-\frac{\alpha}{2}}N_5^{-\frac{\alpha}{2}}N_6^{-\left(\alpha-1\right)^{-}}\lesssim N^{-\alpha^{-}} N_{123}^{0^{+}}N_1^{0^{+}}N_2^{0^{+}}N_3^{0^{+}}N_4^{0^{-}}.
		\end{equation*}
		Thus, we complete the estimate on $\textbf{Term}_2$ and Proposition \ref{d=2}.
	\end{proof}
\begin{proof}[Proof of Lemma \ref{123}]
	From interpolation in Lemma \ref{interpolation}, we can still assume $N=1$. We restrict our attention to several representative cases, as the remaining ones can be treated similarly.  
	
	\vspace{7pt}
	\textbf{Case 1}: $\phi_1, \phi_2, \phi_3$ all have low frequency.
	
	Since $N_{123}\gg 1$, the left of (\ref{gg}) is vanishing and there is nothing to prove.
	
	\vspace{7pt}
	For the rest of cases, it also suffices to show
	\begin{equation*}
		\|I\phi_1\cdot \phi_2\cdot \phi_3\|_{L_{t,x}^{2}([0,T]\times\R^2)}\lesssim N_1^{0\pm}N_2^{0\pm}N_3^{0\pm}\prod_{i=1}^{3}\||D|^{\frac{\alpha}{2}}I\phi_i\|_{X^{0,\frac{1}{2}^{+}}}.
	\end{equation*}
	
	\vspace{7pt}
	\textbf{Case 2}: $\phi_1, \phi_2, \phi_3$ all have high frequency.
	
	Applying H\"{o}lder's inequality and the Strichartz estimate (\ref{Strichartz3}), one can obtain 
	
	\begin{equation*}
		\|I\phi_1\cdot \phi_2\cdot \phi_3\|_{L_{t,x}^{2}([0,T]\times\R^2)}\lesssim \|I\phi_1\|_{L_{t,x}^{6}([0,T]\times\R^2)}\|\phi_2\|_{L_{t,x}^{6}([0,T]\times\R^2)}\|\phi_3\|_{L_{t,x}^{6}([0,T]\times\R^2)}
	\end{equation*}
	\begin{equation*}
		\lesssim N_1^{-\frac{2}{3}(\alpha-1)} N_2^{-\left(s-\frac{2}{3}+\frac{\alpha}{6}\right)}N_3^{-\left(s-\frac{2}{3}+\frac{\alpha}{6}\right)}\prod_{i=1}^{3}\||D|^{\frac{\alpha}{2}}I\phi_i\|_{X^{0,\frac{1}{2}^{+}}}.
	\end{equation*}
	Since $s>\frac{2}{3}-\frac{\alpha}{6}$, we can conclude this case.

	\vspace{7pt}
	\textbf{Case 3}: $\phi_1, \phi_2$ have high frequency and $\phi_3$ has low frequency.
	
	Similarly, one can derive
	\begin{equation*}
		\|I\phi_1\cdot \phi_2\cdot \phi_3\|_{L_{t,x}^{2}([0,T]\times\R^2)}\lesssim \|I\phi_1\|_{L_{t,x}^{4}([0,T]\times\R^2)}\|\phi_2\|_{L_{t,x}^{4}([0,T]\times\R^2)}\|\phi_3\|_{L_{t,x}^{\infty}([0,T]\times\R^2)}
	\end{equation*}
	\begin{equation*}
		\lesssim N_1^{-\left(\frac{3\alpha}{4}-\frac{1}{2}\right)}N_2^{-\left(s-\frac{1}{2}+\frac{\alpha}{4}\right)}N_3^{1-\frac{\alpha}{2}}\prod_{i=1}^{3}\||D|^{\frac{\alpha}{2}}I\phi_i\|_{X^{0,\frac{1}{2}^{+}}}.
	\end{equation*}
	Since $s>\frac{2}{3}-\frac{\alpha}{6}>\frac{1}{2}-\frac{\alpha}{4}$, we see 
	$N_1^{-\left(\frac{3\alpha}{4}-\frac{1}{2}\right)}N_2^{-\left(s-\frac{1}{2}+\frac{\alpha}{4}\right)}N_3^{1-\frac{\alpha}{2}}\lesssim N_1^{0^{-}}N_2^{0^-}N_3^{0^+}.$
\end{proof}

		\vspace{15pt}
		\section{Morawetz estimate}
		The key is to use Balakrishnan's formula for the fractional Laplacian $|D|^{\alpha}$  (see \cite{9,8}), for $\alpha\in (0,2)$:
		\begin{equation*}
			|D|^{\alpha}=\frac{\sin(\frac{\pi\alpha}{2})}{\pi}\int_{0}^{\infty}\lambda^{\frac{\alpha}{2}-1}(-\Delta)(-\Delta+\lambda)^{-1}d\lambda.
		\end{equation*}
	Next, we can modify the virial identity established in \cite{7} (originally for $u$) to the modified solution $Iu$, which states as follows:
		\begin{lemma}\label{virial}
			Suppose $u\in C([0,T];\mathcal{S}(\R^3))$ is a solution to (\ref{FNLS}), then for 
			\begin{equation*}
				M_{\varphi}(t):=2\Im m \int_{\R^3}\overline{Iu}\nabla{Iu}\cdot \nabla \varphi dx,
			\end{equation*}
			we have
			\begin{equation*}
				\frac{d M_{\varphi}}{dt}=\Re e\int_{0}^{\infty}\lambda^{\frac{\alpha}{2}}d\lambda \int_{\R^3}(4\partial_j\partial_k\varphi \partial_j\overline{Iu_\lambda}\partial_kIu_\lambda-\Delta^{2}\varphi |Iu_\lambda|^2 )dx
			\end{equation*}
			\begin{equation}\label{5}
				+\int_{\R^3}\Delta\varphi|Iu|^{4}dx +2\int_{\R^3}\nabla\varphi\cdot\lbrace (u)_{err}, Iu\rbrace dx,
			\end{equation}
			where $\lbrace \cdot, \cdot\rbrace$ is the momentum bracket defined by 
			\begin{equation*}
				\lbrace f, g\rbrace:=\Re e (f\nabla \overline{g}-g\nabla \overline{f}),
			\end{equation*}
			and 
			\begin{equation*}
				u_\lambda:=\sqrt{\frac{\sin(\frac{\pi\alpha}{2})}{\pi}}(\lambda-\Delta)^{-1}u, \quad (u)_{err}:=I(|u|^2u)-|Iu|^2 Iu.
			\end{equation*}
		\end{lemma}
\begin{proof}
	Recall that $Iu$ satisfies $i\partial_t Iu=|D|^{\alpha}Iu +I(|u|^2u)$. Then we have
	\begin{equation*}
		\frac{d M_{\varphi}}{dt}=2\Im m\int_{\R^3}i\overline{|D|^{\alpha}Iu}\nabla Iu\cdot \nabla\varphi dx-i\overline{Iu}\nabla |D|^{\alpha}Iu\cdot\nabla \varphi dx
	\end{equation*}
	\begin{equation*}
		+2 \Im m \int_{\R^3} i \overline{I(|u|^2 u)}\nabla Iu\cdot \nabla\varphi-i \overline{Iu}\nabla I(|u|^2u)\cdot\nabla \varphi dx=:J_1+J_2.
	\end{equation*}
	For $J_2$, we can rewrite it as 
	\begin{equation*}
		J_2= 2\Re e\int_{\R^3} |Iu|^2 \overline{Iu}\nabla Iu\cdot\nabla \varphi-\overline{Iu}\nabla(|Iu|^2 Iu)\cdot \nabla \varphi dx+	2\int_{\R^3} \nabla \varphi \cdot \lbrace (u)_{err}, Iu\rbrace dx
	\end{equation*}
	\begin{equation*}
		=\Re e\int_{\R^3}-\nabla(|Iu|^4)\cdot \nabla\varphi dx+2\int_{\R^3} \nabla \varphi \cdot \lbrace (u)_{err}, Iu\rbrace dx
	\end{equation*}
	\begin{equation*}
		=\int_{\R^3}\Delta\varphi|Iu|^{4}dx +2\int_{\R^3}\nabla\varphi\cdot\lbrace (u)_{err}, Iu\rbrace dx.
	\end{equation*}
	
	For $J_1$, we can proceed with an essentially identical calculation as in \cite{7} (see also \cite{10}). For completeness, we present them as follows, with $c:=\frac{\sin(\frac{\pi\alpha}{2})}{\pi}$.		
	\begin{equation*}
		J_1=2 \Im m\int_{\mathbb{R}^3} (i|D|^{\alpha} \overline{Iu}) \, \nabla Iu \cdot \nabla \varphi \, dx + 2 \Im m \int_{\mathbb{R}^3} \overline{Iu} \, \nabla \bigl(-i|D|^{\alpha} Iu\bigr) \cdot \nabla \varphi \, dx
	\end{equation*}
	\begin{equation*}
		= -2 \sqrt{c} \Re e\int_0^\infty \lambda^{\alpha/2 -1} d\lambda \int_{\mathbb{R}^3} \Delta \overline{Iu}_\lambda \, \nabla Iu \cdot \nabla \varphi \, dx
		+ 2\sqrt{c} \Re e\int_0^\infty \lambda^{\alpha/2 -1} d\lambda \int_{\mathbb{R}^3} \overline{Iu} \, \nabla \varphi \cdot \nabla \Delta Iu_\lambda \, dx
	\end{equation*}
	\begin{equation*}
		= 2\sqrt{c} \Re e \int_0^\infty \lambda^{\alpha/2 -1} d\lambda \int_{\mathbb{R}^3} \nabla \overline{Iu}_\lambda \cdot \nabla \bigl( \nabla Iu \cdot \nabla \varphi \bigr) dx
	\end{equation*}
	\begin{equation*}
		\quad - 2\sqrt{c} \Re e\int_0^\infty \lambda^{\alpha/2 -1} d\lambda \int_{\mathbb{R}^3} \Delta Iu_\lambda \, \nabla \overline{Iu} \cdot \nabla \varphi \, dx
		- 2\sqrt{c} \Re e\int_0^\infty \lambda^{\alpha/2 -1} d\lambda \int_{\mathbb{R}^3} \Delta Iu_\lambda \, \overline{Iu} \, \Delta \varphi \, dx .
	\end{equation*}
	For the inner integral in the first term, 
	\begin{equation*}
		\Re e\int_{\mathbb{R}^3} \nabla \overline{Iu}_\lambda \cdot \nabla \bigl( \nabla Iu \cdot \nabla \varphi \bigr) dx
		= \frac{1}{\sqrt{c}} \Re e\int_{\mathbb{R}^3} \nabla \overline{Iu}_\lambda \cdot \nabla \bigl( \nabla (-\Delta + \lambda) Iu_\lambda \cdot \nabla \varphi \bigr) dx
	\end{equation*}
	\begin{equation*}
		= -\frac{1}{2\sqrt{c}} \Re e\int_{\mathbb{R}^3} |\Delta Iu_\lambda|^2 \Delta \varphi \, dx - \frac{\lambda}{\sqrt{c}} \Re e\int_{\mathbb{R}^3} \Delta \overline{Iu}_\lambda \, \nabla Iu_\lambda \cdot \nabla \varphi \, dx .
	\end{equation*}

	Next, for the inner integral in the second term,
	\begin{equation*}
		\Re e\int_{\mathbb{R}^3} \Delta Iu_\lambda \, \nabla \overline{Iu} \cdot \nabla \varphi \, dx = \frac{1}{\sqrt{c}} \int_{\mathbb{R}^3} \Delta Iu_\lambda \, \nabla \bigl((-\Delta + \lambda) \overline{Iu}_\lambda\bigr) \cdot \nabla \varphi \, dx
	\end{equation*}
	\begin{equation*}
		= -\frac{\lambda}{\sqrt{c}} \int_{\mathbb{R}^3} \nabla Iu_\lambda \cdot \nabla \bigl( \nabla \overline{Iu}_\lambda \cdot \nabla \varphi \bigr) dx + \frac{1}{2\sqrt{c}} \int_{\mathbb{R}^3} |\Delta Iu_\lambda|^2 \Delta \varphi \, dx .
	\end{equation*}
	
	Finally, for the inner integral in the third term,
	\begin{equation*}
		\Re e \int_{\mathbb{R}^3} \Delta Iu_\lambda \, \overline{Iu} \, \Delta \varphi \, dx = \frac{1}{\sqrt{c}} \Re e\int_{\mathbb{R}^3} \Delta Iu_\lambda \, (-\Delta + \lambda) \overline{Iu}_\lambda \, \Delta \varphi \, dx
	\end{equation*}
	\begin{equation*}
		= -\frac{1}{\sqrt{c}} \Re e \int_{\mathbb{R}^3} |\Delta Iu_\lambda|^2 \Delta \varphi \, dx + \frac{\lambda}{\sqrt{c}} \Re e\int_{\mathbb{R}^3} \Delta Iu_\lambda \, \overline{Iu}_\lambda \, \Delta \varphi \, dx
	\end{equation*}
	\begin{equation*}
		= -\frac{1}{\sqrt{c}} \Re e \int_{\mathbb{R}^3} |\Delta Iu_\lambda|^2 \Delta \varphi \, dx + \frac{\lambda}{2\sqrt{c}} \int_{\mathbb{R}^3} \bigl( \Delta (|Iu_\lambda|^2) - 2|\nabla Iu_\lambda|^2 \bigr) \Delta \varphi \, dx
	\end{equation*}
	\begin{equation*}
		= -\frac{1}{\sqrt{c}} \Re e \int_{\mathbb{R}^3} |\Delta Iu_\lambda|^2 \Delta \varphi \, dx + \frac{\lambda}{2\sqrt{c}} \int_{\mathbb{R}^3} |Iu_\lambda|^2 \Delta^2 \varphi \, dx - \frac{\lambda}{\sqrt{c}} \int_{\mathbb{R}^3} |\nabla Iu_\lambda|^2 \Delta \varphi \, dx .
	\end{equation*}
	
	Then we can obtain
	\begin{equation*}
		J_1= 4 \Re e \int_0^\infty \lambda^{\alpha/2} d\lambda \int_{\mathbb{R}^3} \nabla \overline{Iu}_\lambda \cdot \nabla \bigl( \nabla Iu_\lambda \cdot \nabla \varphi \bigr) dx
	\end{equation*}
	\begin{equation*}
		\quad + 2 \int_0^\infty \lambda^{\alpha/2} d\lambda \int_{\mathbb{R}^3} |\nabla Iu_\lambda|^2 \Delta \varphi \, dx
		- \int_0^\infty \lambda^{\alpha/2} d\lambda \int_{\mathbb{R}^3} |Iu_\lambda|^2 \Delta^2 \varphi \, dx .
	\end{equation*}
	Note that 
	\begin{equation*}
		\Re e\int_{\mathbb{R}^3} \nabla \overline{Iu_\lambda} \cdot \nabla(\nabla Iu_\lambda \cdot \nabla \varphi) dx = \Re e\int_{\mathbb{R}^3} \partial_k \overline{Iu_\lambda} \bigl( \partial_{jk}^2 \varphi \, \partial_j Iu_\lambda + \partial_{jk}^2 Iu_\lambda \, \partial_j \varphi \bigr) dx
	\end{equation*}
	\begin{equation*}
		= \Re e\int_{\mathbb{R}^3} \partial_{jk}^2 \varphi \, \partial_j Iu_\lambda \, \partial_k \overline{Iu_\lambda} \, dx - \Re e\int_{\mathbb{R}^3} \partial_j (\partial_k \overline{Iu_\lambda} \, \partial_j\varphi) \, \partial_k Iu_\lambda \, dx
	\end{equation*}
	\begin{equation*}
		= \Re e\int_{\mathbb{R}^3} \partial_{jk}^2 \varphi \, \partial_j Iu_\lambda \, \partial_k \overline{Iu_\lambda} \, dx - \Re e \int_{\mathbb{R}^3} \partial_k Iu_\lambda \, \partial_{jk}^2 \overline{Iu_\lambda} \, \partial_j \varphi \, dx - \int_{\mathbb{R}^3} |\nabla Iu_\lambda|^2 \Delta \varphi \, dx
	\end{equation*}
	\begin{equation*}
		= \Re e\int_{\mathbb{R}^3} \partial_{jk}^2 \varphi \, \partial_j Iu_\lambda \, \partial_k \overline{Iu_\lambda} \, dx - \frac{1}{2} \int_{\mathbb{R}^3} |\nabla Iu_\lambda|^2 \Delta \varphi \, dx .
	\end{equation*}
	Thus, we complete our calculations.
\end{proof}
		To control the momentum-type term $M_{\varphi}$, we briefly recall the following classical lemma (see Lemma A.10 in \cite{11}).
		\begin{lemma}\label{momentum}
			Let $u,v\in \mathcal{S}(\R^d)$ for some $d\ge 3$, and $K$ be a kernel on $\R^d$, which is smooth away from the origin, and obeys the following estimates
			\begin{equation*}
				|K(x)|\lesssim_d 1, \quad |\nabla K(x)|\lesssim_d |x|^{-1}, \quad \forall x\ne 0.
			\end{equation*}
			Then we have 
			\begin{equation*}
				\Big|\int_{\R^d}u(x) K(x)\nabla v(x) dx\Big|\lesssim_d \|u\|_{\dot{H}^{r}(\R^d)}\|v\|_{\dot{H}^{1-r}(\R^d)}, 
			\end{equation*}
			for all $0\le r\le 1$.
		\end{lemma}
		Now take $\varphi = |x|$. We can immediately derive
		\begin{coro}\label{6}
			Suppose $u\in C([0,T];\mathcal{S}(\R^3))$ is a solution to (\ref{FNLS}), then we have
			\begin{equation*}
				\int_{0}^{T}\int_{\R^3}\frac{|Iu(x,t)|^{4}}{|x|}dxdt\lesssim \sup_{t\in [0,T]}\left(\|Iu(t)\|_{\dot{H}^{\frac{\alpha}{2}}(\R^3)}\|Iu(t)\|_{\dot{H}^{\frac{2-\alpha}{2}}(\R^3)}\right)
			\end{equation*}
			\begin{equation*}
				+\sup_{t\in [0,T]}\|Iu(t)\|_{\dot{H}^{\frac{\alpha}{2}}(\R^3)} \|(u)_{err}\|_{L_t^{1}\dot{H}_x^{\frac{2-\alpha}{2}}([0,T]\times\R^3)}.
			\end{equation*}
		\end{coro}
		\begin{proof}
			Simple calculations show that 
			\begin{equation*}
				\varphi \; \text{is convex}, \quad \Delta \varphi=\frac{2}{|x|}, \quad -\Delta^{2}\varphi=8\pi\delta_0.
			\end{equation*}
			We can apply Lemma \ref{momentum} on $M_{\varphi}$ and $\int_{\R^3}\nabla\varphi\cdot \lbrace(u)_{err}, Iu\rbrace dx$, with $K=\frac{x}{|x|}$, $r=\frac{\alpha}{2}$, which yields the desired bound.
		\end{proof} 
		For a radially symmetric solution $u$, we can further establish a more useful Morawetz estimate, which relies on the following result.
		\begin{lemma}\label{radial sobolev}
			Let $d\ge 2$, and let $r$ satisfy $\frac{1}{2}<r<\frac{d}{2}$. Then for every radially symmetric $u\in \dot{H}^{r}(\R^d)$, i.e., $u\in \dot{H}_{rad}^r(\R^d)$, we have 
			\begin{equation*}
				\||x|^{\frac{d-2r}{2}}u\|_{L^{\infty}(\R^d)}\lesssim \|u\|_{\dot{H}^{r}(\R^d)}.
			\end{equation*}
		\end{lemma}
		\begin{proof}
			See Proposition 1 in \cite{12}.
		\end{proof}
		\begin{coro}\label{Morawetz}
			Suppose $u\in C([0,T];\mathcal{S}(\R^3))$ is a radial solution to (\ref{FNLS}), then we have
			\begin{equation*}
				\int_{0}^{T}\int_{\R^3}|Iu(x,t)|^{4+\frac{2}{3-\alpha}}dxdt\lesssim \||D|^{\frac{\alpha}{2}}Iu\|_{L_t^{\infty}L_x^{2}([0,T]\times\R^3)}^{1+\frac{2}{3-\alpha}}
			\end{equation*}
			\begin{equation*}
				\quad \quad \quad \quad \quad \quad\times \left(\||D|^{\frac{2-\alpha}{2}}Iu\|_{L_t^{\infty}L_x^{2}([0,T]\times\R^3)}+\||D|^{\frac{2-\alpha}{2}}(u)_{err}\|_{L_t^{1}L_x^{2}([0,T]\times\R^3)}\right)
			\end{equation*}
		\end{coro}
		\begin{proof}
			From Lemma \ref{radial sobolev} with $r=\frac{\alpha}{2}$, we can obtain
			\begin{equation*}
				\int_{0}^{T}\int_{\R^3}|Iu(x,t)|^{4+\frac{2}{3-\alpha}}dxdt=\int_{0}^{T}\int_{\R^3}\frac{|Iu(x,t)|^{4}}{|x|}|x||Iu(x,t)|^{\frac{2}{3-\alpha}}dxdt
			\end{equation*}
			\begin{equation*}
				\lesssim \||D|^{\frac{\alpha}{2}}Iu\|_{L_t^{\infty}L_x^{2}([0,T]\times\R^3)}^{\frac{2}{3-\alpha}}\int_{0}^{T}\int_{\R^3}\frac{|Iu(x,t)|^{4}}{|x|}dxdt.
			\end{equation*}
			Then it directly follows from Corollary \ref{6}.
		\end{proof}
		From now on, we will always denote $p_\alpha:=4+\frac{2}{3-\alpha}$.

		\vspace{15pt}
		\section{Fractional Schr\"{o}dinger equation with radial data}
		For the radial case, we have the following Strichartz estimate with no derivative loss (see \cite{6,5}).
		
		\begin{lemma}[Strichartz estimate with no derivative loss]\label{Strichartz2}
			\
			
		Suppose $d\ge 2$, $\alpha>1$, and let $u, u_0, F$ be radially symmetric in space satisfying 
		\[
		\begin{cases} 
			i\partial_t u =|D|^{\alpha}u+F, & (x,t) \in \mathbb{R}^d \times \mathbb{R} \\ 
			u(0,x)= u_0(x).
		\end{cases}
		\]
			Then for any \( \gamma \in \mathbb{R} \), and any spacetime slab $I\times \R^d$, it holds
			\[
			\|u\|_{L^q_t L^r_x(I\times\R^d)} + \|u\|_{C(I; \dot{H}^\gamma(\R^d))} \lesssim \|u_0\|_{\dot{H}^\gamma(\R^d)} + \|F\|_{L^{\tilde{q}'}_t L^{\tilde{r}'}_x(I\times \R^d)}, 
			\]
			
			if the following conditions hold:
			
			(1) \((q,r)\) and \((\tilde{q},\tilde{r})\) both satisfy the following conditions:
			\[
			2 \leq q, r \leq \infty, \quad \frac{1}{q} < \left(d - \frac{1}{2}\right)\left(\frac{1}{2} - \frac{1}{r}\right);
			\]
			
			(2) \(\tilde{q}' < q\) and the “gap” condition:
			\[
			\frac{\alpha}{q} + \frac{d}{r} = \frac{d}{2} - \gamma, \quad \frac{\alpha}{\tilde{q}} + \frac{d}{\tilde{r}} = \frac{d}{2} + \gamma.
			\]
		\end{lemma}
		\begin{rem}
			A pair $(q,r)$ is called $\alpha$-admissible if $(q,r)$ satisfies
			\begin{equation*}
				\frac{\alpha}{q}+\frac{d}{r}=\frac{d}{2},  \quad 2\le q,r\le\infty.
			\end{equation*}
			Then if $\alpha\in (\frac{2d}{2d-1},2)$, we can obtain a full set of $\alpha$-admissible Strichartz estimates.
		\end{rem}
		
		In the radial setting, we will replace the Bourgain space with the following Strichartz space, which is more compatible with the Morawetz estimate. In fact, we can define the Strichartz norm of a function $u:[0,T]\times \R^3\to \C$ as
		\begin{equation*}
			\|u\|_{S_T^{0}}:=\sup_{(q,r)\; \text{is}\; \alpha-\text{admissible} }\|u\|_{L_t^{q}L_x^{r}([0,T]\times\R^3)}.
		\end{equation*}
		
		Next, we will establish an a priori estimate, which connects the Morawetz estimate and the Strichartz norm. Before formally stating it, we need a very simple lemma, which can ensure the existence of some parameters.
		
		\begin{lemma}\label{parameter}
			There exist $\alpha$-admissible pairs $(q_1,r_1)$, $(q_2,r_2)$ such that 
			\begin{equation*}
				\frac{1}{q_1}+\frac{1}{q_2}=a, \quad \frac{1}{r_1}+\frac{1}{r_2}=b,
			\end{equation*}
			if and only if 
			\begin{equation}\label{1}
				b=1-\frac{\alpha}{3}a, \quad  0<a\le 1.
			\end{equation}
		\end{lemma}
		\begin{proof}
			The proof is simple; we omit it.
		\end{proof}
        
   \begin{prop}\label{priori}
   	We define the quantity 
   	\begin{equation*}
   		\mu([0,T]):=\int_{0}^{T}\int_{\R^3}|Iu(x,t)|^{p_\alpha}dxdt.
   	\end{equation*}
   	
   There exists a constant $\mu_0=\mu_0(\|Iu_0\|_{\dot{H}^{\frac{\alpha}{2}}})>0$ such that if $\mu([0,T])<\mu_0$ and $u\in C([0,T]; \mathcal{S}(\R^3))$ is a solution to $(\ref{FNLS})$, then we have
   \begin{equation*}
   	Z_{I}([0,T]):=\sup_{(q,r)\; \text{is}\; \alpha-\text{admissible} }\||D|^{\frac{\alpha}{2}}Iu\|_{L_t^{q}L_x^{r}([0,T]\times\R^3)}\lesssim \||D|^{\frac{\alpha}{2}}Iu_0\|_{L^{2}(\R^3)}.
   \end{equation*}
   	
   \end{prop}
   
   \begin{proof}
   	For $\alpha \in (\frac{3}{2}, \frac{7-\sqrt{13}}{2})$, we can determine unique $\theta_1, \theta_2\in (0,1)$, such that 
   	\begin{equation*}
   		\frac{\alpha}{2}=\theta_1\left(\frac{3}{2}-\frac{\alpha}{2}\right)+(1-\theta_1)\frac{\alpha+3}{p_\alpha}=\theta_2\left(\frac{3}{2}-\lambda\right)+(1-\theta_2)\frac{\alpha+3}{p_\alpha},
   	\end{equation*}
   	where $\lambda\in \left(\frac{3-\alpha}{2}, \frac{\alpha}{2}\right)$ is to be determined.
   	
   	Then we claim that there exist $q_3, q_4^{(1)}, q_4^{(2)}$ satisfying
   	\begin{equation}\label{2}
   		0<\frac{1}{q_3}<\frac{1}{2}, \quad 0<\frac{1}{q_4^{(1)}}<\frac{3-\alpha}{2\alpha}, \quad 0<\frac{1}{q_4^{(2)}}<\frac{3-2\lambda}{2\alpha},
   	\end{equation} 
   	and 
   	\begin{equation*}
   		\frac{1}{q_3}=\frac{\theta_1}{q_4^{(1)}}+\frac{1-\theta_1}{p_\alpha}=\frac{\theta_2}{q_4^{(2)}}+\frac{1-\theta_2}{p_\alpha}.
   	\end{equation*}
   	In fact, it is equivalent to 
   	\begin{equation*}
   		I_1:= \left(\frac{1-\theta_1}{p_\alpha},\frac{1-\theta_1}{p_\alpha}+\theta_1\frac{3-\alpha}{2\alpha} \right), \quad I_2:= \left(\frac{1-\theta_2}{p_\alpha},\frac{1-\theta_2}{p_\alpha}+\theta_2\frac{3-2\lambda}{2\alpha} \right),
   	\end{equation*}
   	such that $I_1\cap I_2\cap (0,1/2)\ne \emptyset$.
   	After some calculations, we see that $I_1\subset I_2 \subset (0,1/2)$ for any $\lambda\in \left(\frac{3-\alpha}{2}, \frac{\alpha}{2}\right)$, and then the claim follows.
   	
   	Now, we can apply Lemma \ref{parameter} with  
   	\begin{equation*}
   		\frac{1}{r_3}:=\frac{\alpha}{3}\left(\frac{1}{2}-\frac{1}{q_3}\right), \quad a:=1-\frac{2}{q_3}, \quad b:=1-\frac{2}{r_3}, 
   	\end{equation*}
   	which ensures the existence of $\alpha$-admissible pairs $(q_1,r_1)$, $(q_2,r_2)$ such that
   	\begin{equation*}
   		\frac{1}{q_1'}=\frac{1}{q_2}+\frac{2}{q_3}, \quad \frac{1}{r_1'}=\frac{1}{r_2}+\frac{2}{r_3}.
   	\end{equation*}
   	
   	Applying the Strichartz estimate in Lemma \ref{Strichartz2} and the Duhamel formula for $Iu$, we can deduce that
   	\begin{equation*}
   		Z_{I}([0,T])\lesssim \||D|^{\frac{\alpha}{2}}Iu_0\|_{L^{2}(\R^d)}+\||D|^{\frac{\alpha}{2}}I(|u|^2u)\|_{L_t^{q_1'}L_x^{r_1'}([0,T]\times\R^3)}
   	\end{equation*}
   	\begin{equation*}
   		\lesssim\||D|^{\frac{\alpha}{2}}Iu_0\|_{L^{2}(\R^d)}+\||D|^{\frac{\alpha}{2}}Iu\|_{L_t^{q_2}L_x^{r_2}([0,T]\times\R^3)}\|u\|_{L_t^{q_3}L_x^{r_3}([0,T]\times\R^3)}^2
   	\end{equation*}
   	\begin{equation}\label{3}
   		\lesssim \||D|^{\frac{\alpha}{2}}Iu_0\|_{L^{2}(\R^d)}+ Z_I\cdot\|u\|_{L_t^{q_3}L_x^{r_3}([0,T]\times\R^3)}^2,
   	\end{equation}
   	where the second inequality comes from the Leibniz rule for the fractional derivative.
   	
   	Then we write $u=P_{<N}u+\sum_{j=1}^{\infty}P_{N_j}u=:u_{<N}+\sum_{j=1}^{\infty} u_{N_j}$, where $N_j$ are consecutive dyadic integers starting from $[\log_2(N)]$. 
   	
   	By the triangle inequality, we can derive
   	\vspace{-5pt}
   	\begin{equation*}
   		\|u\|_{L_t^{q_3}L_x^{r_3}([0,T]\times\R^3)}\le \|u_{<N}\|_{L_t^{q_3}L_x^{r_3}([0,T]\times\R^3)}+\sum_{j=1}^{\infty}\|u_{N_j}\|_{L_t^{q_3}L_x^{r_3}([0,T]\times\R^3)}.
   	\end{equation*}
   	
   	\vspace{-5pt}
   	Next, from the constraint (\ref{2}), we can take  
   	\begin{equation*}
   		\frac{1}{r_4^{(1)}}:=\frac{1}{3}\left(\frac{3}{2}-\frac{\alpha}{q_4^{(1)}}\right)>\frac{\alpha}{6},\quad \frac{1}{r_4^{(2)}}:=\frac{1}{3}\left(\frac{3}{2}-\frac{\alpha}{q_4^{(2)}}\right)>\frac{\lambda}{3}.
   	\end{equation*}
   	Applying the Gagliardo-Nirenberg inequality, we can obtain
   	\begin{itemize}
   		\item 
   		\begin{equation*}
   			\|u_{<N}\|_{L_t^{q_3}L_x^{r_3}([0,T]\times\R^3)}= \|Iu\|_{L_t^{q_3}L_x^{r_3}([0,T]\times\R^3)}\lesssim \||D|^{\frac{\alpha}{2}}Iu\|_{L_t^{q_4^{(1)}}L_x^{r_4^{(1)}}}^{\theta_1}\|Iu\|_{L_{t,x}^{p_\alpha}}^{1-\theta_1}
   		\end{equation*}
   		\vspace{-9pt}
   		\begin{equation*}
   			\le Z_I^{\theta_1}\mu_0^{1-\theta_1},
   		\end{equation*}
   		
   		\item 
   		\begin{equation*}
   			\|u_{N_j}\|_{L_t^{q_3}L_x^{r_3}([0,T]\times\R^3)}\lesssim N_j^{\frac{\alpha}{2}-s}N^{s-\frac{\alpha}{2}}\|Iu\|_{L_t^{q_3}L_x^{r_3}([0,T]\times\R^3)}
   		\end{equation*}
   		\begin{equation*}
   			\lesssim N_j^{\frac{\alpha}{2}-s}N^{s-\frac{\alpha}{2}}\||D|^{\lambda}Iu\|_{L_t^{q_4^{(2)}}L_x^{r_4^{(2)}}}^{\theta_2}\|Iu\|_{L_{t,x}^{p_\alpha}}^{1-\theta_2}\lesssim N_j^{\frac{\alpha}{2}-s-(\frac{\alpha}{2}-\lambda)\theta_2}Z_I^{\theta_2}\mu_0^{1-\theta_2}.
   		\end{equation*}
   	\end{itemize}
   	Note that when $\lambda\to \frac{3-\alpha}{2}$, we have $\theta_2\to 1$. Since $s>\frac{3-\alpha}{2}$, we can take $\lambda$ sufficiently close to $\frac{3-\alpha}{2}$ such that $\frac{\alpha}{2}-s-(\frac{\alpha}{2}-\lambda)\theta_2<0$.
   	
   	Thus we can obtain 
   	\begin{equation}\label{4}
   		\|u\|_{L_t^{q_3}L_x^{r_3}([0,T]\times\R^3)}\lesssim Z_I^{\theta_1}\mu_0^{1-\theta_1}+Z_I^{\theta_2}\mu_0^{1-\theta_2}.
   	\end{equation}
   	Substituting (\ref{4}) into (\ref{3}), we can derive 
   	\begin{equation*}
   		Z_I\lesssim \||D|^{\frac{\alpha}{2}}Iu_0\|_{L^{2}(\R^d)}+Z_I^{1+2\theta_1}\mu_0^{2(1-\theta_1)}+Z_I^{1+2\theta_2}\mu_0^{2(1-\theta_2)}, 
   	\end{equation*}
   	which yields the desired estimate.
   \end{proof}
   
         Recall that we have established the Morawetz estimate for the modified solution $Iu$ in Section 6, with an error term $\||D|^{\frac{2-\alpha}{2}}(u)_{err}\|_{L_t^{1}L_x^{2}([0,T]\times\R^3)}$. Now, we can use the Strichartz norm, or specifically $Z_I$, to control this term.
         \begin{prop}\label{error term control}
         	Suppose $u\in C([0,T]; \mathcal{S}(\R^3))$ is a solution to $(\ref{FNLS})$. Then we have the following estimate:
         	\begin{equation*}
         		\||D|^{\frac{2-\alpha}{2}}(u)_{err}\|_{L_t^{1}L_x^{2}([0,T]\times\R^3)}\lesssim N^{-(3\alpha-4)^{-}}\cdot Z_I^{3}.
         	\end{equation*}
         \end{prop}
         \begin{proof}
         	Applying Plancherel and decomposing the error term into Littlewood-Paley pieces, we can obtain 
         	\begin{equation*}
         		\||D|^{\frac{2-\alpha}{2}}(u)_{err}\|_{L_t^{1}L_x^{2}([0,T]\times\R^3)}=\|\widehat{|D|^{\frac{2-\alpha}{2}}(u)_{err}}\|_{L_t^{1}L_\xi^{2}([0,T]\times\R^3)}
         	\end{equation*}
         	\begin{equation*}
         		\le \sum_{N_1, N_2, N_3} \left\|\int_{\substack{
         				\xi_1+\xi_2+\xi_3=\xi \\
         				|\xi_i|\sim N_i
         		}}|\xi|^{\frac{2-\alpha}{2}}\left(m(\xi)-m(\xi_1)m(\xi_2)m(\xi_3)\right)\widehat{u_1}\widehat{u_2}\widehat{u_3} \right\|_{L_t^{1}L_\xi^{2}([0,T]\times\R^3)}
         	\end{equation*}
         	\begin{equation*}
         		=\sum_{N_1, N_2, N_3} \left\|\int_{\substack{
         				\xi_1+\xi_2+\xi_3=\xi \\
         				|\xi_i|\sim N_i
         		}}|\xi|^{\frac{2-\alpha}{2}}\frac{\left(m(\xi)-m(\xi_1)m(\xi_2)m(\xi_3)\right)}{m(\xi_1)m(\xi_2)m(\xi_3)}\widehat{Iu_1}\widehat{Iu_2}\widehat{Iu_3} \right\|_{L_t^{1}L_\xi^{2}([0,T]\times\R^3)}
         	\end{equation*}
         	\begin{equation*}
         		=: \sum_{N_1, N_2, N_3} L(N_1, N_2, N_3).
         	\end{equation*}
         	Here we denote $\widehat{u_i}:=\widehat{u}(\xi_i)$ and ignore the conjugate, which never affects our discussions.  
         	
         	Without loss of generality, we can further assume that $N_1\ge N_2\ge N_3$. We denote the corresponding symbol 
         	\begin{equation*}
         		\sigma(\xi_1,\xi_2, \xi_3):=|\xi_1+\xi_2+\xi_3|^{\frac{2-\alpha}{2}}\frac{\left(m(\xi)-m(\xi_1)m(\xi_2)m(\xi_3)\right)}{m(\xi_1)m(\xi_2)m(\xi_3)}.
         	\end{equation*}
         	Next, we split the different frequency interactions into four cases, depending on the size of the parameters $N, N_i$. We only consider the case $N_i\ge1$ for all $i=1,2,3$, since the low frequency case can be obtained similarly.
         	
         	\textbf{Case 1: $N\gg N_1\ge N_2\ge N_3.$}
         	
         	By the definition of $I$, the symbol $\sigma$ is identically zero and the desired bound holds trivially.
         	
         	\textbf{Case 2: $N_1\gtrsim N\gg N_2\ge N_3.$}
         	
         	From the mean value theorem, we can derive
         	\begin{equation*}
         		|\sigma(\xi_1,\xi_2, \xi_3)|\lesssim N_1^{\frac{2-\alpha}{2}}\cdot \frac{|m(\xi_1+\xi_2+\xi_3)-m(\xi_1)|}{m(\xi_1)}\lesssim N_1^{\frac{2-\alpha}{2}}\cdot \frac{|\nabla m(\xi_1)(\xi_2+\xi_3)|}{m(\xi_1)}\lesssim N_1^{-\frac{\alpha}{2}}N_2.
         	\end{equation*}
         	Then, applying the Coifman-Meyer multiplier theorem, we obtain
         	\begin{equation*}
         		L(N_1, N_2, N_3)\lesssim  N_1^{-\frac{\alpha}{2}}N_2 \|Iu_1\|_{L_t^{2}L_x^{\frac{6}{3-\alpha}}}\|Iu_2\|_{L_t^{2}L_x^{\frac{12}{\alpha}}}\|Iu_3\|_{L_t^{\infty}L_x^{\frac{12}{\alpha}}}
         	\end{equation*}
         	\begin{equation*}
         		\lesssim N_1^{-\frac{\alpha}{2}}N_2 \|Iu_1\|_{L_t^{2}L_x^{\frac{6}{3-\alpha}}} \||D|^{\frac{6-3\alpha}{4}}Iu_2\|_{L_t^{2}L_x^{\frac{6}{3-\alpha}}}\||D|^{\frac{6-\alpha}{4}}Iu_3\|_{L_t^{\infty}L_x^{2}}
         	\end{equation*}
         	\begin{equation*}
         		\lesssim N_1^{-\alpha}N_2^{\frac{5}{2}-\frac{5\alpha}{4}}N_3^{\frac{3}{2}-\frac{3\alpha}{4}}\cdot Z_I^{3}\lesssim N^{-(3\alpha-4)^{-}}N_1^{0^{-}}\cdot Z_I^{3}.
         	\end{equation*}
         	
         	\textbf{Case 3: $N_1\ge N_2\gtrsim N\gg N_3.$}
         	\begin{equation*}
         		|\sigma(\xi_1,\xi_2, \xi_3)|\lesssim N_1^{\frac{2-\alpha}{2}}\cdot\frac{m(\xi)}{m(\xi_1)m(\xi_2)}\lesssim N_1^{\frac{2-\alpha}{2}} \left(\frac{N_1}{N}\right)^{\frac{\alpha}{2}-s}\left(\frac{N_2}{N}\right)^{\frac{\alpha}{2}-s}=N_1^{1-s}N_2^{\frac{\alpha}{2}-s}N^{2s-\alpha}.
         	\end{equation*}
         	Similarly, we can derive 
         	\begin{equation*}
         		L(N_1, N_2, N_3)\lesssim  N_1^{1-s-\frac{\alpha}{2}}N_2^{\frac{3}{2}-\frac{3\alpha}{4}-s}N_3^{\frac{3}{2}-\frac{3\alpha}{4}}N^{2s-\alpha}\cdot Z_I^3\lesssim N^{-(3\alpha-4)^{-}}N_1^{0^{-}}\cdot Z_I^{3}.
         	\end{equation*}
         	
         	\textbf{Case 4: $N_1\ge N_2\ge N_3\gtrsim N.$}
         	\begin{equation*}
         		|\sigma(\xi_1,\xi_2, \xi_3)|\lesssim N_1^{\frac{2-\alpha}{2}}\cdot\frac{m(\xi)}{m(\xi_1)m(\xi_2)m(\xi_3)}\lesssim N_1^{\frac{2-\alpha}{2}} \left(\frac{N_1}{N}\right)^{\frac{\alpha}{2}-s}\left(\frac{N_2}{N}\right)^{\frac{\alpha}{2}-s}\left(\frac{N_3}{N}\right)^{\frac{\alpha}{2}-s}
         	\end{equation*}
         	\begin{equation*}
         		=N_1^{1-s}N_2^{\frac{\alpha}{2}-s}N_3^{\frac{\alpha}{2}-s}N^{3s-\frac{3\alpha}{2}}
         	\end{equation*}
         	Then we see 
         	\begin{equation*}
         		L(N_1, N_2, N_3)\lesssim N_1^{1-s-\frac{\alpha}{2}}N_2^{\frac{3}{2}-\frac{3\alpha}{4}-s}N_3^{\frac{3}{2}-\frac{\alpha}{4}-s}N^{3s-\frac{3\alpha}{2}}Z_I^{3}\lesssim N^{-(3\alpha-4)^{-}}N_1^{0^{-}}\cdot Z_I^{3}.
         	\end{equation*}
         	In fact, one can check that when $s>\frac{3-\alpha}{2}$, we have
         	\begin{equation*}
         		\left(\frac{\alpha}{2}+s-1\right)+\left(\frac{3\alpha}{4}+s-\frac{3}{2}\right)=\frac{5\alpha}{4}-\frac{5}{2}+2s>\frac{3}{2}-\frac{\alpha}{4}-s.
         	\end{equation*}
         	Thus, we complete the proof.
         \end{proof}

	 Next, we can recover the bilinear estimate in terms of Strichartz norm, which will be used in establishing a slightly different ``almost conservation law''.
	 \begin{lemma}\label{bilinear in Strichartz}
	 	With the same assumption in Proposition \ref{priori}, the following bilinear estimate holds:
	 	\begin{equation*}
	 		\|(P_{N_1}|D|^{\frac{\alpha}{2}}Iu_1) (P_{N_2}|D|^{\frac{\alpha}{2}}Iu_2)\|_{L_{t,x}^{2}([0,T]\times \R^3)}\lesssim \frac{\min\lbrace N_1, N_2\rbrace}{\max\lbrace N_1, N_2\rbrace^{\frac{\alpha-1}{2}}}(1+Z_I^6).
	 	\end{equation*}
	 	In particular, if $\||D|^{\frac{\alpha}{2}}I u_0\|_{L^{2}(\R^3)}\lesssim 1$, we have 
	 	\begin{equation*}
	 		\|(P_{N_1}|D|^{\frac{\alpha}{2}}Iu_1) (P_{N_2}|D|^{\frac{\alpha}{2}}Iu_2)\|_{L_{t,x}^{2}([0,T]\times \R^d)}\lesssim \frac{\min\lbrace N_1, N_2\rbrace}{\max\lbrace N_1, N_2\rbrace^{\frac{\alpha-1}{2}}}.
	 	\end{equation*}
	 \end{lemma}
	 \begin{proof}
	 	Applying the bilinear estimate in Lemma \ref{bilinear} with the following Duhamel formula
	 	\begin{equation*}
	 		P_{N_j}|D|^{\frac{\alpha}{2}}Iu_j=e^{-it|D|^{\alpha}}P_{N_j}|D|^{\frac{\alpha}{2}}Iu_j(0)-i\int_{0}^{t}e^{-i(t-t')|D|^{\alpha}}P_{N_j}|D|^{\frac{\alpha}{2}}I(|u_j|^{2}u_j)dt', \quad j=1,2.
	 	\end{equation*}
	 	It suffices to show 
	 	\begin{equation*}
	 		\||D|^{\frac{\alpha}{2}}I(|u_j|^{2}u_j)\|_{L_t^{1}L_x^{2}([0,T]\times\R^3)}\lesssim 1+Z_I^3.
	 	\end{equation*}
	 	Mimicking the proof in Proposition \ref{priori}, we just need to ensure that $q_1=\infty, r_1=2$ can be achieved. Investigating Lemma \ref{parameter}, we just need to require $0<a\le \frac{1}{2}$, which is equivalent to $I_1\cap I_2\cap [\frac{1}{4}, \frac{1}{2})\ne \emptyset$. After some calculations, we see
	 	\begin{equation*}
	 		I_1=\left[\frac{2\alpha-3}{2(3\alpha-4)},\frac{3\alpha^2-10\alpha+9}{2\alpha(3\alpha-4)}\right]\subset I_2.
	 	\end{equation*}
	 	Then $\alpha\in (\frac{3}{2}, \frac{8-\sqrt{10}}{3}]$ implies $I_1\cap I_2\cap [\frac{1}{4}, \frac{1}{2})\ne \emptyset$.
	 \end{proof}
	   
	  \vspace{7pt}
	  Now, we need to recover our ``almost conservation law'' in terms of Strichartz norm, i.e., $Z_I([0,T])$, rather than Bourgain space $X^{0,\frac{1}{2}^+}$.
	  
	  It should be emphasized that the main difference and difficulty, compared with Proposition \ref{d=3}, arises from the fact that \(T\) is no longer necessarily \(\lesssim 1\). Consequently, we can no longer exploit H\"older's inequality in time, which forces a more delicate choice of parameters.
	  
	  \begin{prop}\label{new almost}
	  	With the same assumption in Proposition \ref{priori} and $s>1-\frac{\alpha}{6}$, the following ``almost conservation law'' holds:
	  	\begin{equation*}
	  		E(Iu)(t)-E(Iu)(0)=O(N^{-(2\alpha-3)^-}), \quad \forall t\in [0,T].
	  	\end{equation*}
	  \end{prop}  
	  \begin{proof}
	  	Recall that $\textbf{Term}_1$ and $\textbf{Term}_2$ are defined as
	  	\begin{itemize}
	  		\item $\displaystyle
	  		|\textbf{Term}_1| := \Bigg|\int_{0}^{t}\int_{\sum_{j=1}^{4}\xi_j=0}
	  		\left(1-\frac{m(\xi_1)}{m(\xi_2)m(\xi_3)m(\xi_4)}\right)
	  		\widehat{\overline{|D|^{\alpha}Iu}}(\xi_1)
	  		\widehat{Iu}(\xi_2)
	  		\widehat{\overline{Iu}}(\xi_3)
	  		\widehat{Iu}(\xi_4)\Bigg|;
	  		$
	  		\item $\displaystyle
	  		|\textbf{Term}_2| := \Bigg|\int_{0}^{t}\int_{\sum_{j=1}^{4}\xi_j=0}
	  		\left(1-\frac{m(\xi_1)}{m(\xi_2)m(\xi_3)m(\xi_4)}\right)
	  		\widehat{\overline{I(|u|^2 u)}}(\xi_1)
	  		\widehat{Iu}(\xi_2)
	  		\widehat{\overline{Iu}}(\xi_3)
	  		\widehat{Iu}(\xi_4)\Bigg|.
	  		$
	  	\end{itemize}
	  	
	  	Thus, it suffices to prove
	  	\begin{equation}\label{term1'}
	  		\Bigg|\int_{0}^{t}\int_{\sum_{j=1}^{4}\xi_j=0}
	  		\left(1-\frac{m(\xi_1)}{m(\xi_2)m(\xi_3)m(\xi_4)}\right)
	  		\widehat{\overline{P_{N_1}|D|^{\alpha}Iu}}(\xi_1)
	  		\widehat{P_{N_2}Iu}(\xi_2)
	  		\widehat{\overline{P_{N_3}Iu}}(\xi_3)
	  		\widehat{P_{N_4}Iu}(\xi_4)\Bigg|
	  	\end{equation}
	  	\begin{equation*}
	  		\lesssim N^{-(2\alpha-3)^{-}} N_1^{0\pm} N_2^{0\pm} N_3^{0\pm} N_4^{0\pm} (1+Z_I^6)^2,
	  	\end{equation*}
	  	and 
	  	\begin{equation}\label{term2'}
	  		\Bigg|\int_{0}^{t}\int_{\sum_{j=1}^{6}\xi_j=0}
	  		\left(1-\frac{m(\xi_1+\xi_2+\xi_3)}{m(\xi_4)m(\xi_5)m(\xi_6)}\right)
	  		P_{N_{123}}\widehat{\overline{I(\phi_1\phi_2\phi_3)}}(\xi_1+\xi_2+\xi_3)
	  		\widehat{I\phi_4}(\xi_4)
	  		\widehat{\overline{I\phi_5}}(\xi_5)
	  		\widehat{I\phi_6}(\xi_6)\Bigg|
	  	\end{equation}
	  	\begin{equation*}
	  		\lesssim N^{-(2\alpha-3)^{-}}N_{123}^{0\pm}N_1^{0\pm}N_2^{0\pm}N_3^{0\pm}N_4^{0\pm}N_5^{0\pm}N_6^{0\pm}\prod_{i=1}^{6}\||D|^{\frac{\alpha}{2}}I\phi_i\|_{S_T^{0}},
	  	\end{equation*}
	  	for any functions $\phi_i$, with spatial Fourier transform supported on $|\xi_i|\sim N_i$.
	  	
	  	\vspace{7pt}
	  	Now for $\textbf{Term}_1$, we can apply the bilinear estimate in Lemma \ref{bilinear in Strichartz} instead of bilinear estimate (\ref{e}). Then, following the argument in Proposition \ref{d=3}, we can similarly conclude that (\ref{term1'}) holds and, thus,
	  	\begin{equation*}
	  		|\textbf{Term}_1|\lesssim O(N^{-(2\alpha-3)^{-}}).
	  	\end{equation*}
	  	
	  	For $\textbf{Term}_2$, we first note that, by using Sobolev embedding instead, the following estimates are still true:
	  	\begin{equation*}
	  		\|I\phi_4\|_{L_{t,x}^{\frac{10}{3}}([0,T]\times\R^3)}\lesssim \||D|^{\frac{3(2-\alpha)}{10}}I\phi_4\|_{L_t^{\frac{10}{3}}L_x^{\frac{10}{5-\alpha}}([0,T]\times\R^3)}\lesssim N_4^{-\frac{1}{5}(4\alpha-3)}\||D|^{\frac{\alpha}{2}}I\phi_4\|_{S_T^{0}},
	  	\end{equation*}
	  	\begin{equation*}
	  		\|I\phi_i\|_{L_{t,x}^{10}([0,T]\times\R^3)}\lesssim \||D|^{\frac{12-\alpha}{10}}I\phi_i\|_{L_t^{10}L_x^{\frac{30}{15-\alpha}}([0,T]\times\R^3)}\lesssim N_i^{\frac{1}{5}(6-3\alpha)}\||D|^{\frac{\alpha}{2}}I\phi_i\|_{S_T^{0}}, \quad i=5,6.
	  	\end{equation*}
	  	Thus, it remains to establish Lemma \ref{1234} without using H\"older's inequality in time.
	  	
	  	\vspace{7pt}
	  	We claim that if $\max\lbrace N_1, N_2, N_3\rbrace\ge 1$, then the following estimate holds: 
	  	\begin{equation}\label{5.8}
	  		\|I(\phi_1 \phi_2 \phi_3)\|_{L_{t,x}^{2}([0,T]\times\R^3)}\lesssim N_1^{0\pm}N_2^{0\pm}N_3^{0\pm}\prod_{i=1}^3 \||D|^{\frac{\alpha}{2}}I\phi_i\|_{S_T^{0}}.
	  	\end{equation}
	  	
	  	If we assume the above claim, then the only case left is $\max\lbrace N_1, N_2, N_3\rbrace\le 1$, i.e., all have low frequency. As a result, we have 
	  	\begin{equation*}
	  		N_4\sim N_5\gtrsim N\gg 1\gtrsim N_{123}\;\; \&\;\; \max\lbrace N_1, N_2, N_3\rbrace\le 1.
	  	\end{equation*}
	  	
	  	For this case, (\ref{term2'}) can be alternatively reduced to
	  	\begin{equation*}
	  		\|P_{N_{123}}I(\phi_1\phi_2\phi_3)\|_{L_t^{\frac{\alpha}{\alpha-1}}L_x^{\frac{6}{\alpha-1}^-}([0,T]\times\R^3)}\|I\phi_4\|_{L_t^{2\alpha}L_x^{3}([0,T]\times\R^3)}\|I\phi_5\|_{L_t^{2\alpha}L_x^{3}([0,T]\times\R^3)}\|I\phi_6\|_{L_t^{\infty}L_x^{\frac{6}{3-\alpha}^+}([0,T]\times\R^3)}.
	  	\end{equation*}
	  	
	  	From Sobolev embedding and H\"older's inequality, we can also conclude that
	  	\begin{itemize}
	  		\item 
	  		\begin{equation*}
	  			\|P_{N_{123}}I(\phi_1\phi_2\phi_3)\|_{L_t^{\frac{\alpha}{\alpha-1}}L_x^{\frac{6}{\alpha-1}^-}([0,T]\times\R^3)}\lesssim \prod_{i=1}^3\|\phi_i\|_{L_t^{\frac{3\alpha}{\alpha-1}}L_x^{\frac{18}{\alpha-1}^-}([0,T]\times\R^3)}
	  		\end{equation*}
	  		\begin{equation*}
	  			\lesssim \prod_{i=1}^3\||D|^{\frac{4-\alpha}{2}^-}\phi_i\|_{L_t^{\frac{3\alpha}{\alpha-1}}L_x^{\frac{18}{11-2\alpha}}([0,T]\times\R^3)}\lesssim N_1^{(2-\alpha)^-}N_2^{(2-\alpha)^-}N_3^{(2-\alpha)^-}\prod_{i=1}^3 \||D|^{\frac{\alpha}{2}}I\phi_i\|_{S_T^0}
	  		\end{equation*}
	  		\begin{equation*}
	  			\lesssim N_1^{0^+}N_2^{0^+}N_3^{0^+}\prod_{i=1}^3 \||D|^{\frac{\alpha}{2}}I\phi_i\|_{S_T^0};
	  		\end{equation*}
	  		
	  		\item 
	  		\begin{equation*}
	  			\|I\phi_i\|_{L_t^{2\alpha}L_x^{3}([0,T]\times\R^3)}\lesssim N_i^{-\frac{\alpha}{2}}\||D|^{\frac{\alpha}{2}}I\phi_i\|_{S_T^0}, \quad i=4,5;
	  		\end{equation*}
	  		
	  		\item 
	  		\begin{equation*}
	  			\|I\phi_6\|_{L_t^{\infty}L_x^{\frac{6}{3-\alpha}^+}([0,T]\times\R^3)}\lesssim \||D|^{\frac{\alpha}{2}^+}I\phi_6\|_{L_t^{\infty}L_x^{2}([0,T]\times\R^3)}\lesssim N_6^{0^{+}}\||D|^{\frac{\alpha}{2}}I\phi_6\|_{S_T^0}.
	  		\end{equation*}
	  		
	  	\end{itemize}
	  	
	  	\vspace{7pt}
	  	Finally, one can directly verify that 
	  	\begin{equation*}
	  		\frac{N_{123}^{0^+}}{m(N_4)m(N_5)m(N_6)}N_1^{0^+}N_2^{0^+}N_3^{0^+}N_4^{-\frac{\alpha}{2}}N_5^{-\frac{\alpha}{2}}N_6^{0^{+}}\lesssim N^{-\alpha^-}N_{123}^{0^+}N_1^{0^+}N_2^{0^+}N_3^{0^+}N_4^{0^-}N_6^{0\pm}, 
	  	\end{equation*}
	  	and $\alpha>2\alpha-3$, when $s>1-\frac{\alpha}{6}$, $\; \frac{3}{2}<\alpha<2$.
	  	
	  	\vspace{7pt}
	  	Now, it remains to prove the claim, i.e., (\ref{5.8}). Before starting the proof, we need a simple lemma, which is similar to Lemma \ref{parameter}.
	  	\begin{lemma}\label{parameter1}
	  		Let $\gamma_1, \gamma_2\ge 0$. There exist pairs $(q_1, r_1)$, $(q_2, r_2)\in [1,\infty]^{2}$ satisfying 
	  		\begin{equation*}
	  			\begin{cases}
	  				\frac{1}{q_1}+\frac{2}{q_2}=\frac{1}{2} \\[4pt]
	  				\frac{1}{r_1}+\frac{2}{r_2}=\frac{1}{2}
	  			\end{cases}
	  			\quad\text{and}\quad
	  			\begin{cases}
	  				\frac{\alpha}{q_1}+\frac{3}{r_1}+\gamma_1=\frac{3}{2} \\[4pt]
	  				\frac{\alpha}{q_2}+\frac{3}{r_2}+\gamma_2=\frac{3}{2}
	  			\end{cases}
	  		\end{equation*}
	  		if and only if 
	  		\begin{equation*}
	  			\gamma_1+2\gamma_2=3-\frac{\alpha}{2},\quad \gamma_1<\frac{3}{2}.
	  		\end{equation*}
	  	\end{lemma}
	  	We still postpone the proof as it is relatively irrelevant to our main discussions.
	  	
	  	From the interpolation in Lemma \ref{interpolation}, we can still assume $N=1$. As before, we just need to bound 
	  	\begin{equation*}
	  		\|I(\phi_1)\cdot \phi_2 \cdot \phi_3\|_{L_{t,x}^{2}([0,T]\times\R^3)}
	  	\end{equation*}
	  	and can split the different frequency interactions into the following cases.
	  	
	  	\vspace{7pt}
	  	$\textbf{Case 1:}$ $\phi_1, \phi_2, \phi_3$ all have high frequency.
	  	
	  	Take $\gamma_1:= \frac{\alpha}{2}^-$ and $\gamma_2:=\frac{1}{2}\left(3-\frac{\alpha}{2}-\gamma_1\right)$. Then, by Lemma \ref{parameter1}, we can find corresponding pairs $(q_1, r_1)$, $(q_2, r_2)$, $(q_3, r_3)$. 
	  	
	  	Applying H\"older's inequality and Sobolev embedding, one can obtain
	  	\begin{equation*}
	  		\|I(\phi_1)\cdot \phi_2 \cdot \phi_3\|_{L_{t,x}^{2}([0,T]\times\R^3)}\le \|I\phi_1\|_{L_t^{q_1}L_x^{r_1}([0,T]\times\R^3)}\prod_{i\ne 1}\|\phi_i\|_{L_t^{q_2}L_x^{r_2}([0,T]\times\R^3)}
	  	\end{equation*}
	  	\begin{equation*}
	  		\lesssim \||D|^{\gamma_1}I\phi_1\|_{L_t^{q_1}L_x^{\frac{3r_1}{3+r_1 \gamma_1 }}([0,T]\times\R^3)}\prod_{i\ne 1}\||D|^{\gamma_2}\phi_i\|_{L_t^{q_2}L_x^{\frac{3r_2}{3+r_2 \gamma_2 }}([0,T]\times\R^3)}
	  	\end{equation*}
	  	\begin{equation*}
	  		\lesssim N_1^{-\left(\frac{\alpha}{2}-\gamma_1\right)}(N_2 N_3)^{-(s-\gamma_2)}\prod_{i=1}^3 \||D|^{\frac{\alpha}{2}}I\phi_i\|_{S_T^{0}}.
	  	\end{equation*}
	  	Since $s>\frac{1}{2}(3-\alpha)$, we have $s-\gamma_2>0$.
	  	
	  	\vspace{7pt}
	  	$\textbf{Case 2:}$ $\phi_1$ has high frequency, $\phi_2, \phi_3$ have low frequency. 
	  	
	  	Take $\gamma_1:=0$ and $\gamma_2:=\frac{6-\alpha}{4}$. Then we similarly can derive 
	  	\begin{equation*}
	  		\|I(\phi_1)\cdot \phi_2 \cdot \phi_3\|_{L_{t,x}^{2}([0,T]\times\R^3)}\lesssim N_1^{-\frac{\alpha}{2}}(N_2N_3)^{\frac{6-3\alpha}{4}}\prod_{i=1}^3 \||D|^{\frac{\alpha}{2}}I\phi_i\|_{S_T^{0}}.
	  	\end{equation*}
	  	Since $\alpha<2$, we see $\frac{6-3\alpha}{4}>0$.
	  	
	  	\vspace{7pt}
	  	$\textbf{Case 3:}$ $\phi_1, \phi_2$ have high frequency, $\phi_3$ has low frequency. 
	  	
	  	Take $\gamma_1:= \frac{\alpha}{2}^+ $ and $\gamma_2:= \frac{1}{2}\left(3-\frac{\alpha}{2}-\gamma_1\right)$. Then we can get
	  	\begin{equation*}
	  		\|I(\phi_1)\cdot \phi_2 \cdot \phi_3\|_{L_{t,x}^{2}([0,T]\times\R^3)}\le \|I\phi_3\|_{L_t^{q_1}L_x^{r_1}([0,T]\times\R^3)}\prod_{i\ne 3}\|\phi_i\|_{L_t^{q_2}L_x^{r_2}([0,T]\times\R^3)}
	  	\end{equation*}
	  	\begin{equation*}
	  		\lesssim N_1^{-\left(\frac{\alpha}{2}-\gamma_2\right)}N_2^{-\left(s-\gamma_2\right)}N_3^{\gamma_1-\frac{\alpha}{2}}\prod_{i=1}^3 \||D|^{\frac{\alpha}{2}}I\phi_i\|_{S_T^{0}}.
	  	\end{equation*}
	  	
	  	\vspace{7pt}
	  	$\textbf{Case 4:}$ $\phi_1$ has low frequency, $\phi_2, \phi_3$ have low frequency. 
	  	
	  	Take $\gamma_1:= \frac{\alpha}{2}^+ $ and $\gamma_2:= \frac{1}{2}\left(3-\frac{\alpha}{2}-\gamma_1\right)$ as before. We can also obtain 
	  	\begin{equation*}
	  		\|I(\phi_1)\cdot \phi_2 \cdot \phi_3\|_{L_{t,x}^{2}([0,T]\times\R^3)}\lesssim N_1^{\gamma_1-\frac{\alpha}{2}}(N_2 N_3)^{-(s-\gamma_2)}\prod_{i=1}^3 \||D|^{\frac{\alpha}{2}}I\phi_i\|_{S_T^{0}},
	  	\end{equation*}
	  	with $\gamma_1-\frac{\alpha}{2}>0,\; s-\gamma_2>0$.
	  	
	  	\vspace{7pt}
	  	$\textbf{Case 5:}$ $\phi_1, \phi_2$ have low frequency, $\phi_3$ has high frequency. 
	  	
	  	Take $\gamma_1:=0$ and $\gamma_2:=\frac{6-\alpha}{4}$:
	  	\begin{equation*}
	  		\|I(\phi_1)\cdot \phi_2 \cdot \phi_3\|_{L_{t,x}^{2}([0,T]\times\R^3)}\lesssim (N_1N_2)^{\frac{6-3\alpha}{4}}N_3^{-s}\prod_{i=1}^3 \||D|^{\frac{\alpha}{2}}I\phi_i\|_{S_T^{0}}.
	  	\end{equation*}
	  	
	  	Thus, we have completed the proof of Proposition \ref{new almost}.
	  	
	  \end{proof}
	  \begin{proof}[Proof of Lemma \ref{parameter1}]
	  	The necessity is trivial. For the sufficiency, we can choose $x$ satisfying
	  	\begin{equation*}
	  		\max\left\lbrace 0, \frac{\alpha+2\gamma_1-3}{4\alpha}\right\rbrace<x<\frac{1}{4}.
	  	\end{equation*}
	  	Then we can set 
	  	\begin{equation}
	  		q_1:=\frac{2}{1-4x}, \quad q_2:=\frac{1}{x},
	  	\end{equation}
	  	\begin{equation*}
	  		y:=\frac{\alpha+\gamma_1-4\alpha x}{12}, \quad r_1:=\frac{2}{1-4y}, \quad r_2:=\frac{1}{y}.
	  	\end{equation*}
	  	One can directly check that the above construction satisfies all the requirements.
	  \end{proof}
	  
	  \vspace{10pt}
	  Now, we can conclude our main result by combining all the propositions we have established so far. The framework follows \cite{19}.
	  
	  \begin{proof}[Proof of Theorem \ref{main radial}]
	  	
	  	Recall the definitions of the scaled solution and its initial data:
	  	\begin{equation*}
	  		u^{\lambda}(x,t):=\frac{1}{\lambda^{\frac{\alpha}{2}}}u\left(\frac{x}{\lambda},\frac{t}{\lambda^{\alpha}}\right), \quad u_0^{\lambda}(x):=\frac{1}{\lambda^{\frac{\alpha}{2}}}u_0\left(\frac{x}{\lambda}\right).
	  	\end{equation*}
	  	
	  	To guarantee that the modified solution $Iu^{\lambda}$ has bounded initial data, we still choose the parameter $\lambda$ as follows:
	  	\begin{equation*}
	  		\lambda\sim N^{\frac{\frac{\alpha}{2}-s}{s+\frac{\alpha}{2}-\frac{3}{2}}}.
	  	\end{equation*}
	  	
	  	Combining with the following calculation
	  	\begin{equation*}
	  		\|Iu_0^{\lambda}\|_{L^4(\R^3)}\lesssim \|u_0^{\lambda}\|_{L^4(\R^3)}=\lambda^{-\frac{1}{2}\left(\alpha-\frac{3}{2}\right)}\|u_0\|_{L^4(\R^3)},
	  	\end{equation*}
	  	we know the initial modified energy $E(Iu_0^{\lambda})\lesssim 1$.
	  	
	  	Next, for any arbitrarily large time $T_0$, we consider the following set:
	  	\begin{equation*}
	  		S:=\left\lbrace 0\le t\le \lambda^{\alpha}T_0:   \|Iu^{\lambda}\|_{L_{t,x}^{p_\alpha}([0,t]\times\R^3)}\le K N^{\theta}\right\rbrace,
	  	\end{equation*} 
	  	where $K$ and $\theta$ are positive constants to be chosen later.
	  	
	  	We claim that $S=[0, \lambda^{\alpha}T_0]$. If not, then from the continuity, there exists $T\in \left(0, \lambda^{\alpha}T_0\right)$ such that 
	  	\begin{equation}\label{5.9}
	  		\|Iu^{\lambda}\|_{L_{t,x}^{p_\alpha}([0,T]\times\R^3)}>KN^{\theta},\quad
	  		\|Iu^{\lambda}\|_{L_{t,x}^{p_\alpha}([0,T]\times\R^3)}\le 2KN^{\theta}.
	  	\end{equation}
	  	
	  	Then one can split $[0,T]$ into sub-intervals $J_k$, $k=1,2,\cdots ,L$, so that 
	  	\begin{equation*}
	  		\|Iu^{\lambda}\|_{L_{t,x}^{p_\alpha}(J_k\times\R^3)}^{p_\alpha}\le \mu_0,
	  	\end{equation*}
	  	where $\mu_0$ is the constant in Proposition \ref{priori}. From the assumption (\ref{5.9}), we can deduce that
	  	\begin{equation*}
	  		L\sim \frac{K^{p_\alpha} N^{p_{\alpha}\theta}}{\mu_0}\sim N^{p_\alpha \theta}.
	  	\end{equation*}
	  	
	  	From the ``almost conservation law'' in Proposition \ref{new almost}, we have, for $s>1-\frac{\alpha}{6}$, 
	  	\begin{equation*}
	  		\sup_{t\in [0,T]} E(Iu^{\lambda}(t))\lesssim E(Iu_0^{\lambda})+\frac{L}{N^{(2\alpha-3)^-}}.
	  	\end{equation*}
	  	To guarantee that 
	  	\begin{equation}\label{zzz}
	  		\sup_{t\in [0,T]}\||D|^{\frac{\alpha}{2}}Iu^{\lambda}(t)\|_{L^2(\R^3)}\lesssim \sup_{t\in [0,T]} E(Iu^{\lambda}(t))\lesssim 1,
	  	\end{equation}
	  	we need to require
	  	\begin{equation}\label{aaa}
	  		N^{p_\alpha \theta}\sim L\lesssim N^{(2\alpha-3)^-} \Longleftrightarrow \;\; p_\alpha \theta <2\alpha-3.
	  	\end{equation}
	  	
	  	\vspace{7pt}
	  	Next, we will apply the Morawetz estimate to contradict our previous assumption (\ref{5.9}). From Corollary \ref{Morawetz}, we have obtained
	  	\begin{equation}\label{ccc}
	  		\int_{0}^{T}\int_{\R^3}|Iu^{\lambda}(x,t)|^{p_\alpha}dxdt\lesssim \||D|^{\frac{\alpha}{2}}Iu^{\lambda}\|_{L_t^{\infty}L_x^{2}([0,T]\times\R^3)}^{1+\frac{2}{3-\alpha}}
	  	\end{equation}
	  	\begin{equation*}
	  		\quad \quad \quad \quad \quad \quad\times \left(\||D|^{\frac{2-\alpha}{2}}Iu^{\lambda}\|_{L_t^{\infty}L_x^{2}([0,T]\times\R^3)}+\||D|^{\frac{2-\alpha}{2}}(u^{\lambda})_{err}\|_{L_t^{1}L_x^{2}([0,T]\times\R^3)}\right).
	  	\end{equation*}
	  	
	  	By interpolation, we see 
	  	\begin{equation*}
	  		\||D|^{\frac{2-\alpha}{2}}Iu^{\lambda}\|_{L_t^{\infty}L_x^{2}([0,T]\times\R^3)}\lesssim \||D|^{\frac{\alpha}{2}}Iu^{\lambda}\|_{L_t^{\infty}L_x^{2}([0,T]\times\R^3)}^{\frac{2-\alpha}{\alpha}}\|Iu^{\lambda}\|_{L_t^{\infty}L_x^{2}([0,T]\times\R^3)}^{\frac{2\alpha-2}{\alpha}}.
	  	\end{equation*}
	  	
	  	Recall (\ref{zzz}) and mass conservation; we can further deduce that
	  	\begin{equation}\label{xxx}
	  		\||D|^{\frac{2-\alpha}{2}}Iu^{\lambda}\|_{L_t^{\infty}L_x^{2}([0,T]\times\R^3)}\lesssim \|Iu^{\lambda}\|_{L_t^{\infty}L_x^{2}([0,T]\times\R^3)}^{\frac{2\alpha-2}{\alpha}}\lesssim \lambda^{\frac{(3-\alpha)(\alpha-1)}{\alpha}}.
	  	\end{equation}
	  	
	  	From Proposition \ref{error term control} and \ref{priori}, we also have 
	  	\begin{equation*}
	  		\||D|^{\frac{2-\alpha}{2}}(u^{\lambda})_{err}\|_{L_t^{1}L_x^{2}(J_k\times\R^3)}\lesssim N^{-(3\alpha-4)^-}, \quad \forall k=1,2, \cdots, L.
	  	\end{equation*}
	  	Then we sum all the sub-intervals $J_k$ up, which yields
	  	\begin{equation}\label{vvv}
	  		\||D|^{\frac{2-\alpha}{2}}(u^{\lambda})_{err}\|_{L_t^{1}L_x^{2}([0,T]\times\R^3)}\lesssim \frac{L}{N^{(3\alpha-4)^-}}\lesssim \frac{N^{(2\alpha-3)^-}}{N^{(3\alpha-4)^-}}\lesssim 1.
	  	\end{equation}
	  	
	  	Now, combining (\ref{zzz}), (\ref{ccc}), (\ref{xxx}) and (\ref{vvv}), we can derive that 
	  	\begin{equation*}
	  		\int_{0}^{T}\int_{\R^3}|Iu^{\lambda}(x,t)|^{p_\alpha}dxdt\lesssim \lambda^{\frac{(3-\alpha)(\alpha-1)}{\alpha}}+1\lesssim N^{\frac{(3-\alpha)(\alpha-1)\left(\frac{\alpha}{2}-s\right)}{\alpha \left(s+\frac{\alpha}{2}-\frac{3}{3}\right)}}.
	  	\end{equation*}
	  	
	  	Thus, we can take $K$ sufficiently large and 
	  	\begin{equation*}
	  		\theta:=\frac{(3-\alpha)(\alpha-1)\left(\frac{\alpha}{2}-s\right)}{\alpha \left(s+\frac{\alpha}{2}-\frac{3}{3}\right)p_\alpha},
	  	\end{equation*}
	  	which contradicts (\ref{5.9}).
	  	
	  	To satisfy the requirement (\ref{aaa}), we can solve 
	  	\begin{equation}\label{better range}
	  		\frac{\alpha}{2}>s>\frac{\alpha(3-\alpha)(3\alpha-4)}{2(\alpha^2+\alpha-3)}.
	  	\end{equation}
	  	
	  	Finally, from $S=[0, \lambda^\alpha T_0]$ and the above argument, we see
	  	\begin{equation}\label{bbb}
	  		\sup_{t\in [0, \lambda^{\alpha}T_0]}\||D|^{\frac{\alpha}{2}}Iu^{\lambda}(t)\|_{L^2(\R^d)}\lesssim 1.
	  	\end{equation}
	  	
	  	With the bound (\ref{bbb}), one can eventually derive the following uniform estimate on the Sobolev norm, which directly implies our desired global well-posedness.
	  	
	  	\begin{equation*}
	  		\sup_{t\in [0,T_0]}\|u(t)\|_{H^s(\R^3)}\lesssim \sup_{t\in [0,T_0]}\|u(t)\|_{L^2(\R^3)}+\sup_{t\in [0, \lambda^{\alpha}T_0]}\|u(t)\|_{\dot{H}^s(\R^3)}
	  	\end{equation*}
	  	\begin{equation*}
	  		=\|u_0\|_{L^2(\R^3)}+\lambda^{s+\frac{\alpha}{2}-\frac{3}{2}}\sup_{t\in [0,\lambda^\alpha T_0]}\|u^{\lambda}(t)\|_{\dot{H}^s(\R^3)}
	  	\end{equation*}
	  	\begin{equation*}
	  		\lesssim \lambda^{s+\frac{\alpha}{2}-\frac{3}{2}}\left(1+\sup_{t\in [0,\lambda^\alpha T_0]}\|Iu^{\lambda}(t)\|_{L^2(\R^3)}\right)\lesssim \lambda^s\sim N^{\frac{s\left(\frac{\alpha}{2}-s\right)}{s+\frac{\alpha}{2}-\frac{3}{2}}}.
	  	\end{equation*}
	  	Note that $N$ is independent of $T_0$ and the latter can be arbitrarily large; we can conclude our desired Theorem \ref{main radial}.
	  \end{proof}
	  
	  \vspace{5pt}
	  Next, we can establish the scattering results in Theorem \ref{main scattering}. For convenience, we only check the scattering when $t\to +\infty$.
	  
	  Following the same calculation as in Lemma \ref{virial}, we can derive the following Morawetz estimate for the solution $u$, which has been obtained in \cite{10}.
	  
	  \begin{lemma}\label{old virial}
	  	Suppose $u\in C([0,T];\mathcal{S}(\R^3))$ is a solution to $(\ref{FNLS})$. Then for 
	  	\begin{equation*}
	  		\widetilde{M}_{\varphi}(t):=2\Im m \int_{\R^3}\overline{u}\nabla{u}\cdot \nabla \varphi dx,
	  	\end{equation*}
	  	
	  	\vspace{-4pt}
	  	we have
	  	\begin{equation*}
	  		\frac{d \widetilde{M}_{\varphi}}{dt}=\Re e\int_{0}^{\infty}\lambda^{\frac{\alpha}{2}}d\lambda \int_{\R^3}(4\partial_j\partial_k\varphi \partial_j\overline{u_\lambda}\partial_ku_\lambda-\Delta^{2}\varphi |u_\lambda|^2 )dx
	  		+\int_{\R^3}\Delta\varphi|u|^{4}dx.
	  	\end{equation*}
	  \end{lemma}
	  
	  \vspace{6pt}
	  Let $\varphi:=|x|$ and apply Lemma \ref{momentum} and \ref{radial sobolev}; one can obtain 
	  \begin{equation*}
	  	\int_{0}^{\infty}\int_{\R^3}|u(t,x)|^{4+\frac{2}{3-2s}}dxdt\lesssim  \sup_{t\in [0,+\infty)}\|u(t)\|_{H^s(\R^3)}^{1+\frac{2}{3-2s}}  \sup_{t\in [0,+\infty)}\|u(t)\|_{H^{1-s}(\R^3)}.
	  \end{equation*}
	  
	  Recall that $s>\frac{1}{2}$ and the $H^s$-norm is uniformly bounded in time; we can further deduce the following uniform control:
	  \begin{equation}\label{k}
	  	\|u\|_{L_{t,x}^{p_s}([0,+\infty)\times\R^3)}\le C(\|u_0\|_{H^s(\R^3)}), \quad p_s:=4+\frac{2}{3-2s}.
	  \end{equation}
	  
	  Next, we need an analogue of Proposition \ref{priori}.
	  \begin{prop}\label{priori1}
	  	For a space-time slab $[0,T]\times\R^3$, we suppose 
	  	\begin{equation*}
	  		\int_{0}^T\int_{\R^3}|u(t,x)|^{p_s}dxdt<\widetilde{\mu}_0,
	  	\end{equation*}
	  	where $\widetilde{\mu}_0=\widetilde{\mu}_0(\|u_0\|_{H^s})$ is a small constant, and $u\in C([0,T]; \mathcal{S}(\R^3))$ is a solution to $(\ref{FNLS})$. Then we have
	  	\begin{equation*}
	  		Z([0,T]):=\sup_{(q,r)\; \text{is}\; \alpha-\text{admissible} }\|\langle D\rangle^{s}u\|_{L_t^{q}L_x^{r}([0,T]\times\R^3)}\lesssim \|\langle D\rangle^{s}u_0\|_{L^{2}(\R^3)}.
	  	\end{equation*}
	  \end{prop}
	  \begin{proof}
	  	For $s, \alpha$ in Theorem \ref{main radial}, one can determine a unique $\theta\in (0,1)$ such that
	  	\begin{equation*}
	  		\frac{\alpha}{2}=\theta\cdot(\frac{3}{2}-s)+(1-\theta)\cdot\frac{\alpha+3}{p_s}.
	  	\end{equation*}
	  	Simple calculations show that 
	  	\begin{equation*}
	  		\left(0,\frac{1}{2}\right)\cap \left(\frac{1-\theta}{p_s},\; \theta\cdot \frac{3-2s}{2\alpha}+(1-\theta)\cdot\frac{1}{p_s}\right)\ne \emptyset,
	  	\end{equation*}
	  	which implies the existence of $q_3, q_4$ satisfying
	  	\begin{equation*}
	  		0<\frac{1}{q_3}<\frac{1}{2}, \quad 0<\frac{1}{q_4}<\frac{3-2s}{2\alpha},\quad 
	  		\frac{1}{q_3}=\theta\cdot \frac{1}{q_4}+(1-\theta)\cdot\frac{1}{p_s}.
	  	\end{equation*}
	  	Then we can set 
	  	\begin{equation*}
	  		\frac{1}{r_3}:=\frac{\alpha}{3}\left(\frac{1}{2}-\frac{1}{q_3}\right), \quad\frac{1}{r_4}:=\frac{\alpha}{3}\left(\frac{1}{2}-\frac{1}{q_4}\right).
	  	\end{equation*}
	  	Now one can apply Lemma \ref{parameter} with 
	  	\begin{equation*}
	  		a:=1-\frac{2}{q_3}, \quad b:=1-\frac{2}{r_3},
	  	\end{equation*}
	  	which ensures the existence of $\alpha$-admissible pairs $(q_1, r_1)$, $(q_2, r_2)$ satisfying
	  	\begin{equation*}
	  		\frac{1}{q_1'}=\frac{1}{q_2}+\frac{2}{q_3}, \quad \frac{1}{r_1'}=\frac{1}{r_2}+\frac{2}{r_3}.
	  	\end{equation*}
	  	Applying the Strichartz estimate in Lemma \ref{Strichartz2} and the fractional Leibniz rule, we can obtain
	  	\begin{equation*}
	  		Z([0,T])\lesssim \|u_0\|_{H^{s}(\R^3)}+\|\langle D\rangle^{s}(|u|^2u)\|_{L_t^{q_1'}L_x^{r_1'}([0,T]\times\R^3)}
	  	\end{equation*}
	  	\begin{equation*}
	  		\lesssim \|u_0\|_{H^{s}(\R^3)}+\|\langle D\rangle^{s}u\|_{L_t^{q_2}L_x^{r_2}([0,T]\times\R^3)}\|u\|_{L_t^{q_3}L_x^{r_3}([0,T]\times\R^3)}^2.
	  	\end{equation*}
	  	From the Gagliardo-Nirenberg inequality, one also has
	  	\begin{equation*}
	  		\|u\|_{L_t^{q_3}L_x^{r_3}([0,T]\times\R^3)}\lesssim \|\langle D\rangle^{s}u\|_{L_t^{q_4}L_x^{r_4}([0,T]\times\R^3)}^\theta \|u\|_{L_{t,x}^{p_s}([0,T]\times\R^3)}^{1-\theta},
	  	\end{equation*}
	  	and then we can deduce that 
	  	\begin{equation*}
	  		Z([0,T])\lesssim \|u_0\|_{H^{s}(\R^3)}+ Z([0,T])^{1+2\theta}\cdot \widetilde{\mu_0}^{2-2\theta}.
	  	\end{equation*}
	  	Since $\widetilde{\mu_0}$ is sufficiently small, we conclude Proposition \ref{priori1}.
	  \end{proof}
	  
	  Now, the scattering in Theorem \ref{main scattering} follows immediately:
	  \begin{proof}[Proof of Theorem \ref{main scattering}]
	  	 Since the uniform control (\ref{k}) holds, we can decompose $[0,+\infty)$ into finitely many sub-intervals $J_k$, $k=1,2,\dots,m$, such that 
	  	 \begin{equation*}
	  	 	\int_{J_k}\int_{\R^3}|u(t,x)|^{p_s}dxdt<\widetilde{\mu}_0,\quad \forall k=1,2,\cdots,m.
	  	 \end{equation*}
	  	 Then applying Proposition \ref{priori1} on each $J_k$ and summing up, we can derive another uniform estimate:
	  	 \begin{equation*}
	  	 	Z([0,+\infty))=\sup_{(q,r)\; \text{is}\; \alpha-\text{admissible} }\|\langle D\rangle^{s}u\|_{L_t^{q}L_x^{r}([0,+\infty)\times\R^3)}\le C(\|u_0\|_{H^s(\R^3)}).
	  	 \end{equation*}
	  	 Note that the asymptotic completeness is equivalent to 
	  	 \begin{equation*}
	  	 	\lim_{t\to \infty} \left\|\int_{t}^{\infty}\langle D\rangle^s e^{i(t-\tau)|D|^\alpha}(|u(\tau)|^2u(\tau))d\tau\right\|_{L^2(\R^3)}=0.
	  	 \end{equation*}
	  	 Using Strichartz estimate starting at infinity (see Corollary 3.2.7 in \cite{37}), one can directly obtain
	  	 \begin{equation*}
	  	 \left\|\int_{t}^{\infty}\langle D\rangle^s e^{i(t-\tau)|D|^\alpha}(|u(\tau)|^2u(\tau))d\tau\right\|_{L^2(\R^3)}\lesssim \|\langle D\rangle^s (|u|^2u)\|_{L_t^{q_1'}L_x^{r_1'}([t,+\infty)\times\R^3)}
	  	 \end{equation*}
	  	 \begin{equation*}
	  	 	\lesssim \|\langle D\rangle^s u\|_{L_t^{q_2}L_x^{r_2}([t,+\infty)\times\R^3)}\| u\|_{L_t^{q_3}L_x^{r_3}([t,+\infty)\times\R^3)}^2
	  	 \end{equation*}
	  	 \begin{equation*}
	  	 	\lesssim Z([0,+\infty))^{1+2\theta} \|u\|_{L_{t,x}^{p_s}([t,+\infty)\times\R^3)}^{2-2\theta}\to 0,
	  	 \end{equation*}
	  	 where the parameters $q_i,r_i, \theta$ come from the proof of Proposition \ref{priori1}. Thus, we have proven the surjectivity of wave operator $\Omega_{+}$.
	  	 
	  	 For completeness, we show the existence of wave operator, which is equal to finding a fixed point of the following map:
	  	 \begin{equation*}
	  	 	\Phi(u):=e^{-it|D|^{\alpha}}u_{+}+i\int_{t}^{\infty}e^{-i(t-\tau)|D|^{\alpha}}(|u(\tau)|^2 u(\tau))d\tau.
	  	 \end{equation*}
	  	 
	  	 We first choose $0<\frac{1}{q_4}<\frac{1}{2}$ and set
	  	 \begin{equation*}
	  	 	\frac{1}{r_3}:=\frac{\alpha}{3}\left(\frac{1}{2}-\frac{1}{q_4}\right), \quad \frac{1}{r_4}:=\frac{1}{3}\left(\frac{3}{2}-\frac{\alpha}{q_4}\right). 
	  	 \end{equation*}
	  	 Applying Lemma \ref{parameter} again with 
	  	 \begin{equation*}
	  	 	a:=1-\frac{2}{q_4}, \quad b:=1-\frac{2}{r_3},
	  	 \end{equation*}
	  	 we can get corresponding $\alpha$-admissible pairs $(q_1,r_1),$ $(q_2,r_2)$.
	  	 
	  	 Since $(q_2,r_2)$ and $(q_4, r_4)$ are all $\alpha$-admissible, we can choose $T_0$ large enough so that
	  	 \begin{equation*}
	  	 	\varepsilon_0:=\left\|e^{-it|D|^\alpha}u_{+}\right\|_{L_t^{q_2}W_x^{s,r_2}([T_0,+\infty)\times\R^3)}+\left\|e^{-it|D|^\alpha}u_{+}\right\|_{L_t^{q_4}W_x^{s,r_4}([T_0,+\infty)\times\R^3)}
	  	 \end{equation*}
	  	 is sufficiently small.
	  	 
	  	 Then we consider the following set
	  	\[
	  	X:=\left\{
	  	u\in L_t^{q_2}W_x^{s,r_2}\cap L_t^{q_4}W_x^{s,r_4}
	  	:
	  	\|u\|_{L_t^{q_2}W_x^{s,r_2}([T_0,+\infty)\times\R^3)}
	  	+
	  	\|u\|_{L_t^{q_4}W_x^{s,r_4}([T_0,+\infty)\times\R^3)}
	  	\le 2\varepsilon_0
	  	\right\},
	  	\]
	  	endowed with the norm $\|\cdot\|_X:=\|\cdot\|_{L_t^{q_2}W_x^{s,r_2}}+\|\cdot\|_{L_t^{q_4}W_x^{s,r_4}}$.
	  	 
	  	By Strichartz estimate and Sobolev embedding, we can derive
	  	\begin{equation*}
	  		\|\Phi(u)\|_X\le \varepsilon_0 + C\|\langle D\rangle^s (|u|^2u)\|_{L_t^{q_1'}L_x^{r_1'}([T_0,+\infty)\times\R^3)}
	  		\end{equation*}
	  		\begin{equation*}
	  		\le \varepsilon_0+ C\|\langle D\rangle^s u\|_{L_t^{q_2}L_x^{r_2}([T_0,+\infty)\times\R^3)}\| u\|_{L_t^{q_4}L_x^{r_3}([T_0,+\infty)\times\R^3)}^2
	  		\end{equation*}
	  		\begin{equation*}
	  			\le \varepsilon_0+ C\|\langle D\rangle^s u\|_{L_t^{q_2}L_x^{r_2}([T_0,+\infty)\times\R^3)}\|\langle D\rangle^s u\|_{L_t^{q_4}L_x^{r_4}([T_0,+\infty)\times\R^3)}^2\le \varepsilon_0+C(2\varepsilon_0)^3.
	  		\end{equation*}
	  		Taking $\varepsilon_0\ll1$, we see $\Phi$ maps $X$ to $X$. Similarly, one can prove $\Phi$ is also a contraction on $X$, which implies the existence of fixed point. 
	  		
	  		From the global well-posedness and time-reversibility, we are allowed to extend such $u$ to $[0,+\infty)$, which guarantees the existence of wave operator $\Omega_{+}$.
	  \end{proof}

	  \vspace{15pt}
	  \section{Analogous Results for $\alpha>2$}
	  In this section, we briefly discuss several analogous results for the fractional nonlinear Schr\"odinger equation in the regime $\alpha>2$. Although the dispersive behavior differs substantially from the case $\alpha\in(1,2)$, many arguments developed in the previous sections still apply, with suitable modifications. 
	  
	  For simplicity, we still focus on $d=2,3$. Also recall that when the equation is not $L^2$-supercritical, i.e., $\frac{d-\alpha}{2}\le 0$, we can choose an ``inhomogeneous'' framework, which can greatly simplify our argument and even yield better estimates.
	  
	  Thus, we will divide our results into three cases: $(d=2,\; \alpha>2)$, $(d=3, \; 2<\alpha<3)$, and $(d=3,\; \alpha\ge3)$. Similarly, for higher dimensions $d>3$, our arguments still work with some necessary modifications, and we leave them to interested readers.
	  
	  \vspace{7pt}
	  Since our bilinear estimates (\ref{e}) and (\ref{ee}) still work for $\alpha>2$, our proof in Proposition \ref{LWP} can be directly applied to obtain the following local well-posedness theory.
	  
	  \begin{prop}\label{LWP1}
	  	Let $d>\alpha$, $\frac{\alpha}{2}>s>\frac{d-\alpha}{2}$, and the initial data for the modified fractional Schr\"{o}dinger equation (\ref{IFNLS}) satisfies $\||D|^{\frac{\alpha}{2}}Iu_0\|_{L^{2}(\R^d)}\le M$. Then there exists a constant $T=T(M)>0$ such that the equation is locally well-posed on the time interval $[0,T]$, with the following bound:
	  	\begin{equation*}
	  		\||D|^{\frac{\alpha}{2}}Iu\|_{X^{0,\frac{1}{2}^{+}}([0,T]\times\R^d)}\lesssim \||D|^{\frac{\alpha}{2}}Iu_0\|_{L^{2}(\R^d)}.
	  	\end{equation*}
	  \end{prop}
	  
	  \begin{prop}\label{LWP2}
	  	Let $\alpha> d$, $\frac{\alpha}{2}>s\ge0$, and the initial data for the modified fractional Schr\"{o}dinger equation (\ref{IFNLS}) satisfies $\|Iu_0\|_{H^{\frac{\alpha}{2}}(\R^d)}\le M$. Then there exists a constant $T=T(M)>0$ such that the equation is locally well-posed on the time interval $[0,T]$, with the following bound:
	  	\begin{equation*}
	  		\|Iu\|_{X^{\frac{\alpha}{2},\frac{1}{2}^{+}}([0,T]\times\R^d)}\lesssim \|Iu_0\|_{H^{\frac{\alpha}{2}}(\R^d)}.
	  	\end{equation*}
	  	Moreover, the above results also hold for $\alpha=d$ and $\frac{\alpha}{2}>s>0$.
	  \end{prop}
	  
	  For $d=2, \alpha>2$, we can establish the following ``almost conservation law'':
	  \begin{prop}\label{5.14}
	  	For $d=2$, $\alpha>2$, and $s>\max\lbrace0^-, 1-\frac{\alpha}{4}\rbrace$, we suppose $u$ is the solution in Proposition \ref{LWP2}; then the following estimate holds:
	  	\begin{equation*}
	  		E(Iu)(t)-E(Iu)(0)=O(N^{-\min\left\lbrace \alpha, \frac{3}{2}(\alpha-1)\right\rbrace^{-}}), \quad \forall t\in [0,T].
	  	\end{equation*}
	  \end{prop}
\begin{proof}
	We only choose several representative cases to illustrate the origin of the decay rate
	\[
	N^{-\min\left\{\alpha,\frac{3}{2}(\alpha-1)\right\}^{-}},
	\]
	and some necessary modifications.
	
	\vspace{7pt}
	$\textbf{Term}_1$, \textbf{Case 1:} $N_1\sim N_2\gtrsim N\gg 1\ge N_3\ge N_4.$
	
	Applying the bilinear estimate (\ref{e}) and the mean value theorem, one can derive
	\begin{equation*}
		(\ref{term1})\lesssim \frac{N_3}{N_2}\cdot \frac{N_3^{\frac{1}{2}}}{N_1^{\frac{\alpha-1}{2}}}\cdot\frac{N_4^{\frac{1}{2}}}{N_2^{\frac{\alpha-1}{2}}}\cdot \frac{N_1^{\frac{\alpha}{2}}}{N_2^{\frac{\alpha}{2}}\langle N_3\rangle^{\frac{\alpha}{2}}\langle N_4\rangle^{\frac{\alpha}{2}}}\|\phi_1\|_{X^{-\frac{\alpha}{2}, \frac{1}{2}^{+}}}\|\phi_2\|_{X^{\frac{\alpha}{2}, \frac{1}{2}^{+}}}\|\phi_3\|_{X^{\frac{\alpha}{2}, \frac{1}{2}^{+}}}\|\phi_4\|_{X^{\frac{\alpha}{2}, \frac{1}{2}^{+}}}.
	\end{equation*}
	Simple calculations show that
	\begin{equation*}
		\frac{N_3}{N_2}\cdot \frac{N_3^{\frac{1}{2}}}{N_1^{\frac{\alpha-1}{2}}}\cdot\frac{N_4^{\frac{1}{2}}}{N_2^{\frac{\alpha-1}{2}}}\cdot \frac{N_1^{\frac{\alpha}{2}}}{N_2^{\frac{\alpha}{2}}\langle N_3\rangle^{\frac{\alpha}{2}}\langle N_4\rangle^{\frac{\alpha}{2}}}\lesssim N^{-\alpha^-}N_2^{0^-}N_3^{0^+} N_4^{0^+}.
	\end{equation*}
	
	\vspace{7pt}
	$\textbf{Term}_1$, \textbf{Case 5:} $N_2\sim N_3\gtrsim N\gg 1\ge N_4 \; \&\; N_1\ge 1.$ 
	
	Similarly, it suffices to control
	\begin{equation*}
		\frac{1}{m(N_2)m(N_3)}\cdot \frac{N_1^{\frac{1}{2}}}{N_2^{\frac{\alpha-1}{2}}}\cdot\frac{N_4^{\frac{1}{2}}}{N_3^{\frac{\alpha-1}{2}}}\cdot \frac{ N_1^{\frac{\alpha}{2}}}{N_2^{\frac{\alpha}{2}} N_3^{\frac{\alpha}{2}}\langle N_4\rangle^{\frac{\alpha}{2}}}
	\end{equation*}
	
	In this case, we actually obtain the bound $N^{-\frac{3}{2}(\alpha-1)^-}N_2^{0^-}N_4^{0^+}$.
	
	\vspace{7pt}
	For $\textbf{Term}_2$, we claim the following stronger estimate:
	\begin{equation}\label{b}
		\|P_{N_{123}}I(\phi_1\phi_2\phi_3)\|_{L_{t,x}^2([0,T]\times\R^2)}\lesssim \langle N_{123}\rangle^{-\frac{\alpha}{2}^-}N_1^{0\pm}N_2^{0\pm}N_3^{0\pm}\prod_{i=1}^3\|I\phi_i\|_{X^{\frac{\alpha}{2},\frac{1}{2}^+}}.
	\end{equation}
	
	Assuming this estimate first, we can derive:
	
	$\textbf{Term}_2$, \textbf{Case 1:} $N_{123}\sim N_4\gtrsim N\gg 1\ge N_5\ge N_6.$ 
	
	Applying  H\"{o}lder's inequality with the factors in $L_{t,x}^2$, $L_t^{2^+}L_x^{\infty}$, $L_t^{\infty}L_x^{4}$, and $L_t^{\infty}L_x^{4}$, one can obtain
	\begin{equation*}
		(\ref{term2})\lesssim_T  \langle N_{123}\rangle^{-\frac{\alpha}{2}}N_1^{0\pm}N_2^{0\pm}N_3^{0\pm}N_4^{-(\alpha-1)}\frac{N_5^\frac{1}{2}}{\langle N_5\rangle^{\frac{\alpha}{2}}}\frac{N_6^\frac{1}{2}}{\langle N_6\rangle^{\frac{\alpha}{2}}}\prod_{i=1}^6\|I\phi_i\|_{X^{\frac{\alpha}{2},\frac{1}{2}^+}}
	\end{equation*}
	\begin{equation*}
		\lesssim N_1^{0\pm}N_2^{0\pm}N_3^{0\pm}N^{-\frac{1}{2}(3\alpha-2)^-}N_4^{0^-}N_5^{0^+}N_6^{0^+}\prod_{i=1}^6\|I\phi_i\|_{X^{\frac{\alpha}{2},\frac{1}{2}^+}}.
	\end{equation*}
	
	\vspace{7pt}
	$\textbf{Term}_2$, \textbf{Case 4:} $N_{4}\sim N_5\gtrsim N \;\&\; N_{123}\gg 1$. 
	
	Applying  H\"{o}lder's inequality with the factors in $L_{t,x}^2$, $L_{t,x}^{4}$, $L_{t,x}^{4}$, and $L_{t,x}^{\infty}$, we can similarly deduce the following estimate
	\begin{equation*}
		(\ref{term2})\lesssim_T \frac{N_1^{0\pm}N_2^{0\pm}N_3^{0\pm}}{m(N_4)m(N_5)m(N_6)}N_4^{-\frac{1}{4}(3\alpha-2)}N_5^{-\frac{1}{4}(3\alpha-2)}\frac{N_6}{\langle N_6\rangle^{\frac{\alpha}{2}}} \prod_{i=1}^6\|I\phi_i\|_{X^{\frac{\alpha}{2},\frac{1}{2}^+}}
	\end{equation*}
	\begin{equation*}
		\lesssim N_1^{0\pm}N_2^{0\pm}N_3^{0\pm}N^{-\frac{1}{2}(3\alpha-2)^-}N_4^{0^-}N_6^{0\pm}\prod_{i=1}^6\|I\phi_i\|_{X^{\frac{\alpha}{2},\frac{1}{2}^+}}.
	\end{equation*}
	
	\vspace{7pt}
	To prove our claim (\ref{b}), we can play the same trick as before. In fact, we can assume $N=1$ and reduce it to 
	\begin{equation*}
		\|\langle \nabla\rangle^{\frac{\alpha}{2}^-} I(\phi_1)\cdot\phi_2\cdot\phi_3\|_{L_{t,x}^2([0,T]\times\R^2)}\lesssim  N_1^{0\pm}N_2^{0\pm}N_3^{0\pm}\prod_{i=1}^3\|I\phi_i\|_{X^{\frac{\alpha}{2},\frac{1}{2}^+}}.
	\end{equation*}
	For simplicity, we just check the high frequency case.
	
	\textbf{Case 2:} $\phi_1, \phi_2, \phi_3$ all have high frequency.
	\begin{equation*}
		\|\langle \nabla\rangle^{\frac{\alpha}{2}^-} I(\phi_1)\cdot\phi_2\cdot\phi_3\|_{L_{t,x}^2([0,T]\times\R^2)}\lesssim \|\langle \nabla\rangle^{\frac{\alpha}{2}^-} I(\phi_1)\|_{L_t^{\infty}L_x^{2}([0,T]\times\R^2)}\|\phi_2\|_{L_t^{4}L_x^{\infty}([0,T]\times\R^2)}\|\phi_3\|_{L_t^{4}L_x^{\infty}([0,T]\times\R^2)}
	\end{equation*}
	\begin{equation*}
		\lesssim N_1^{0^{-}}N_2^{-(s-1+\frac{\alpha}{4})}N_3^{-(s-1+\frac{\alpha}{4})}\prod_{i=1}^3\|I\phi_i\|_{X^{\frac{\alpha}{2},\frac{1}{2}^+}}.
	\end{equation*}
	
\end{proof}
\begin{rem}
	The condition $s>1-\frac{\alpha}{4}$ can in fact be slightly weakened by choosing different parameters in the analysis of $\textbf{Term}_2$. Since this lower bound does not affect our final results, we do not pursue optimality here.
\end{rem}

\vspace{7pt}
Next we can directly derive the following global well-posedness:
\begin{thm}\label{main'}
	Let $d=2$, $\alpha>2$, and $s$ satisfying
	\begin{equation*}
		\begin{cases}
			\frac{\alpha}{2}>s>\frac{\alpha}{2}-\frac{3(\alpha-1)^2}{5\alpha-3}, \quad   2<\alpha\le 3\\[4pt]
			\frac{\alpha}{2}>s>\frac{1}{2}, \quad \quad \qquad \quad \;\;\alpha>3.
		\end{cases}
	\end{equation*}
	Then the fractional Schr\"{o}dinger equation (\ref{FNLS}) is globally well-posed in $H^{s}(\R^2)$. 
\end{thm}
	  
	 In the same spirit, we can establish the following ``almost conservation law'' and global well-posedness for $d=3, \; \alpha\ge 3$.
	 
	 \begin{prop}\label{5.15}
	 	For $d=3$, $\alpha\ge 3$, and $s>\max\lbrace0^-, \frac{3}{2}-\frac{\alpha}{4}\rbrace$, we suppose $u$ is the solution in Proposition \ref{LWP2}; then the following estimate holds:
	 	\begin{equation*}
	 		E(Iu)(t)-E(Iu)(0)=O(N^{-\min\left\lbrace \alpha, \frac{1}{2}(3\alpha-4)\right\rbrace^{-}}), \quad \forall t\in [0,T].
	 	\end{equation*}
	 \end{prop}

	 \begin{thm}\label{main''}
	 	Let $d=3$, $\alpha\ge 3$, and $s$ satisfying
	 	\begin{equation*}
	 		\begin{cases}
	 			\frac{\alpha}{2}>s>\frac{\alpha}{2}-\frac{(3\alpha-4)(2\alpha-3)}{2(5\alpha-4)}, \quad   3\le\alpha\le 4\\[4pt]
	 			\frac{\alpha}{2}>s>\frac{3}{4}, \quad \quad \qquad \qquad \;\;\quad \alpha>4.
	 		\end{cases}
	 	\end{equation*}
	 	Then the fractional Schr\"{o}dinger equation (\ref{FNLS}) is globally well-posed in $H^{s}(\R^3)$. 
	 \end{thm}
	 
	 \vspace{7pt}
	 Now, it remains to deal with the case: $d=3, \; 2<\alpha<3$. We will follow the proof of Proposition \ref{d=2} in Section 5.
	 
	 \begin{prop}\label{5.151}
	 	For $d=3$, $2<\alpha<3$, and $s>1-\frac{\alpha}{6}$, we suppose $u$ is the solution in Proposition \ref{LWP1}; then the following estimate holds:
	 	\begin{equation*}
	 		E(Iu)(t)-E(Iu)(0)=O(N^{-\frac{1}{2}(-\alpha^2+6\alpha-6)^-}), \quad \forall t\in [0,T].
	 	\end{equation*}
	 \end{prop}
	 \begin{proof}
	 	We still choose some typical cases to illustrate the origin of such decay rate and some new parameters.
	 	
	 	\vspace{7pt}
	 	$\textbf{Term}_1$, \textbf{Case 2:} $N_1\sim N_2\gtrsim N\gg N_3\ge 1\ge N_4.$
	 	
	 	Interpolating the bilinear estimate \eqref{bilinear2} with the Strichartz estimate \eqref{Strichartz3}, we obtain
	 	\begin{align*}
	 		\|\phi_2\phi_4\|_{L_{x,t}^{2}([0,T]\times \R^3)}
	 		&\lesssim \frac{N_4}{N_2^{\frac{\alpha-1}{2}}}
	 		\, N_2^{-\frac{\alpha}{2}} N_4^{-\frac{\alpha}{2}}
	 		\||D|^{\frac{\alpha}{2}}\phi_2\|_{X^{0, \frac{1}{2}^{+}}}
	 		\||D|^{\frac{\alpha}{2}}\phi_4\|_{X^{0, \frac{1}{2}^{+}}}, \\
	 		\|\phi_2\phi_4\|_{L_{x,t}^{2}([0,T]\times \R^3)}
	 		&\lesssim N_2^{-\frac{\alpha}{2}} N_4^{\frac{3-\alpha}{2}}
	 		\||D|^{\frac{\alpha}{2}}\phi_2\|_{X^{0, \frac{1}{2}^{+}}}
	 		\||D|^{\frac{\alpha}{2}}\phi_4\|_{X^{0, \frac{1}{2}^{+}}}.
	 	\end{align*}
	 	Interpolating between these two bounds yields
	 	\begin{equation*}
	 		\|\phi_2\phi_4\|_{L_{x,t}^{2}([0,T]\times \R^3)}
	 		\lesssim N_2^{-\frac{1}{2}(-\alpha^2+5\alpha-3)^{-}}
	 		\, N_4^{0^{+}}
	 		\||D|^{\frac{\alpha}{2}}\phi_2\|_{X^{0, \frac{1}{2}^{+}}}
	 		\||D|^{\frac{\alpha}{2}}\phi_4\|_{X^{0, \frac{1}{2}^{+}}}.
	 	\end{equation*}
	 	Then one can derive
	 	\begin{equation*}
	 		\frac{N_3}{N_1}\cdot \frac{N_3}{N_1^{\frac{\alpha-1}{2}}}\cdot \frac{N_1^{\frac{\alpha}{2}}}{N_3^{\frac{\alpha}{2}}}\cdot N_2^{-\frac{1}{2}(-\alpha^2+5\alpha-3)^{-}}
	 		\cdot N_4^{0^{+}}\lesssim N^{-\frac{1}{2}(-\alpha^2+6\alpha-6)^{-}}N_1^{0^{-}}N_4^{0^{+}}.
	 	\end{equation*}
	 	
	 	\vspace{7pt}
	 	$\textbf{Term}_2$, \textbf{Case 2:} $N_{123}\sim N_4\gtrsim N\gg 1\ge N_5\ge N_6.$
	 	
	 	Applying H\"{o}lder's inequality with the factors in $L_{t,x}^2$, $L_t^{\frac{2^+}{3-\alpha}}L_x^{\frac{6^-}{2\alpha-3}}$, $L_t^{\infty}L_x^{\frac{6^+}{3-\alpha}}$, and $L_t^{\infty}L_x^{\frac{6^+}{3-\alpha}}$, we can obtain 
	 	\begin{equation*}
	 		(\ref{term2})\lesssim N_1^{0\pm}N_2^{0\pm}N_3^{0\pm}N_4^{-\frac{1}{2}(-\alpha^2+6\alpha-6)^{-}}N_5^{0^+}N_6^{0^+}\prod_{i=1}^6\||D|^{\frac{\alpha}{2}}I\phi_i\|_{X^{0,\frac{1}{2}^+}}.
	 	\end{equation*}
	 	Note that Lemma \ref{1234} still holds for $N_{123}\gg 1$ when $2<\alpha<3$; the remaining cases can be handled by following the argument in Proposition \ref{d=2}.
	 \end{proof}
	 
	 \vspace{7pt}
	 Consequently, we have the following global well-posedness:
	 \begin{thm}\label{main'''}
	 	Let $d=3$, $2<\alpha<3$, and $s$ satisfying
	 	\begin{equation*}
	 		\frac{\alpha}{2}>s>\frac{\alpha}{2}-\frac{(-\alpha^2+6\alpha-6)(2\alpha-3)}{2(-\alpha^2+8\alpha-6)}.
	 	\end{equation*}
	 	Then the fractional Schr\"{o}dinger equation (\ref{FNLS}) is globally well-posed in $H^{s}(\R^3)$. 
	 \end{thm}
	  
	   \vspace{15pt}
	   \section{Polynomial growth of higher-order Sobolev norms}
	   The upside-down $I$-method was first introduced in \cite{21}, but in the low regularity regime (below the energy threshold). Then Sohinger applied this method to control the evolution of higher-order Sobolev norms (see \cite{22,23}).
	   
	   For simplicity, we will go back to our original cases: $(d=2,\; 1<\alpha<2)$ and $(d=3,\; \frac{3}{2}<\alpha<2)$. Some modifications may yield similar results for $\alpha>2$.
	   
	   \vspace{7pt}
	   We first recall the setting of the upside-down $I$-method: given $s>\frac{\alpha}{2}$ and a parameter $N\gg 1$, we define the $I$-operator $I_N$ as follows.
	   \begin{equation*}
	   	\widehat{I_N f}(\xi):=m_N(\xi) \widehat{f}(\xi),
	   \end{equation*}
	   where the multiplier $m_N(\xi)$ is smooth, radially symmetric, and 
	   \begin{equation*}
	   	m_N(\xi):=
	   	\begin{cases}
	   		\;\;1\quad\;\;\;\;\;\;\;,  \quad |\xi|\le N\\[4pt]
	   		\left(\frac{|\xi|}{N}\right)^{s-\frac{\alpha}{2}}, \quad |\xi|\ge 2N.
	   	\end{cases}
	   \end{equation*}
	   
	   We have the following important relations between $\|\phi\|_{H^s(\R^d)}$ and $\|I_N \phi\|_{H^{\frac{\alpha}{2}}(\R^d)}$:
	   \begin{equation*}
	   	\|I_N \phi\|_{H^{\frac{\alpha}{2}}(\R^d)}\lesssim \|\phi\|_{H^s(\R^d)}.
	   \end{equation*}
	   \begin{equation}\label{nnn}
	   	N^{\frac{\alpha}{2}-s}\|\phi\|_{\dot{H}^s(\R^d)}\lesssim \|I_N \phi\|_{\dot{H}^{\frac{\alpha}{2}}(\R^d)}
	   \end{equation}
	   For convenience, we still drop the subscript $N$ from the notation and write $m(|\xi|):=m(\xi)$.
	   
	   Then, with some modifications of the proof in Proposition \ref{LWP}, the following local well-posedness still holds:
	   
	   \begin{prop}\label{LWP3}
	   	Let $s>\frac{\alpha}{2}$, and let the initial data for the modified fractional Schr\"{o}dinger equation (\ref{IFNLS}) satisfy $\||D|^{\frac{\alpha}{2}}Iu_0\|_{L^{2}(\R^d)}\le M$. Then there exists a constant $T=T(M)>0$ such that the equation is locally well-posed on the time interval $[0,T]$, with the following bound:
	   	\begin{equation*}
	   		\||D|^{\frac{\alpha}{2}}Iu\|_{X^{0,\frac{1}{2}^{+}}([0,T]\times\R^d)}\lesssim \||D|^{\frac{\alpha}{2}}Iu_0\|_{L^{2}(\R^d)}.
	   	\end{equation*}
	   \end{prop}
	   
	   \vspace{5pt}
	   Now, if we have the following ``almost conservation law'':
	   \begin{equation*}
	   	E(Iu)(t)-E(Iu)(0)=O(N^{-\gamma}), \quad \forall t\in [0,T],
	   \end{equation*} 
	   with $T>0$ as in Proposition \ref{LWP3}. Then we can iterate it $N^{\gamma}$ times, which leads to 
	   \begin{equation*}
	   	E(Iu(T_0))\lesssim 1, \quad T_0\sim N^{\gamma}.
	   \end{equation*}
	   
	   Applying (\ref{nnn}) and conservation of mass, one can directly obtain
	   \begin{equation*}
	   	\|u(T_0)\|_{H^{s}(\R^d)}\lesssim N^{s-\frac{\alpha}{2}}\|I_N \phi\|_{\dot{H}^{\frac{\alpha}{2}}(\R^d)}\lesssim N^{s-\frac{\alpha}{2}}\sim T_0^{\frac{s-\frac{\alpha}{2}}{\gamma}}.
	   \end{equation*}
	   
	   Then, with some simple modifications, from Proposition \ref{d=3} and Proposition \ref{d=2} we can derive the following polynomial-in-time growth of higher-order Sobolev norms:
	   
	   \begin{thm}
	   	Let $s>\frac{\alpha}{2}$, $1<\alpha<2$, and let $u\in C([0, +\infty); H^s(\R^3))$ be the global solution to the fractional Schr\"{o}dinger equation (\ref{FNLS}). Then the following a priori bound holds:
	   	\begin{equation*}
	   		\|u(T)\|_{H^{s}(\R^3)}\lesssim (1+T)^{\frac{s-\frac{\alpha}{2}}{2\alpha-3}^+}, \quad \forall T>0.
	   	\end{equation*}
	   \end{thm}
	   
	   \begin{thm}
	   	Let $s>\frac{\alpha}{2}$, $\frac{3}{2}<\alpha<2$, and let $u\in C([0, +\infty); H^s(\R^2))$ be the global solution to the fractional Schr\"{o}dinger equation (\ref{FNLS}). Then the following a priori bound holds:
	   	\begin{equation*}
	   		\|u(T)\|_{H^{s}(\R^2)}\lesssim (1+T)^{\frac{2s-\alpha}{-\alpha^2+5\alpha-4}^+}, \quad \forall T>0.
	   	\end{equation*}
	   \end{thm}
	   \begin{rem}
	   	Besides the upside-down $I$-method, there are still many strategies to derive polynomial-in-time growth of Sobolev norms. For example, the normal form reduction introduces some appropriate symplectic transformations to cancel non-resonant parts in the Hamiltonian (see \cite{39,40}), which yields better control of Sobolev norms. One can also introduce other new quantities to replace $\|u(t)\|_{H^s}$, so that the two are equivalent when the latter is big enough, but the former has a more tractable time derivative. We refer to \cite{41,42,43} for detailed discussions.
	   \end{rem}
	   
		\newpage
		\section*{Acknowledgement}
		The author is grateful to Prof. Alex Cohen for helpful introductions and discussions.
		
		\section*{Conflict of interest statement}
		The author does not have any possible conflict of interest.
		
		\section*{Data availability statement}
		The manuscript has no associated data.
		\bigskip
		\bigskip

		\bibliographystyle{alpha}
		\bibliography{Imethod}

	\end{document}